\documentclass[11pt]{article}

\makeatletter
\renewcommand\@biblabel[1]{#1}

\usepackage[T1]{fontenc}
\usepackage[margin=3cm]{geometry}
\usepackage{amsfonts,amssymb,stmaryrd,amsthm,amsmath}
\usepackage{frcursive}
\usepackage{titling}
\usepackage{changepage}
\usepackage{lipsum}
\usepackage{adjustbox}
\usepackage{extarrows}

\usepackage{graphicx,tikz,tikz-cd, tabularx}
\usepackage{imakeidx}
\usepackage{marginnote}
\usepackage{quiver}
\usetikzlibrary{bending}
\usetikzlibrary{matrix,arrows,decorations.markings}
\usepackage{mathtools}

\usepackage{amsthm}
\usepackage{textcomp}
\usepackage{amssymb}
\usepackage{amsmath}
\usepackage{tipa}
\usepackage{graphicx}
\usepackage{enumerate}
\usepackage{dsfont}
\graphicspath{{Images/}}

\newcommand\bmat{\begin{pmatrix}}
\newcommand\emat{\end{pmatrix}}
\newcommand\longto{{\, \longrightarrow \, }}
\newcommand\ifff{{ \, \Longleftrightarrow \,}}
\newcommand\RR{\mathbb{R}}
\newcommand\CC{\mathbb{C}}
\newcommand\NN{\mathbb{N}}\newcommand\ZZ{\mathbb{Z}}
\newcommand\KK{\mathbb{K}}\newcommand\GG{\mathbb{G}}
\newcommand\FF{\mathbb{F}}
\newcommand\LL{\mathbb{L}}

\newcommand\inv{{^{-1}}}
\newcommand\Li{{\mathcal
    L}}

\newcommand\coker{{\operatorname{coker}}}

\newcommand\Spec{{\operatorname{Spec}}}

\newcommand\Divv{{\operatorname{Div}}}

\newcommand\hPhi{{\widehat\Phi}}
\newcommand\hW{{\widehat W}}
\newcommand\hnu{{\widehat\nu}}
\newcommand\hx{{\widehat x}}

\newcommand\hG{{\widehat G}}
\newcommand\hB{{\widehat B}}
\newcommand\hT{{\widehat
    T}}
    \newcommand\hU{{\widehat U}} 
     
    \newcommand{\hg}{\widehat g}
    \newcommand{\hlambda}{\widehat \lambda}

\newcommand{\wh}{\widehat}

\newcommand\SL{\operatorname{SL}}\newcommand\Sp{\operatorname{Sp}}

\newcommand\Orb{{\mathcal O}}
\newcommand\Hom{{\operatorname{Hom}}}

\newcommand\Id{{\operatorname{Id}}}

\renewcommand\sl{{\mathfrak{sl}}}
\renewcommand{\lg}{{\mathfrak g}}
\newcommand{\lh}{{\mathfrak h}}
\newcommand{\lb}{{\mathfrak b}}
\newcommand{\lu}{{\mathfrak u}}

\newcommand\Br{{\operatorname{Br}}}

\newcommand{\clu}{\mathcal{A}}
\newcommand{\uclu}{\overline{\mathcal{A}}}
\newcommand{\upp}{\mathcal{U}}

\newcommand{\val}{\mathcal{V}}
\newcommand{\cval}{\mathcal{C}\mathcal{V}}
\newcommand\supp{{\operatorname{supp}}}

\newcommand\coef{{\operatorname{Cf}}}
\newcommand{\lX}{{\mathfrak{X}}}

\newcommand{\lif}{\upharpoonleft \kern-0.35em}
\newcommand{\spc}{\kern0.2em}

\newcommand{\lifa}{\xlongleftarrow{\nu}}

\newcommand{\lifap}{\xlongleftarrow{\nu'}}

\newcommand{\ii}{\mathbf{i}}
\newcommand{\jj}{\mathbf{j}}
\newcommand{\sba}{\overline{s}_\alpha}

\newcommand{\sbb}{\overline{s}_\beta}

\newcommand{\phiam}{\varphi_{\varpi_\alpha}^-}
\newcommand{\vap}{v_{\varpi_\alpha}^+}

\newcommand{\pia}{\varpi_\alpha}
\newcommand{\pib}{\varpi_\beta}

\newcommand{\delavw}{\Delta^{\varpi_\alpha}_{v,w}}
\newcommand{\delbvw}{\Delta^{\varpi_\beta}_{v,w}}
\newcommand{\delaew}{\Delta^{\varpi_\alpha}_{e,w}}
\newcommand{\delbew}{\Delta^{\varpi_\beta}_{e,w}}
\newcommand{\delaee}{\Delta^{\varpi_\alpha}_{e,e}}
\newcommand{\delbee}{\Delta^{\varpi_\beta}_{e,e}}

\newcommand{\wkm}{w_{\leq k}^{-1}}

\usepackage[utf8]{inputenc}  

\usepackage[T1]{fontenc}

\usepackage{hyperref}

\newtheorem{prop}{Proposition}[subsection]
\newtheorem{theo}[prop]{Theorem}
\newtheorem{lemma}[prop]{Lemma}
\newtheorem{obj}[prop]{Objective}
\newtheorem{coro}[prop]{Corollary}
\newtheorem{question}[prop]{Question}

\theoremstyle{definition}
\newtheorem{definition}[prop]{Definition}
\usepackage{soul}
\newtheorem{remark}[prop]{Remark}
\newtheorem{example}[prop]{Example}

\usepackage{soul}

\newcommand\DynkinNodeSize{2mm}
\newcommand\DynkinArrowLength{3mm}
\tikzset{
  dnode/.style={
    circle,
    inner sep=0pt,
    minimum size=\DynkinNodeSize,
    fill=white,
    draw},
  middlearrow/.style={
    decoration={markings,
      mark=at position 0.6 with
      {\draw (0:0mm) -- +(+135:\DynkinArrowLength); \draw (0:0mm) -- +(-135:\DynkinArrowLength);},
    },
    postaction={decorate}
  },
  leftrightarrow/.style={
    decoration={markings,
      mark=at position 0.999 with
      {
      \draw (0:0mm) -- +(+135:\DynkinArrowLength); \draw (0:0mm) -- +(-135:\DynkinArrowLength);
      },
      mark=at position 0.001 with
      {
      \draw (0:0mm) -- +(+45:\DynkinArrowLength); \draw (0:0mm) -- +(-45:\DynkinArrowLength);
      },
    },
    postaction={decorate}
  },
  sedge/.style={
  },
  dedge/.style={
    middlearrow,
    double distance=0.5mm,
  },
  tedge/.style={
    middlearrow,
    double distance=1.0mm+\pgflinewidth,
    postaction={draw}, 
  },
  infedge/.style={
    leftrightarrow,
    double distance=0.5mm,
  }
}

\newcommand\revddots{\mathinner{\mkern1mu\raise\p@\vbox{\kern7\p@\hbox{.}}\mkern2mu\raise4\p@\hbox{.}\mkern2mu\raise7\p@\hbox{.}\mkern1mu}}
\makeatletter
\def\revddots{\mathinner{\mkern1mu\raise\p@\vbox{\kern7\p@\hbox{.}}\mkern2mu\raise4\p@\hbox{.}\mkern2mu\raise7\p@\hbox{.}\mkern1mu}}
\makeatother

\title{The minimal monomial lifting of cluster algebras I: branching problems}
\author{Luca Francone \\
francone at math.univ-lyon1.fr}
\date{}

\begin{document}

\maketitle
\begin{abstract}

Let $\hG \subseteq G$ be complex reductive algebraic groups. The branching problem that aims to study $G$-modules as $\hG$-modules is encoded by a collection of branching multiplicities parameterised by pairs of dominant weights. The branching algebra $\Br(G,\hG)$ is a graded algebra whose dimension of homogeneous components are precisely the branching multiplicities. Here, we endow $\Br(G, \hG)$ with the structure of a graded upper cluster algebra, for some pair of groups. Our result holds if $\hG$ is a Levi subgroup of $G$ or in the tensor product case, that is when $\hG$ is the diagonal in $G= \hG \times \hG$, assuming that $G$ is semisimple and simply connected. This sharpens J.Fei's result who got the same statement for $\hG=T$ a maximal torus of $G$ and for $G \subseteq G \times G$, assuming $G$ simple, simply laced and simply connected. To prove our result we develop a new geometric and compbinatorial technique called minimal monomial lifting.

    Let $Y$ be a complex scheme with cluster structure, $T$ be a complex torus and $\lX$ be a suitable partial compactification of $T \times Y$. The minimal monomial lifting produces a canonically graded upper cluster algebra $\uclu$ inside $\Orb_\lX(\lX)$ which is, in a precise sense, the best candidate to give a cluster structure on $\lX$ compatible with the one on $Y$. We develop some geometric criteria to prove the equality between $\uclu$ and $\Orb_\lX(\lX)$, which doesn't always hold and has some remarkable consequences. This technique is very flexible and will be used elsewhere to endow other classical algebras with the structure of a graded upper cluster algebra.

\end{abstract}
\tableofcontents

    \section{Introduction }

This is the first step in a project which aims to study branching problems in representation theory, through the use of cluster algebras.

\subsubsection*{The branching problem} Let $\hG$ be a complex, connected, reductive algebraic subgroup of the connected reductive group $G$. The branching problem in representation theory asks to understand how, irreducible representations of $G$, decompose under the natural $\hG$-action.

Fix maximal tori $\hT\subseteq T$ and Borel subgroups $B\supseteq T$ and
$\hB\supseteq \hT$ of $G$ and $\hG$ respectively.
Let $X(T)$ denote the group of characters of $T$ and let $X(T)^+$
denote the set of dominant characters.
For $\lambda \in X(T)^+$, $V(\lambda)$ denotes the
irreducible representation of highest weight $\lambda$.
Similarly, we use the notation $X(\widehat T)$,  $X(\widehat T)^+$,  $V(\widehat\lambda)$
relatively to $\widehat G$. 
For any $\lambda \in X(T)^+$, there is a natural $\hG$-equivariant isomorphism:
$$
\begin{array}{rcl}
  \displaystyle \bigoplus_{\hlambda\in X(\hT)^+} \Hom(V(\hlambda),V(\lambda))^\hG\otimes
  V(\hlambda)&\longto&V(\lambda)\\
f\otimes v&\longmapsto&f(v).
\end{array}
$$
Hence the \textit{multiplicity} of $V(\hlambda)$ in $V(\lambda)$ is $c_{\lambda}^{\hlambda}= \dim \Hom(V(\hlambda), V(\lambda))^G$.
A result one would like to achieve, in solving the branching problem is the following.

\begin{obj}
\label{obj: branching}
For the pair $(G, \hG)$, construct a positive combinatorial model for the multiplicities.
    
\end{obj}

Examples of such models are given by the Littlewood-Ricardson rule, Gelfand-Tseltlin's patterns, Littelmann's paths \cite{littelmann1995paths}, Sundaram's dominos \cite{Sunda} and the celebrated Knutson-Tao's hive model \cite{KT:saturation}. The last one is the essential tool for the proof of the saturation conjecture by Buch \cite{buch2000saturation}, after \cite{KT:saturation}. 

Though for some specific pairs $(G, \hG)$ such a model exists, no general technique for constructing those models, uniformly for some families of pairs, is known. To the author best knowledge, the only remarkable exceptions for which a positive model for multiplicities can be constructed in families, are the following.

\begin{enumerate}
    \item The Levi case: $\hG$ is a Levi subgroup of $G$.
    \item The tensor product case: $\hG$ is diagonally embedded in $G=\hG \times \hG$, for which the branching problem amounts to decompose the tensor product of irreducible representations of $\hG$.
\end{enumerate}

 A breakthrough has been realised by Berenstein and Zelevisnky \cite{berenstein1999tensor}. They constructed, under the assumption that $G$ is semisimple and simply connected, (many) polyhedral models for multiplicities, for any branching problem belonging to the previous list. Their proof is actually quite involved. It uses deep results and objects, such as: the canonical basis and the Lusztig's parametrizations \cite{lusztig2010introduction}, its relation to total positivity and double Bruhat cells \cite{lusztig1994total} \cite{fomin1999double} and tropicalisation.
 We refer to \cite{zelevinsky2001littlewood} for a beautiful survey on the proof. One of the reasons why it is difficult to generalise Berenstein-Zelevinsky approach to other pairs $\hG \subseteq G$ is because the previous objects are very specific of the Levi and the tensor product case.
 One key ingredient to link all these objects are the detereminantal identientities \cite{fomin1999double}[Theorem 1.16,1.17], which also stand at the very base of all the known cluster algebra structures on the ring of varieties related to $G$. Among such varieties: double Bruhat cells \cite{berenstein2005cluster3} and certain $T$-stable subgroups of the unipotent radical of $B$ \cite{geiss2011kac}, \cite{goodearl2021integral}. 

A more evident link between cluster algebras and branching problems has been recently realised by Magee \cite{magee2015fock} \cite{magee2020littlewood}, building on the work of Gross, Hacking, Keel, Kontsevich \cite{gross2018canonical} and Fock, Goncharov \cite{fock2006moduli} (see also \cite{gross2018canonical}[Corollary 0.20, 0.21]). Magee constructs (many) polyhedral models for the weight space decomposition (which corresponds to branching to a maximal torus) and for the tensor product decomposition of $\SL_n$, using cluster algebras. Moreover, one of Magee's models is unimodular to the Knutson-Tao's hive model. Fei obtained the same results \cite{fei2016tensor}[Theorem 8.1, 9.2] for any simple, simply laced and simply connected algebraic group. 

\bigskip

For a pair $\hG \subseteq G$, all  the information for the associated branching problem is contained in a $X(T) \times X(\hT)$ graded algebra $\Br(G,\hG)$ called \textit{branching algebra}. Indeed, we have natural isomorphisms 

$$
\begin{array}{ll}
\Br(G,\hG)_{\lambda,\hlambda} \simeq \Hom(V(\hlambda),V(\lambda))^\hG & \text{if} \, (\lambda, \hlambda) \in X(T)^+ \times X(\hT)^+\\
\Br(G,\hG)_{\lambda,\hlambda}  = 0 & \text{otherwise},
\end{array}$$
where, for $(\lambda, \hlambda) \in X(T) \times X(\hT)$, $\Br(G,\hG)_{\lambda, \hlambda}$ denotes the homogeneous component of degree $(\lambda,\hlambda)$ of $\Br(G,\hG)$.
A fundamental step for the results of Magee, Gross, Hacking Keel, Kontsevich and Fei previously discussed is the identification of $\Br(G,\hG)$, for the corresponding groups $\hG \subseteq G$, with a graded upper cluster algebra with non-invertible frozen variables. This allows the authors to use cluster theory to construct homogeneous bases of the branching algebra. Combinatorial analysis of such bases ultimately leads to the birth of combinatorial models for multiplicities.
The main result of this paper is the following theorem.

\begin{theo}
    \label{thm:main intro}
 In the following two cases:
    \begin{enumerate}
        \item $G$ is simple, simply connected and $\hG$  is a Levi subgroup of $G$.
        \item $\hG$ is semisimple, simply connected and diagonally embedded in $G= \hG \times \hG$.
    \end{enumerate}
    The branching algebra $\Br(G,\hG)$ is a graded upper cluster algebra, of geometric type, with non-invertible frozen variables.
\end{theo}

In the setting of Theorem \ref{thm:main intro}, our present work doesn't yield combinatorial models for multiplicities or homogeneous bases of the upper cluster algebra $\Br(G,\hG)$. Obviously, this can be interpreted as a weakness regarding Objective \ref{obj: branching} and  other authors' results previously mentioned.

On the contrary, we believe that this is due to great simplifications comparing to other proves and constructions, especially Fei's ones. Simultaneously, we reach a great generalisation of the hypothesis in which Theorem \ref{thm:main intro} holds, matching the same generality of the work of Berenstein and Zelevinsky on the $\ii$-trails models.
Moreover, Theorem \ref{thm:main intro} allows to rephrase Objective \ref{obj: branching} in the setting of cluster algebras, which is a much more suitable context regarding this objective.
Notably, the results of \cite{gross2018canonical} and \cite{qin2022bases}, give some explicit combinatorial criteria for constructing homogeneous bases of upper cluster algebras and corresponding parametrizations. These criteria should apply to the upper cluster algebras of Theorem \ref{thm:main intro}. The combinatorial models described in \cite{magee2020littlewood} arise precisely in this way. 
 
 Thus, Theorem \ref{thm:main intro} ultimately
sheds light on a precise path that could lead to the realisation of Objective \ref{obj: branching}. Simultaneously, we hope that the previously discussed Berenstein and Zelevinsky's result, which is currently maximum in generality, can be surpassed.
Indeed, the minimal monomial lifting, which is the technique we develop to identify the branching algebra with a graded upper cluster algebra is flexible and applies to several branching problems. Thus, we hope to generalise Theorem \ref{thm:main intro} to other pairs $\hG \subseteq G$.

\subsubsection*{Cluster algebras in brief}

Starting from a combinatorial data called seed, usually denoted by $t$, and an iterative procedure called mutation, one can define the \textit{cluster algebra} $\clu(t)$ and the \textit{upper cluster algebra} $\uclu(t)$. In general one has $\clu(t) \subseteq \uclu(t)$. The cluster algebra is more combinatorial in nature, while the upper cluster algebra is more geometric. 
Thanks to a huge effort due to many authors, there's nowadays a quite well developed structural theory of cluster algebras.

\bigskip  

To explain some features of this theory, we recall that part of the defining data of a seed $t$ is the vertex set $I$, which has a partition $I=I_{uf} \sqcup I_{f}$ into unfrozen (or mutable) and frozen vertices. Each seed also posses a distinguished set of cluster variables $x=(x_i)_{i \in I}$, indexed by the vertex set, which are elements of the cluster algebra $\clu(t).$ 
Defining the degree of the cluster variables of the seed $t$, under some constraints arising from the mutation process, gives a global graduation on both $\clu(t)$ and $\uclu(t)$ to which we refer as a \textit{cluster graduation}. In this case, the collection of the degrees (which have value in an abelian group) of the cluster variables of $t$ is called \textit{degree configuration}.
There's a well developed theory for constructing special bases, called \textit{good bases}, of $\uclu(t)$ \cite{dupont2011generic}, \cite{qintriangular}, \cite{gross2018canonical}, \cite{qin2022bases}. If we allow some frozen variables not to be invertible in $\uclu(t)$, some results are still conjectural. Nevertheless, under some technical assumption on the seed $t$, good bases exist and have many parametrizations by rational points into polyhedral cones which are related by invertible piece-wise linear maps.   

\bigskip

We say that a scheme $\lX$ has cluster structure if $\Orb_\lX(\lX)= \uclu(t)$ for a certain seed $t$.

\subsubsection*{Monomial liftings and minimal monomial lifting}

In this paper we introduce the notion of \textit{lifting configuration} on a seed $t$. Let $I$ be the vertex set of $t$ and $D$ be a finite set. By little abuse, a lifting configuration $\nu$, on $t$, is a matrix $\nu \in \ZZ^{D \times I}$. Given a lifting configuration $\nu$, on $t$, we introduce a seed  called \textit{monomial lifting} of $t$ associated to $\nu$, which is denoted by $\lif t$. The vertex set $\lif I$ of $\lif t$ consists of $I$ (the vertex set of $t$) and $D$, the latter being a subset of the frozen vertex set. 
In particular, $t$ and $\lif t$ have the same set of mutable variables. Moreover, for every $d \in D$, the associated frozen cluster variable of $\lif t$ is denoted by $\lif x_d$. 

We prove that the mutation of a monomial lifting is again a monomial lifting. Formally, for any mutable vertex $k \in I_{uf}$, we introduce the notion of \textit{mutation of lifting configurations} at $k$, and prove the following properties.

\begin{prop}
\label{prop:mon lift intro}
Let $\nu \in \ZZ^{D \times I}$ be a lifting configuration on $t$. 
    \begin{enumerate}
        \item Lifting a seed commutes with mutation. 
        \item For every seed $t'$ mutation equivalent to $t$, through the mutation process, we can associate a well defined lifting configuration $\nu'$ on $t'.$
        \item $\clu(\lif t)=\clu(t)[\lif x_d ^{\pm 1}]_d$.
        \item If $t$ is of maximal rank,  $\uclu(\lif t)=\uclu(t)[\lif x_d ^{\pm 1}]_d$.
    \end{enumerate}
\end{prop}

Moreover, if the seed $t$ is graded by a degree configuration $\sigma$, we produce a degree configuration $\lif \sigma$, on $\lif t$, called the \textit{lifting} of the degree configuration $\sigma$. Since every seed $t$ is trivially graded, the cluster algebra $\uclu(\lif t)$ is canonically $\ZZ^D$-graded by the lifting of the trivial degree configuration on $t$, which is denoted by $\lif 0$.

\bigskip 

Suppose from now on that $t$ is of maximal rank and that $\lif t$ is the monomial lifting with respect to a lifting configuration $\nu$. If we allow the frozen variables $\lif x_d$ for $d \in D$ to vanish, that is if we consider an upper cluster algebra with non-invertible frozen variables indexed by $D$, we get the algebra $\uclu(\lif t^D)$. This is the object we are most interested in. Suppose for a moment that $\uclu(\lif t^D)$ is of finite type, then $\uclu(t)$ also is. Let $$\lX= \Spec(\uclu(\lif t^D)) \quad \text{and} \quad Y= \Spec(\uclu(t)).$$
By Proposition \ref{prop:mon lift intro}, we have an open embedding $\phi_\nu : \GG_m^D \times Y \longto \lX$ whose image is the non-vanishing locus of the variables $\lif x_d$ for $d \in D$, which is canonically identified with $\Spec(\uclu(\lif t))$. It's easy to verify the following statement.

\begin{prop}
    \label{prop:ausiliary intro}
    The triple $(\lX, \,  \phi_\nu, \, \lif x= (\lif x_d)_{d \in D})$ satisfies the following properties: 

\begin{enumerate}
     \item The scheme $\lX$ is noetherian, normal and integral.
    \item For each $d \in D$, the zero locus $V(\lif x_d)$ of $\lif x_d$ in $\lX$ is irreducible. If $\cval_d: \CC(\lX) \longto \ZZ \cup \{\infty\}$ is the valuation associated to the divisor $V(\lif x_d)$, then $$\cval_{d_1}(\lif x_{d_2})= \delta_{d_1,d_2}.$$
    \item The map $\phi_\nu: \GG_m^D \times Y \longto \lX \setminus \cup_d V(\lif x_d)$ is an isomorphism such that $\phi_\nu^*(\lif x_d)= (x_d \otimes 1)$, where the collection of $x_d$ is a base of the character group $X\bigl( \GG_m^D \bigr)$ of the torus $ \GG_m^D$. Moreover, $Y$ is an irreducible $\CC$-scheme.
\end{enumerate}
\end{prop}

For $d \in D$, the valuation $\cval_d$ is called \textit{cluster valuation}. We can identify $\ZZ^D$ with $X\bigl( \GG_m^D \bigr)$ by means of the base $(x_d)_{d \in D}$. Then, the graduation on $\uclu(\lif t^D)$ induced by the degree configuration $\lif 0$ on $\lif t$, corresponds to an action of $\GG_m^D$ on $\lX$ with respect to which $\phi_\nu$ is equivariant. The geometry of $\lX$ allows to trace the lifting configuration $\nu$, indeed: 
$$\cval_d(1 \otimes x_i)= - \nu_{d,i} \quad \text{for any} \, d \in D, i \in I$$
where, with little abuse $1 \otimes x_i$ denotes $\phi_{\nu, *}(1 \otimes x_i) \in \CC(\lX)$. Also, $\nu_{\bullet,i}=(\nu_{d,i})_{d \in D}$ is the degree of the cluster variable $\lif x_i$ with respect to the $\GG_m^D$ action on $\lX.$

\bigskip

Based on the previous discussion, we consider a triple $(\lX, \phi, X)$ where $\lX$ is a non-necessarily affine complex scheme, $\phi: \GG_m^D \times Y \longto \lX$ is an open embedding (here $Y$ is an irreducible complex scheme) and $X=(X_d)_{d \in D} \in \Orb_\lX(\lX)^D$ is a collection of regular functions. 
Consider analogues of the three statements of Proposition \ref{prop:ausiliary intro}, relative to the triple $(\lX, \phi, X)$. For $d \in D$, we denote by $\val_d$ the valuation corresponding to the divisor $V(X_d)$. 
Up to some details, we say that the triple $(\lX, \phi, X)$ is \textit{suitable for lifting} is these three statements hold.
For such a triple, we identify $\GG_m^D \times Y$ with an open subset of $\lX$ via $\phi$. We have that, a non-zero function $f \in \Orb_Y(Y)$, can be "homogenised" to a global regular function on $\lX$ by multiplying for a Laurent monomial in the $X_d$ whose exponent is given by the integers $-\val_d(1 \otimes f).$

If $\Orb_Y(Y)= \uclu(t)$, the matrix $\nu \in \ZZ^{D \times I}$ defined by $\nu_{d,i}=-\cval_d(1 \otimes x_i)$ is called the \textit{minimal lifting matrix} of the seed $t$ with respect to $(\lX,\phi, X).$ The monomial lifting $\lif t^D$ (with non-invertible frozen variables indexed by $D$) of $t$ with respect to $\nu$ is the \textit{minimal monomial lifting} of $t$ with respect to $(\lX, \phi, X)$. The seed $\lif t^D$ is the best possible candidate to give $\lX$ a cluster structure compatible with the one on $Y$, in the sense of Theorem \ref{unicity minimal monomial lifting}. In particular we have the following theorem.

\begin{theo}
    We have a natural inclusion $\uclu(\lif t^D) \subseteq \Orb_\lX(\lX)$. This inclusion is an isomorphism over $\GG_m^D \times Y$, meaning that $  \uclu(\lif t^D)_{\prod_d \lif \spc x_d}= \uclu(\lif t)= \Orb_\lX( \GG_m^D \times Y).$
    Moreover, if a seed $\widetilde t$ is compatible with $t$ in the sense of Theorem \ref{unicity minimal monomial lifting}, and $\uclu(\widetilde t)=\Orb_\lX(\lX)$, then $\widetilde t= \lif t^D$.
\end{theo}

The inclusion between $\uclu(\lif t^D)$ and $\Orb_\lX(\lX)$ may be strict. Ultimately, the difference between $\uclu(\lif t^D)$ and $\Orb_\lX(\lX)$, depends on the behaviour of  $\lif t^D$ along the divisors of $\lX$ in the complement of $\GG_m^D \times Y$. Thus, the question if $\uclu(\lif t^D)$ equals $\Orb_\lX(\lX)$ can be studied geometrically. For example, we have the following proposition.

\begin{prop}
    \label{prop:equality intro} The equality
      $\uclu(\lif t^D)= \Orb_\lX(\lX)$ holds
    if and only if, for any $d \in D$, $\cval_d = \val_d$ over $\Orb_\lX(\lX).$
\end{prop}

If there exist an action of $\GG_m^D$ on $\lX$ which makes the map $\phi$ $\GG_m^D$-equivariant, we say that the triple $(\lX, \phi, X)$ is homogeneously suitable for lifting. In that case, the inclusion between $\uclu(\lif t^D)$ and $\Orb_\lX(\lX)$ is of graded algebras, where $\uclu(\lif t^D)$ is graded by the trivial degree configuration $\lif 0$ and $\Orb_\lX(\lX)$ has a natural grading induced by the $\GG_m^D$-action. In this situation, the minimal monomial lifting is the best possible candidate to give $\lX$ a cluster structure simultaneously compatible with the one on $Y$ and with the $\GG_m^D$-action. 

\begin{theo}
    \label{thm:unicity torus intro}
    Let $\widetilde t$ be a graded seed compatible with $t$ in the sense of Theorem \ref{thm:unicity min mon lifting torus}. If $\uclu(\widetilde t)= \Orb_\lX(\lX)$ as graded algebras, then $\widetilde t= \lif t$ as graded seeds.
\end{theo}

\subsubsection*{Monomial lifting for the branching scheme}

Given a pair $\hG \subseteq G$ of reductive groups, we consider the \textit{branching scheme} $\lX(G,\hG)= \Spec \bigl( \Br(G,\hG) \bigr)$ with its $T \times \hT$ action induced by the previously discussed graduation. We prove that, whenever $G$ is semisimple and simply connected, the scheme $\lX(G, \hG)$ has a natural structure of homogeneously suitable for lifting scheme. Moreover, the torus $\GG_m^D$ can be naturally identified with $T$. 

Let $U$ and $\hU$ be the unipotent radicals of $B$ and $\hB$ respectively. We denote by $C[U]^\hU$ the algebra of $\hU$ right-invariant functions on $U$. The scheme $Y$ of the suitable for lifting structure of $\lX(G,\hG)$ is precisely $Y= \Spec \bigl( \CC[U]^\hU \bigr)$. Now, the interesting question is whether $Y$ has a meaningful cluster structure. However, this is not clear in general. 

Nevertheless, it is often true that $Y$ can be naturally identified with $U(w)= U \cap w\inv U^- w$, for some $w$ in the Weyl group of $G$. The varieties $U(w)$ are known to have cluster structure by \cite{geiss2011kac} and \cite{goodearl2021integral}. For these cluster structures, we prove that the $X(T)$-graduation induced by the conjugation action of $T$  on $U(w)$ is a cluster graduation. Then, applying Proposition \ref{prop:equality intro}, we deduce the following more precise version Theorem \ref{thm:main intro}.

\begin{theo}
    \label{thm:intro branching} In the following two cases:
    \begin{enumerate}
        \item $G$ is simple, simply connected and $\hG$  is a Levi subgroup of $G$.
        \item $\hG$ is semisimple, simply connected and diagonally embedded in $G= \hG \times \hG$.
    \end{enumerate}
    If \,$\lX=\lX(G, \hG)$, there exists a seed $t$ constructed in \cite{goodearl2021integral}, \cite{geiss2011kac}, which is $X(\hT)$-graded by a degree configuration $\sigma$, such that $\uclu(\lif t^D)= \Orb_\lX(\lX)$ as $X(T) \times X(\hT)$ graded algebras. Here,  $\uclu(\lif t^D)$ is graded by the degree configuration $\lif \sigma$.
\end{theo}

Actually, in Theorem \ref{thm:intro branching} the graduation we consider on $\Orb_\lX(\lX)$ is a twist of the previously discussed one, but we prefer to leave this merely technical, and not relevant point out of this introduction.

 \subsubsection*{Some further remarks}

 \begin{enumerate}
\item\textbf{Schemes vs affine varieties.} In this paper we apply the minimal monomial lifting technique to schemes which are actually affine varieties. We prefer to carry out the discussion on the minimal monomial lifting in terms of schemes in view of future applications. We also think that the language of schemes is more adapted to tackle a geometric study of cluster algebras and its applications for many reasons.
For example, for a given seed $t$, the algebra $\uclu(t)$ may not be noetherian \cite{gross2013birational}, \cite{speyer2013infinitely}. Still, in the language of Fock and Goncharov \cite{fock2009cluster}, it may be interpreted as the ring of regular functions on a (generally non-affine) locally of finite type smooth variety which is called $\mathcal{A}$-cluster variety. The $\mathcal{A}$-cluster variety often has a big open subset which is of finite type. Of course, the $\mathcal{A}$-cluster variety is a more useful geometric model for $\uclu(t)$, than $\Spec \bigl (\uclu(t) \bigr)$, if $\uclu(t)$ is not of finite type. Moreover, in the present paper we only discuss $\mathcal{A}$-cluster algebras. Nevertheless, it is natural to carry out similar constructions for $\mathcal{X}$-cluster algebras. Then, one has to deal with the fact that the $\mathcal{X}$-cluster variety may not be separated. \cite{gross2013birational}.

\bigskip

\item\textbf{Pole filtrations and biperfect bases.}
Let $\phi: \GG_m^D \times Y \longto \lX$ be the open embedding of an homogeneously suitable for lifting structure on $\lX$ (that is a tripe $(\lX, \phi, X)$ which is homogeneously suitable for lifting). 
Then, we have a restriction map $s: \Orb_\lX(\lX) \longto \Orb_Y(Y)$ defined as the pullback along the map $Y \longto \lX$ sending $y$ to $\phi(e,y)$. The map $s$ injects the spaces of semi-invariants functions on $\lX$ in $\Orb_Y(Y)$, giving birth to a $X\bigl(\GG_m^D \bigr)$-filtration on $\Orb_Y(Y)$ that we call \textit{pole filtration} and which reflects, geometrically, the properties of the natural $X\bigl(\GG_m^D \bigr)$-graduation on $\Orb_\lX(\lX)$.

This situation generalises the classical embedding of the irreducible $G$-representations into the ring of functions on $U$, for a semisimple simply connected algebraic group $G$ with a maximal unipotent subgroup $U$. See Section \ref{sec:base aff space} for more details.

The fact that, for a pair $\hG \subseteq G$ of reductive groups with $G$ semisimple and simply connected, the branching scheme $\lX(G,\hG)$ has an homogeneously suitable for lifting structure allows to deduce the following proposition.

\begin{prop}
\label{prop:perfect base intro}
    For any $(\lambda, \hlambda) \in X(T)^+ \times X(\hT)^+$, we have an isomorphism $$\Hom(V(\lambda), V(\hlambda))^\hG \longto \CC[U]^\hU_{\lambda, \hlambda -\rho(\lambda)}$$
where
\begin{equation*}
\begin{array}{r l l }
\CC[U]^\hU_{\lambda , \hlambda - \rho(\lambda)} = \{ f \in \CC[U]^\hU \, : &  f(\wh h^\inv u \wh h)= (\hlambda-\lambda)( \wh h) f(u) &  \text{for} \quad \wh h \in \hT, u \in U \ , \, \text{and} \\
& \val_\alpha (1 \otimes f) \geq -\langle \lambda, \alpha^\vee \rangle & \text{for} \quad  \alpha \in D\}.
\end{array}
\end{equation*}
\end{prop}
 If $G= \hG \times \hG$, we can identify $\CC[\hU \times \hU]^\hU$ with $\CC[\hU]$, and a base of $\CC[\hU]$ which is adapted to the subspaces appearing in the previous proposition is known as \textit{biperfect basis}. Indeed, Proposition \ref{prop:perfect base intro} can be interpreted as a generalisation of \cite{zelevinsky}[Proposition 1.4], which is the starting point for the interest in biperfect bases. These kind of bases include the classical limit of the dual canonical basis and have arguably been one of the most used items to study tensor product decomposition. 
See \cite{kamnitzer2022perfect} for a beautiful survey on biperfect bases.

 \item\textbf{Minimal monomial lifting and \cite{fei2016tensor}.} In Section \ref{sec: prod, tensor prod} we prove that Fei's cluster structures \cite{fei2016tensor} are obtained through minimal monomial lifting. It's probable that Fei's structures are special cases of the ones constructed in the present text, but I'm not capable of carrying out this comparison in detail. Finally, Example \ref{ex:G_2 tensor prod} and Theorem \ref{thm:intro branching} partially answer to \cite{fei2016tensor}[Conjecture 8.2] in the $G_2$-case. 

\item\textbf{Other branching problems.} There are several pairs $\hG \subseteq G$ of reductive groups, with $G$ semisimple and simply connected, for which we can identify a $X(\hT)$-graded cluster structure on $\Spec\bigl( \CC[U]^\hU \bigr)$. However, it is not clear to the author how to identify these pairs. From this point of view, Question \ref{ques: U=U(z) x hU?} is of interest. When we identify a cluster structure on $\Spec\bigl( \CC[U]^\hU \bigr)$, it may happen that the upper cluster cluster algebra of the associated minimal monomial lifting is strictly contained in $\Br(G,\hG)$. This is the case for $\Sp_{2n} \subseteq \SL_{2n}$. Nevertheless, the existence of a graded upper cluster algebra inside $\Br(G,\hG)$ already has some interesting applications. Some of them are briefly discussed in Section \ref{sec:on question }.
  \end{enumerate}
\subsubsection*{Outline of the paper}

Section \ref{preliminaries} contains the preliminaries on cluster algebras. We formalise two notions already existing in the literature: the difference between \textit{highly-frozen} and  \textit{semi-frozen} vertices and define the cluster valuation $\cval_i$ associated to a frozen vertex $i$. The cluster variable associated to a highly-frozen vertex is not invertible, while the one associated to a semi-frozen one is. Then, we recall the notion of graded seed and degree configuration.

In Section \ref{monomial lifting section} we define and study monomial liftings. Proposition \ref{prop:ausiliary intro} is proved. In Section \ref{pole filtration section} we define \textit{pole filtrations} and \textit{lifting graduations.} This section is devoted to understand, geometrically, the canonical $\ZZ^D$-graduation on $\uclu(\lif t)$ and the kind of graduations we hope to understand using the minimal monomial lifting technique.

 Section \ref{minimal monomial lifting section} is essential for the paper. We define suitable and homogeneously suitable for lifting schemes and develop the minimal monomial lifting. We prove some unicity results on the minimal monomial lifting, namely Theorem \ref{unicity minimal monomial lifting}, \ref{thm:unicity min mon lifting torus} and discuss geometric criteria to study whether $\uclu(\lif t^D)$ equals $\Orb_\lX(\lX)$. 

Section \ref{Preliminaries alg groups section} is devoted to some preliminaries on algebraic groups. We study in detail some technical properties of generalised minors, that are crucial for the proof of Theorem \ref{thm:intro branching}. The reader can skip this part at first.
In Section \ref{application to G section}, we consider an example of minimal monomial lifting where the inclusion $\uclu(\lif t^D) \subseteq \Orb_\lX(\lX)$ is strict. We report this toy example since we consider it instructive.

In Section \ref{monomial lift branchig section}, we discuss the general setting for applying the minimal monomial lifting to branching problems and derive Proposition \ref{prop:perfect base intro}.

Finally, in Section \ref{study of equality branching section} we prove Theorem \ref{thm:intro branching} and study some relations with \cite{fei2016tensor}.

\bigskip

\textbf{Acknowledgements}.
I would like to thank Bernard Leclerc for the invitation to Caen and the fruitful discussions and insights on cluster algebras, and Nicolas Ressayre, for introducing me to the branching problem, for the support and the many discussions that made this work possible. 

    \section{Preliminaries on cluster algebras}
    \label{preliminaries}
\subsection{Cluster algebras}
\label{cluster algebras}

We give a definition of cluster algebras and upper cluster algebras of geometric type in the spirit of \cite{geiss2013factorial}. The only differences with the standard setting are encoded in the notions of \textit{semi-frozen} (invertible) and \textit{highly-frozen} (not invertible) variables. The type of frozen variables affects the definition of the coefficient ring of the cluster algebra. This terminology provides a natural framework for studying functions over some partial compactifications of cluster varieties. Note that our cluster algebras are very special cases of the ones defined in \cite{bucher2019upper} and of \textit{generalized cluster algebras} \cite{gekhtman2018drinfeld}.

   \begin{definition} Let $\KK$ be a field extension of $\CC$. A \textit{seed} $t$ of $\KK$ is a collection $$(I_{uf}, I_{sf}, I_{hf}, B, x ) $$ consisting of:

   \begin{itemize}
       \item[--] Three disjoint finite sets $I_{uf}, I_{sf}, I_{hf}$. An element of $I:= I_{uf} \sqcup I_{sf} \sqcup I_{hf}$ (resp. $I_f:= I_{hf} \sqcup I_{sf }$) is called \textit{vertex} (resp. \textit{frozen} vertex).  A vertex is respectively called \textit{unfrozen} (or \textit{mutable}), \textit{semi-frozen}, \textit{highly frozen} if it belongs to $I_{uf}, I_{sf}, I_{hf}.$ 
       \item[--] The \textit{extended exchange matrix} $B \in \ZZ^{I \times I_{uf}}$, which is an integer matrix, such that it's \textit{principal part} $B^\circ :=B_{| I_{uf} \times I_{uf}}$ is \textit{skew-symmetrizable}. That is: there exists $d_i \in \NN_{>0} $, for $i \in I_{uf}$, such that for any $i,j \in I_{uf}$, $d_ib_{i,j} = - d_jb_{j,i}.$
       \item[--] An \textit{extended cluster} $x \in (\KK^*)^I$,  consisting of a transcendence basis of $\KK$ over $\CC$, whose elements are called \textit{cluster variables}. A variable is said to be \textit{unfrozen} (or \textit{mutable}) (resp. \textit{semi-frozen}, resp. \textit{highly-frozen}, resp. \textit{frozen}) if its corresponding vertex is unfrozen (resp. semi-frozen, resp. highly frozen, resp. frozen).
   \end{itemize}

   \end{definition}

If a seed is denoted by $t$, we implicitly assume that its defining data are denoted as in the previous definition. If the seed is called $t^\bullet$ (resp. $t_\bullet$), where $\bullet$ is any superscript (resp. subscript), then we add a $\bullet$ superscript (resp. subscript) to all the notation. For example, we write $t'=(I'_{uf}, I'_{sf}, I'_{hf}, B', x')$ or $t_\ii=(I_{\ii, uf} , \dots , x_\ii)$. If needed, we add the dependence on $t$ writing $B(t)$, $x(t)$ and so on. 

\bigskip
\textbf{Graphical notation for seeds.} Given a seed $t$, the only data required to identify the isomorphism class of the associated cluster and upper cluster algebra is the vertex set and the extended exchange matrix. This is encoded graphically in a \textit{valued quiver} $Q$ (or $Q(t)$ if the dependence on $t$ is needed) as follows.

\begin{itemize}
    \item [-] The vertex set of $Q$ is $I$. For $i \in I$, the corresponding vertex of $Q$ is pictured by a symbol $\bigcirc$ (resp. $\square$, resp. $\blacksquare$) if it is unfrozen (resp. semi-frozen, resp. highly frozen) and labelled by $i$.
\end{itemize}
Given two vertices $i$ and $j$ of $Q$, we use the convention that $b_{i,j}=0$ if $i,j$ are both frozen and that $b_{i,j}=-b_{j,i}$ if exactly one between $i$ and $j$ is frozen. Note that, in the last case, exactly one between $b_{i,j}$ and $b_{j,i}$ is defined as a coefficient of the generalised exchange matrix, according to which vertex is unfrozen. Then
\begin{itemize}
    \item[-]  There is an arrow between $i$ and $j$, pointing towards $j$, if and only if $b_{i,j} > 0$. In this case, the arrow is labelled by: "$b_{i,j}, - b_{j,i}$". Moreover, if $b_{i,j} = - b_{j,i}$, for low values of $b_{i,j}$, we may write $b_{i,j}$ distinct arrows from $i$ to $j$ instead of a labelled arrow.
\end{itemize}
\begin{example}
    \label{ex label seed}
    Let $t=( \{1,2\}, \{3\}, \{4\}, B, x)$ be a seed whose generalised exchange matrix is $$ B= \bmat 0 & 3 \\
    -1 & 0 \\
    0 & -2 \\
    0 & 1
    \emat
    $$
    where the rows and columns of $B$ are labelled in the obvious way. The following two pictures are both a graphical representation of the seed $t$

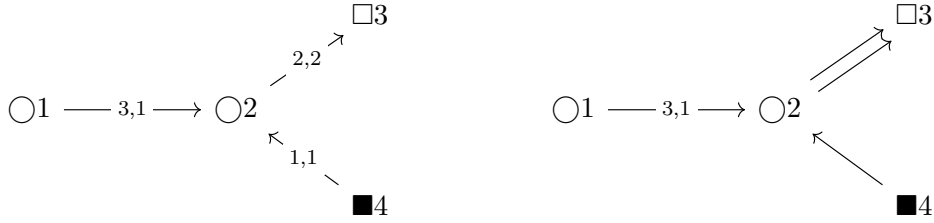
    \[\begin{tikzcd}
	&&& {\square 3} &&&&& {\square 3} \\
	{\bigcirc 1} && {\bigcirc 2} &&& {\bigcirc 1} && {\bigcirc 2} \\
	&&& {\blacksquare 4} &&&&& {\blacksquare 4}
	\arrow["{2,2}"{description}, from=2-3, to=1-4]
	\arrow["{3,1}"{description}, from=2-1, to=2-3]
	\arrow["{3,1}"{description}, from=2-6, to=2-8]
	\arrow[shift left, from=2-8, to=1-9]
	\arrow[shift right, from=2-8, to=1-9]
	\arrow["{1,1}"{description}, from=3-4, to=2-3]
	\arrow[from=3-9, to=2-8]
\end{tikzcd}\]
    
\end{example}
\textbf{General notation.} If $b \in \RR$, then $$ b^+:= \max\{b,0\} \quad b^-:= \max\{-b,0\} \quad \text{so that} \quad b=b^+-b^-.$$
If $J$, $K$ are finite sets and $M \in \RR^{J \times K}$, then $M^\pm \in \RR_{\geq 0}^{J \times K}$ is the matrix obtained applying component-wise to $M$ the corresponding operation. If $k \in K$, the $k$\textit{-th column} of $M$ is  $M_{\bullet,k}=M_{|J \times \{k\}} \in \RR^J.$ 
Similarly,  if $M \in \RR^{J \times K}$ and $J_1 \subseteq J$, $K_1 \subseteq K$, then $M_{J_1 \times K_1}= M_{| J_1 \times K_1}$ and $M_{\bullet,K_1}= M_{J \times K_1}.$\\
If $S$ is any set and we consider  $S$-valued matrices of size $J \times K $, that is elements of $S^{J \times K}$, then we use analogue notations for restrictions.\\
Let $H$ be an abelian group and $h \in H^J$. If $m \in \ZZ^J$ and $M \in \ZZ^{J \times K}$, then $$h^m:= \sum m_j h_j \quad \text{while} \quad h^M =(h^{M_\bullet,k})_{k \in K} \in (H)^K .$$
If $H= \KK^*$, then $h^m$ is a multiplicative monomial in the $h_j$.
Note that, if  $I$ is a third finite set and $N \in \ZZ^{K \times I}$ then $$h^{M N}=(h^M)^N$$ 
where the product $ \ZZ^{J \times K} \times \ZZ^{K \times I} \longto \ZZ^{J \times I }$ corresponds to composition of morphisms in the canonical identification between $\ZZ^{J \times K}$ and $\Hom_\ZZ(\ZZ^K , \ZZ^J).$
The elements of the canonical basis of $\ZZ^K$ are denoted by $e_k$, for $k \in K$. Finally, if $n \in \NN$ we note $$[n]= \{1, \dots , n\}.$$

\bigskip

Given a seed $t$ and a mutable vertex $k \in I_{uf}$, we have an operation called \textit{seed mutation at k}, denoted by $\mu_k$. This operation produces a new seed $\mu_k(t)=t'$, of $\KK$, defined as follows.
\begin{itemize}
    \item[--] The vertex sets of $t'$ are the same of $t$, that is $I_{uf}'=I_{uf}, I_{sf}'=I_{sf}$ and $I_{hf}'=I_{hf}$.
    \item[--] The extended exchange matrix $B' $ satisfies:

   \begin{equation}
       \label{matrix mutation}
       b'_{i,j}=\begin{aligned}
           \begin{cases} 
           -b_{i,j} & \text{if} \quad i=k \quad \text{or} \quad j=k\\
           b_{i,j}+b_{i,k}^+b_{k,j}^+ - b_{i,k}^-b_{k,j}^- & \text{otherwise}.
           \end{cases}
       \end{aligned}
      \end{equation}
   \item[--] The extended cluster $x'$ is defined by
   \begin{equation}
       \label{cluster mutation}
      x_i'= \begin{aligned}
           \begin{cases} 
           x_i & \text{if} \quad i\neq k \\
           x \strut^{B_{\bullet,k}^+ -e_k} + x \strut^{B_{\bullet,k}^- - e_k} & \text{if} \quad i = k.
           \end{cases}
       \end{aligned}
      \end{equation}
\end{itemize}
The identity 

\begin{equation}
    \label{exchange relation}
    x_kx_k'= x \strut^{B_{\bullet,k}^+} + x \strut^{B_{\bullet,k}^-}
\end{equation}
is called \textit{exchange relation}.
We often denote the two monomials on the right hand side of the exchange relation by $M^+_k$ and $M_k^-$, in the obvious way.

Given a seed $t'$ of $\KK$, we say that $t'$ is \textit{(mutation-)equivalent} to $t$, and write $t' \sim t$, if there exists $i_1, \dots , i_l \in I_{uf}$ such that $$ \mu_{i_l} \circ \dots \circ \mu_{i_1}(t)=t'.$$
We call $\Delta$ the set of seeds of $\KK$ equivalent to $t$. 

\bigskip

From now on, we fix  an equivalence class of seeds $\Delta$ of $\KK$. Since any seed in $\Delta$ has the same set of frozen variables, say $x_i$ for $i \in I_f$, we can define the \textit{coefficient ring} 
\begin{equation}
 \label{coefficient ring}   
\coef[\Delta]:= \CC [x_i]_{i \in I_{hf}} [x_i^{\pm1}]_{i \in I_{sf}}.
\end{equation}

\begin{definition}
    \label{cluster algebra}
    The \textit{cluster algebra} $\clu(\Delta)$ is the $\coef[\Delta]$-algebra generated by all the cluster variables of all the seeds in $\Delta$. 
\end{definition}

If $t \in \Delta$, we define the ring of \textit{Laurent Polynomials}

\begin{equation}
    \label{Laurent Poly}
    \Li(t):= \coef[\Delta][x_i^{\pm1}]_{i \in I_{uf}}
\end{equation}
Note that the unfrozen and semi-frozen variables of $t$ are invertible in $\Li(t)$, while the highly frozen are not.

\begin{definition}
The \textit{upper cluster algebra} $\uclu(\Delta)$ is  $$ \uclu(\Delta):= \bigcap_{t \in \Delta} \Li(t).$$
\end{definition}
Both $\clu(\Delta)$ and $\uclu(\Delta)$ are domains, moreover $\uclu(\Delta)$ is normal since it is defined as the intersection of normal rings. Still, both the cluster algebra and the upper cluster algebra are not noetherian in general. Note that, the choice of the ambient field $\KK$ (as the choice of the variables $x$) is immaterial for the isomorphism class of the  cluster algebra and the upper cluster algebra. These choices become important when we try to identify cluster algebra structures on a given ring. So, when we only deal with (upper) cluster algebras, we often omit to specify the ambient field $\KK$. The following theorem can be found in \cite{gross2018canonical}[Corollary 0.4] or in \cite{lee2015positivity} for skew-symmetric seeds.

\begin{theo}[Positivity of the Laurent phenomenon]
\label{laurent pheno}
For any $t, t' \in \Delta$ and $i \in I$, $$ x_i' \in \NN[x_j]_{j \in I_f}[x_k^{\pm1}]_{k \in I_{uf}}$$
In particular $\clu(\Delta) \subseteq \uclu(\Delta).$
\end{theo}

 The weaker statement that $ x_i' \in \ZZ[x_j]_{j \in I_f}[x_k^{\pm1}]_{k \in I_{uf}}$, usually called \textit{Laurent phenomenon}, is much more elementary and also sufficient for most applications. It  can be found in \cite[Proposition 11.2]{fomin2002cluster}. (See also  \cite[Theorem 3.1]{fomin2002clusterfoundation}).

\bigskip

Clearly, if $t \in \Delta$, we also use the notation $\clu(t)$ (resp. $\uclu(t)$) to denote $\clu(\Delta)$ (resp. $\uclu(\Delta)$).

\begin{definition}
    \label{upper bound}
Let $t \in \Delta$. For $i \in I_{uf}$, let $t_i:= \mu_i(t)$ and $t_0:= t$. The \textit{upper bound} $\upp(t)$ at $t$ is $$ \upp(t):= \bigcap_{i \in I_{uf} \cup \{0\} }\Li(t_i) .$$
\end{definition}

If the matrix $B$ is of maximal rank, we say that $t$ is of \textit{maximal rank}. The rank of the generalised exchange matrix is invariant under mutation by \cite[Lemma 3.2]{berenstein2005cluster3}. The following theorem is a very special case of \cite[Theorem 3.11]{gekhtman2018drinfeld}. The case with no highly frozen vertices had already been proved in \cite[Corollary 1.7]{berenstein2005cluster3}.
\begin{theo}
    \label{upper cluster equals upper bound}
    If $t$ is of maximal rank, $\uclu(\Delta)=\upp(t)$.
\end{theo}

We say that no mutable vertex of $t$ is completely disconnected if, for any $k \in I_{uf}$, there exists $i \in I$ such that $b_{i,k} \neq 0.$ It's easy to see that this property is invariant under mutation. So, we say that no mutable vertex of $\Delta$ is completely disconnected if one, hence any, of its seeds has this property. Note that completely disconnected mutable vertices are the ones that produce trivial exchange relations.

\begin{theo}[\cite{geiss2013factorial}]
\label{cluster variables irred}
    Suppose that  no mutable vertex of $\Delta$  is completely disconnected. Let $t, t' \in \Delta$ and $i,j \in I_{uf} \sqcup I_{hf}$, then 
    \begin{enumerate}
        \item The variable $x_i$ is irreducible in $\uclu(\Delta)$.
        \item The ideals $(x_i)$ and $(x_j')$ of $\uclu(\Delta)$ are equal if and only if $x_i=x_j'$.
         \item The invertible elements of $\uclu(\Delta)$, up to scalar, are the monomials in the semi-frozen variables.
    \end{enumerate} 
\end{theo}

The above theorem is proved in \cite[Theorem 1.3, Corollary 2.3]{geiss2013factorial} for the ordinary cluster algebra under the assumption that the seed $t$ is connected (see \cite[Section 1.2]{geiss2013factorial}). The exact same proof also works for the upper cluster algebra and the upper bound under the weaker assumption that no mutable vertex is completely disconnected.

\subsection{Disjoint union of seeds}
\label{sec:disj union}
This section  is only needed for an application in Section \ref{application to G section}.

\bigskip

We introduce a notation. If $M \in  \RR^{(J_1 \sqcup J_2) \times (K_1 \sqcup K_2)}$, then 
\begin{equation}
    \label{eq:restr matrices}
    M= \begin{pmatrix}
    M_{1,1} & M_{1,2}\\
    M_{2,1} & M_{2,2}
\end{pmatrix}
\end{equation}
 means that $M_{J_a \times K_b}= M_{a,b} $, for $a,b \in \{1,2\}$. 
\label{disjoint union of seeds section}

\begin{definition}
    \label{disjoint union seeds defi}
    Let $t $  and $t'$ be two seeds. The \textit{disjoint union} of $t$ and $t'$ is the seed $t|t'= (J_{uf}, J_{sf},J_{hf}, C, z)$ of $\CC(x_i,x'_{i'})_{i \in I,i' \in I'}$ defined by
    \begin{itemize}
        \item[-] The set of vertices 
        $$J_{uf}= I_{uf} \sqcup I'_{uf} \quad J_{sf}= I_{sf} \sqcup I'_{sf} \quad J_{hf}= I_{hf} \sqcup I'_{hf}. $$
        \item[-]  The generalised exchange matrix
        $$ C= \begin{pmatrix}
            B & 0 \\
            0 & B'
        \end{pmatrix}.$$
        \item[-] The cluster $z$, defined by 
       $$ z_I= x \quad z_{I'}= x'.$$
    \end{itemize}
\end{definition}

We identify $\CC(x_i,x'_{i'})_{i \in I,i' \in I'}$ with the fraction field of $\Li(t) \otimes \Li(t')$ in the natural way. It's clear that $\Li(t|t')=\Li(t) \otimes \Li(t')$. The following property is immediate to verify:

\begin{lemma}
\label{commutation disjoint union}
    Let $t,t'$ as above, $i \in I_{uf}$, $i' \in I'_{uf}$, then
    $$\mu_i(t|t')=\mu_i(t)|t' \quad  \text{and} \quad \mu_{i'}(t|t')=t|\mu_i(t').$$
\end{lemma}

\begin{lemma}
\label{dijoint union tensor product}
    If $t$ and $t'$ are of maximal rank, then $\uclu(t|t')=\uclu(t) \otimes \uclu(t').$
\end{lemma}

\begin{proof}
    Clearly $t|t'$ is of maximal rank. The lemma follows easily from repeated application of Theorem \ref{upper cluster equals upper bound}, the previous lemma and the fact that the tensor product commutes with finite intersections.
\end{proof}

\subsection{Cluster valuations induced by frozen variables}
We introduce the notion of \textit{cluster valuation} at a frozen vertex. This notion implicitly appears in the existing literature. We develop here some technical detail because of a lack of an appropriate reference.

If $t\in \Delta$, we denote by $\CC[t]$ (resp. $\CC(t)$) the polynomial ring (resp. the field of fractions) in the cluster variables $x_i$. By the Laurent phenomenon (Theorem \ref{laurent pheno}) we have that $\CC[t] \subseteq \uclu(\Delta)$ and the fraction field of $\uclu(\Delta)$ is $\CC(t)$, which from now on will be denoted by $\CC(\Delta)$.

Any cluster variable $x_i$ generates a prime ideal in $\CC[t]$, so we can consider the induced discrete valuation $$\val_i^t : \CC(\Delta) \longto \ZZ \cup \{ \infty \}.$$
For an element $p \in \CC[t] \subseteq \CC(\Delta)$, the exponent of $x_i$ in the factorisation of $p$ into irreducible factors, in the ring $\CC[t]$, is $\val_i^t(p).$

\begin{lemma}
    \label{cluster valuation}
    If $i \in I_f$ and $t \sim t'$, then $\val^t_i= \val^{t'}_i$. In particular, we refer to any such valuation as  the \textit{cluster valuation} at the frozen vertex $i$ and denote it by $\cval_i$.
\end{lemma}

\begin{proof}
    We can assume that $t'= \mu_k(t)$, where $k \in I_{uf}.$ Let $$\CC[I\setminus\{k\}]:= \CC[x_j]_{j \in I \setminus\{k\}}=  \CC[x_j']_{j \in I \setminus \{ k \} }$$
    and $\CC(I \setminus\{k\})$ the fraction field of $\CC[ I \setminus\{k\}]$. The valuations $\val_i^t$ and $\val_i^{t'}$ obviously coincide on $\CC(I \setminus\{k\})$.
    Since the fraction field of $\CC(I\setminus\{k\})[x_k^{\pm1}]$ is equal to $\CC(\Delta)$, it's sufficient to prove that $ \val_i^t(f)=\val_i^{t'}(f)$ for any $f \in \CC(I\setminus\{k\})[x_k^{\pm1}].$
If
    
    \begin{equation}
        \label{expression f}
    f = \sum_{n=-N}^N f_nx_k^n   
    \end{equation} 
with $f_n \in  \CC(I\setminus\{k\})$,  we claim that
    
    \begin{equation}
        \label{valuation and min}
        \val^t_i(f) = \min \{ \val_i^t(f_n) \, : \, -N \leq n \leq N \}.
    \end{equation} 
   To prove \eqref{valuation and min}, it's sufficient to notice that if $f_n \neq 0$, then  
\begin{equation}
\label{fn substitutio}
\val_i^t (f_n)= v_n \ifff f_n = x_i^{v_n}  \frac{p_n}{q_n}\end{equation}
with $ p_n, q_n \in \CC[I\setminus\{k\}] $ not divisible by $ x_i$. Then, we can compute $\val_i^t(f)$ plugging the right hand side of \eqref{fn substitutio} into \eqref{expression f} and we easily obtain \eqref{valuation and min}.

Using the exchange relation \eqref{exchange relation}, we compute that $$ f = \sum_{n=-N}^N f_n (M_k^+ + M_k^-)^n (x_k')^{-n} \in \CC(I \setminus \{ k \} ) [(x_k')^{\pm 1}] $$

But $M_k^+ + M_k^-$ is not divisible by $x_i$, hence 
$$ \val^t_i(f_n)= \val^t_i(f_n (M_k^+ + M_k^-)^n)=\val_i^{t'}(f_n (M_k^+ + M_k^-)^n). $$
The statement follows using \eqref{valuation and min} and the similar formula relative to the seed $t'$.

\end{proof}

At the level of cluster varieties, the previous lemma says that the gluing maps between cluster tori preserve the valuation defined by the frozen variables. We can strengthening Theorem \ref{cluster variables irred} in the case of highly frozen variables as follows.

\begin{coro}
\label{highly frozen prime}
If $i \in I_{hf}$, then $x_i$ generates a prime ideal of $\uclu(\Delta)$. Moreover, $\uclu(\Delta)_{(x_i)}$ is a discrete valuation ring whose induced valuation on $\CC(\Delta)$ equals the cluster valuation $\cval_i.$
\end{coro}
\begin{proof}
Fix $t \in \Delta$. Take $a, b \in \uclu(\Delta)$ such that $ab \in (x_i).$ Since $x_i$ generates a prime ideal in $\Li(t)$, we can assume that $\frac{a}{x_i} \in \Li(t)$. Hence, $\val^t_i( \frac{a}{x_i} )\geq 0$. By the definition of $\uclu(\Delta)$, it's clear that for any $t' \in \Delta$, $\frac{a}{x_i} \in \frac{1}{x_i}\Li(t') \subseteq\Li(t')_{x_i}$. Lemma \ref{cluster valuation} implies that $\val^{t'}_i(\frac{a}{x_i}) \geq 0$, hence $\frac{a}{x_i} \in \Li(t')$.
It follows that $\frac{a}{x_i}\in \uclu(\Delta)$ and the first part of the statement follows.
For the second statement, notice that we have inclusions $$\CC[t] \subseteq \uclu(\Delta) \subseteq \Li(t). $$
In particular, by the definition of $\Li(t)$, localizing the above inclusions at $\prod_{j \in I_{uf} \sqcup I_{sf}} x_j$ gives a sequence of isomorphisms. Thus $\CC[t]_{(x_i)}= \uclu(\Delta)_{(x_i)}= \Li(t)_{(x_i)}$. 
\end{proof}

\begin{remark}
One can easily prove that the cluster valuation is a tropical valuation on $\uclu(\Delta)$ in the sense of \cite[Definition 7.1]{berenstein2005cluster3}. Also, it agrees with one of the valuations constructed in \cite[Lemma 7.3]{berenstein2005cluster3}.
\end{remark}

\subsection{Highly-freezing and semi-freezing}
We introduce two operations on seeds that change the nature of some frozen vertices from highly-frozen to semi-frozen, or the other way around. The notation is designed to be coherent with the fact that, passing from $t$ to $t_F$ corresponds to a localisation at the level of upper cluster algebras, while passing from $t$ to $t^F$ is the inverse operation, as described in Lemma \ref{properties freezing}.

\begin{definition}
\label{semi-higly-freezing}Let $ t$ be a seed.

\begin{itemize}
    \item[-] If $F \subseteq I_{hf}$, we define the seed $$t_F= (I_{uf}, \,  I_{sf} \sqcup F, \, I_{ hf} \setminus F, \,  B, \, x).$$ 
    We say that $t_F$ is obtained from $t$ by \textit{semi-freezing} the vertices in $F$. If $\Delta= \Delta(t)$, then we  denote $\Delta_F= \Delta(t_F)$.
    
    \item[--] If $F \subseteq I_{sf}$, we define the seed 
    $$t^F= (I_{uf}, \, I_{sf} \setminus F, \,  I_{ hf} \sqcup F, \, B, \, x).$$ 
    We say that $t^F$ is obtained from $t$ by \textit{higly-freezing} the vertices in $F$. If $\Delta= \Delta(t)$, then we denote $\Delta^F= \Delta(t^F)$.
\end{itemize}

\end{definition}

\begin{lemma}
\label{properties freezing} Let $t$ be a seed, $\Delta= \Delta(t)$, $F \subseteq I_{hf}$, $G \subseteq I_{sf}$, $k$ (resp. $j$) a mutable (resp. frozen) vertex of $t$. 
\begin{enumerate}
    \item  $(t_F)^F=t$.
    \item $(t^G)_G=t$.
    \item $\CC(\Delta)= \CC(\Delta_F)= \CC(\Delta^G)$.
    \item $\val^{t^G}_j= \val^t_j= \val^{t_F}_j.$
    \item $\mu_k(t_F)= \mu_k(t)_F$ \quad and \quad $\mu_k(t^G)= \mu_k(t)^G$.
    
    \item $\Delta_F = \{ t'_F \, : \, t' \in \Delta \}$ \quad and \quad $\Delta^G= \{ (t')^G \, : \, t' \in \Delta \}.$
    \item If $z:= \prod_{i \in F}(x_i)$, the natural inclusion $\uclu(\Delta)\subseteq \uclu(\Delta_F)$ induces an equality $$ \uclu(\Delta)_z = \uclu(\Delta_F).$$
    \item $\uclu(\Delta^G)= (\bigcap_{i \in G} \cval_i\inv (\ZZ_{\geq 0} \cup \{ \infty \})) \bigcap \uclu(\Delta)$.
\end{enumerate}
\end{lemma}

\begin{proof}
The first six statements are obvious. 
Then it's clear that $\Li(t^G) \subseteq \Li(t) \subseteq \Li(t_F)$. Using 6, we get that $$\uclu(\Delta^G) \subseteq \uclu(\Delta) \subseteq \uclu(\Delta_F). $$
Moreover, $z$ is invertible in $\uclu(\Delta_F)$, hence we have an inclusion $\uclu(\Delta)_z \subseteq \uclu(\Delta_F)$. 
If $f \in \uclu(\Delta_F)$, by Lemma \ref{cluster valuation} and the definition of upper cluster algebra, we get that $$( \prod_{i \in F} x_i^{-\val^t_i(f)})f \in \uclu(\Delta). $$
This proves 7. Statement 8 also follows immediately from Lemma \ref{cluster valuation} and the definition of upper cluster algebra. 

\end{proof}

\subsection{Graded seeds}
\label{graded seed subsection}

Cluster algebras and upper cluster algebras can be graded almost as if they where polynomial rings. In particular, defining the degree of the cluster variables of an initial seed and using the mutation process induces, under some hypothesis, global graduations. This is formalised by the notion of graded seed. We refer to \cite{grabowski2015graded} for more details.

\bigskip

Let $t$ be a seed.

\begin{definition}
Let $H$ be an abelian group and $\sigma \in H^I$. We say that $\sigma$ is an $H$-\textit{degree configuration} on $t$ if

\begin{equation}
\label{graduation condition}
    \sigma^{B^+}= \sigma^{B^-}.
\end{equation}
A seed with an $H$-degree configuration is called $H$-\textit{graded seed}.
\end{definition}

We drop the dependence on $H$, in the notation, if the group is clear from the context or meaningless. If we denote a graded seed by $t$ (resp. $t^\bullet$, resp. $t_\bullet$ for a certain symbol $\bullet$), we implicitly assume that its degree configuration is $\sigma$ (resp. $\sigma^\bullet$, resp. $\sigma_\bullet$) and that the grading group is $H$.

\bigskip

From now on, $t$ is an $H$-graded seed. We have a notion of mutation of degree configuration. Let $k \in I_{uf}$ and $t'=\mu_k(t)$. Then  $\sigma'=\mu_k(\sigma)$ is the $H$-degree configuration, on $t'$, defined by the following formula:

\begin{equation}
  \label{mutation graduation}  
  \sigma'_i =\begin{aligned}
           \begin{cases} 
           \sigma_i & \text{if} \quad i \neq k \\
         \sigma^{B_{\bullet,k}^+} - \sigma_k & \text{if} \quad i=k.
          \end{cases}
       \end{aligned}    
\end{equation}

We say that $\sigma'$ is the \textit{mutation at $k$} of the degree configuration $\sigma.$ Any sequence of mutations defines a degree configuration on the resulting seed, which actually only depends on the seed and not on the sequence of mutations. Hence, any $t' \in \Delta(t)$ is in a canonical way a graded seed.

\bigskip

Note that the algebra $\Li(t)$ is canonically $H$-graded by  setting $\deg(x_i)= \sigma_i$. The set of Laurent monomials is an homogeneous base of $\Li(t)$.

\begin{prop}
   Let $t^* \in \Delta(t)$. Then $\clu(t)$ and $\uclu(t)$ are graded subalgebras of $\Li(t^*)$. The graduation defined on $\clu(t)$ and $\uclu(t)$, by this inclusion, is independent on $t^*$. In particular, $\clu(t)$ and $\uclu(t)$ are canonically graded. Moreover, cluster variables are homogeneous.
   
\end{prop}

\begin{remark}
    Since any cluster variable is homogeneous, the construction above is compatible with the operation of highly freezing or semi-freezing a set of vertices.
\end{remark}

Note that all the basis constructed in \cite{qin2022bases} consist of homogeneous elements.  
Finally, if $t$ is a not necessarily graded seed, and we identify $B$ as an element of $\Hom_\ZZ(\ZZ^{I_{uf}}, \ZZ^I)$, then we have an universal $\coker(B)$-graduation $\overline{\sigma}$ on $t$. This is defined by $\overline{\sigma}_i=\overline{e}_i$, where $\overline{e}_i$ is the class of $e_i$ in $\coker(B)$. In particular, roughly speaking, having many frozen vertices allows to construct fine graduations on $\clu(t)$ and $\uclu(t)$.
\section{Monomial liftings}
\label{monomial lifting section}

From now on, $D$ denotes a finite set. We define the notion of $D$-lifting configuration on a seed $t$. For any such configuration, we construct a seed $\lif t$, which is naturally graded, to which we refer as a \textit{monomial lifting} of $t$. The name is explained as follows. Let $Y$ be a variety carrying a cluster structure. Under mild assumptions on $Y$, any monomial lifting of $t$ gives a cluster structure to $\GG_m^D \times Y$, which extends the natural one on $\{e\} \times Y$ compatibly with the $\GG_m^D$-action. Conversely, any homogeneous cluster structure on $\GG_m^D \times Y$ which extends the one on $\{e\} \times Y$ is defined by a monomial lifting. We start with some definitions.

\bigskip

From now on, $t$ is a seed of the field $\KK$ and $\Delta= \Delta(t).$ We assume that $I \cap D= \emptyset.$

\begin{definition}
\label{pointed field ext}
    A $D$-pointed field extension of $\KK$ is a field extension $\LL / \KK $, with a $D$-uple $(x_d)_{d \in D} \in (\LL^*)^D$ of $\KK$-algebraically independent elements, such that $\LL=\KK(x_d)_{d \in D}.$ 
\end{definition}
If we denote a $D$-pointed field extension of $\KK$ by $\LL$, we implicitly assume that it's defining $D$-uple of elements is denoted by $(x_d)_d.$ Whenever a seed $t$ of $\KK$ and a $D$-pointed field extension $\LL$ of $\KK$ are given, we denote by 
    \begin{equation}
        \label{tlif x}
        \hx = (x_j)_{j \in I \sqcup D} \in (\LL^*)^{I \sqcup D}.
    \end{equation}
\begin{definition}
\label{seed extension}
    A \textit{$D$-seed extension} of $t$ is a seed $  t^*$, of a field extension $\LL/\KK$, such that
    
    \begin{itemize}
        \item [-] $I_{uf}=  I^*_{uf} \quad I_{sf} \subseteq  I^*_{sf} \quad  I_{hf} \subseteq  I^*_{hf} $ \quad and \quad $I^*_f \setminus I_f=D$.
        \item [-] $B^*_{I \times I_{uf}}=B.$
    \end{itemize}
    If moreover $\LL$ is $D$-pointed and for any $d \in D$ we have $x^*_d=x_d$, then we say that $t^*$ is a $D$-pointed seed extension of $t$. 
\end{definition}
In the next definition, we repeatedly use the notation introduced in \eqref{eq:restr matrices}.
\begin{definition}[monomial lifting]
\label{monomial lifting defi}
    A \textit{$D$-lifting configuration} on $t$ is a pair $(\LL, \nu)$ consisting of: a $D$-pointed field extension $\LL$ of $\KK$ and an integer matrix $\nu \in \ZZ^{D \times I}$. We call $\nu$ the $D$-\textit{lifting matrix} of the $D$-lifting configuration. 
    The \textit{monomial lifting} of $t$, defined by $(\LL, \nu)$, is the seed $\lif t =( \lif I_{uf}, \lif I_{sf}, \lif I_{uf}, \lif B , \lif x) $ defined by:

    \begin{itemize}
        \item[--] The set of vertices  $\lif I_{uf}=I_{uf}, \quad \lif I_{sf}=I_{sf}\sqcup D, \quad  \lif I_{hf}= I_{hf}. $
        \item[--] The matrix
        \begin{equation}
            \label{formula lifting B}
            \lif B= \begin{pmatrix}
            B \\
            -\nu B
        \end{pmatrix} \in \ZZ^{(I \sqcup D) \times I_{uf}}.
        \end{equation}
         The coefficients of $\lif B$ are denoted by $\lif b_{i,j}$ for $i \in \, \lif I, j \in \, \lif I_{uf}$.
        
\item[--] The cluster
\begin{equation}
  \label{formula lifting variables}  
 \lif x= \hx\strut^{\hnu} \quad \text{where} \quad  \hnu= \bmat Id & 0 \\
 \nu & Id
 \emat \in \ZZ^{(I \sqcup D) \times (I \sqcup D)}.
\end{equation}
    \end{itemize}
    In this case we say that $\nu$ lifts $t$ to $\lif t$ and write: $\lif t \lifa t.$ 
\end{definition}

Any monomial lifting $\lif t$ is a $D$-pointed seed extension of $t$. Recall that, 
$\hx_D= (x_d)_d \in (\LL^*)^D.$ So, formula \eqref{formula lifting variables} means that 
$$\lif x_i = \begin{cases}
    x_i \,\hx\strut^{\nu_{\bullet,i}}_D= x_i \prod_d x_d^{\nu_{d,i}} & \text{if} \quad i \in I\\
    x_d & \text{if} \quad i=d \in D.
\end{cases} $$
Hence, we can think about $\nu$ as a collection of monomials, indexed by $I$, in the variables $\hx_D$. The cluster variables $\lif x_i$, for $i \in I$, are the product of $x_i$ and the corresponding monomial.

\begin{example}
    \label{ex:monomial lifting} 
    Let $t$ be a seed of the field $\KK$, which is graphically described by the following quiver 
    \[\begin{tikzcd}
	{\bigcirc 1} & {\bigcirc 2} & \blacksquare3
	\arrow[from=1-3, to=1-2]
	\arrow[from=1-2, to=1-1]
\end{tikzcd}.\]
That is: $I_{uf} = \{1,2\}$, $I_{sf}= \emptyset$, $I_{hf}=\{3\}$ and 
$$B=\bmat 
0 & -1\\
1 & 0\\
0 & 1
\emat.$$
Let $D= \{d\}$. Here $d$ is considered as a symbol, $\LL= \KK(x_d)$ where $x_d$ is a variable and $\nu=(1,2,3).$ Then, a simple matrix multiplication implies that $\lif t$ corresponds graphically to the quiver
\[\begin{tikzcd}
	& \blacksquare d \\
	{\bigcirc 1} & {\bigcirc 2} & \blacksquare3
	\arrow[from=2-3, to=2-2]
	\arrow[from=2-2, to=2-1]
	\arrow[shift left, from=2-1, to=1-2]
	\arrow[shift left, from=2-2, to=1-2]
	\arrow[from=2-1, to=1-2]
	\arrow[from=2-2, to=1-2]
\end{tikzcd} \quad \text{and} \quad
\lif B=\bmat 
0 & -1\\
1 & 0\\
0 & 1 \\
-2 & -2
\emat. \]
Moreover, from the definition, we have that 
$$\lif x_1= x_1 x_d \quad \lif x_2= x_2 x_d^2 \quad \lif x_3= x_3 x_d^3 \quad \lif x_d=x_d.$$
\end{example}

Monomial liftings can be performed in steps. In fact, suppose that $D=D_1 \sqcup D_2$ and $\LL$ is a $D$-pointed field extension of $\KK$. Then the field $\LL_1=\KK(x_{d_1})_{d_1 \in D_1}$ is a $D_1$-pointed extension of $\KK$ and $\LL$ is a $D_2$-pointed extension of $\LL_1$. The following lemma is obvious from the definitions.

\begin{lemma}
    \label{lift in steps}
   Let $(\LL,\nu)$ be a $D$-lifting data on $t$. If $D=D_1 \sqcup D_2$, and $\nu_j=\nu_{D_j, \bullet}$,  for $j=1,2$, then the following diagram is commutative.

   \[\begin{tikzcd}
  	& \bullet \\
	\bullet && t
	\arrow["{\nu_1}"', from=2-3, to=1-2]
	\arrow["{(\nu_2, 0)}"', from=1-2, to=2-1]
	\arrow["\nu"', from=2-3, to=2-1]
\end{tikzcd}\]

\end{lemma}
We introduce the following notion of mutation of lifting configuration.

\begin{definition}[mutation of lifting configuration]
    \label{mutation lifting configuration defi}
    Let $(\LL,\nu)$ be a $D$-lifting configuration on $t$, $k \in I_{uf}$  and $t':= \mu_k(t)$. Consider the $D$-lifting matrix $\nu'$ on $t'$ defined by

    \begin{equation}
  \label{mutation lifting configuration formula}  
  \nu'_{\bullet,i}  =\begin{aligned}
           \begin{cases} 
           \nu_{\bullet,i} & \text{if} \quad  i \neq k  \\
          \max \{ \nu B_{\bullet,k}^+ \, , \,  \nu B_{\bullet, k}^- \} - \nu_{\bullet,k} & \text{if} \quad i=k
          \end{cases}
       \end{aligned}    
\end{equation}
where the $\max$ in \eqref{mutation lifting configuration formula} is taken component-wise. We say that $\nu'=\mu_k(\nu)$ is the mutation at $k$ of the lifting matrix $\nu$. Moreover, the lifting configuration $(\LL, \nu')= \mu_k(\LL, \nu)$, on $t'$, is the mutation at $k$ of the lifting configuration $(\LL, \nu).$
\end{definition}
\begin{remark}
\label{alternative mutation lifting}
    For any $a,b \in \RR$ we have that $\max\{a, b\}= (a-b)^+ + b$. Using this formula and the fact that $b=b^+-b^-$, we get two alternative forms of formula \eqref{mutation lifting configuration formula}, namely:  \begin{equation}
    \label{eq 18}
        \nu'_{\bullet, k}= ( \nu B_{\bullet, k})^+ + \nu B^-_{\bullet,k}- \nu_{\bullet, k}
         \end{equation}
    and 
    \begin{equation}
    \label{eq 19}
        \nu'_{\bullet, k}=  ( - \nu B_{\bullet, k})^+ + \nu B^+_{\bullet,k}- \nu_{\bullet,k}.
        \end{equation}
\end{remark}
If the $D$-pointed field $\LL$ is clear, we identify a lifting configuration with its lifting matrix. From now on, we fix $\LL$ and suppose that any lifting configuration has $\LL$ as defining field. The following lemma is crucial.
\begin{lemma}
\label{mutation commutes lifting}
Lifting seeds commutes with mutation. That is, for any lifting configuration $\nu$ on $t$ and any $k \in I_{uf}$, if we denote by $t'=\mu_k(t)$ and $\nu'=\mu_k(\nu)$, then the following diagram is commutative

\[\begin{tikzcd}
	\lif t & \lif t' \\
	t & {t'}
	\arrow["{\nu'}"', from=2-2, to=1-2]
	\arrow["{\mu_k}"', from=2-1, to=2-2]
	\arrow["\nu", from=2-1, to=1-1]
	\arrow["{\mu_k}", from=1-1, to=1-2]
\end{tikzcd}\]
\end{lemma}

\begin{proof}
  Call $t^*=\mu_k(\lif t)$. We have to prove that $t^*=\lif t'$, where $\lif t' \lifap t$. It's clear that the two seeds have the same vertices sets. To avoid possible confusion, we stress that $\lif B^\pm$ denotes $(\lif B)^\pm$ and that $\lif b_{i,j}^\pm$ are the coefficients of $\lif B^\pm$.

\bigskip
  
 First we prove that $x^*=\lif x'$. Clearly, if $d \in D$, then $x^*_d=\lif x'_d= x_d$. If $i \in I $ and $i \neq k$, then we have $x_i^*=\lif x_i= \lif x_i'$ because of \eqref{mutation lifting configuration formula} and \eqref{cluster mutation}. Finally, we can compute that:

   \begin{align}
         x_k^* = & \spc  \displaystyle  \lif x \strut^{ \lif \kern0.2em B_{\bullet, k}^+ - e_k} +  \lif x\strut^{\lif \spc B_{\bullet, k}^- - e_k} \notag  \\[0.4em]
          = & \spc  \hx \strut^{ \hnu \cdot \lif \spc B^+_{\bullet,k} - \hnu_{\bullet,k}}+ \hx\strut^{ \hnu \cdot \lif \spc B^-_{\bullet,k} - \hnu_{\bullet,k}}\notag \\[0.4em]
          =  & \spc  x\strut^{B^+_{\bullet,k}-e_k} \hx_D\strut^{\nu \cdot B^+_{\bullet,k}+ (-\nu \cdot B_{\bullet,k})^+ - \nu_{\bullet,k}} + x\strut^{B^-_{\bullet,k}-e_k} \hx_D\strut^{\nu \cdot B^-_{\bullet,k}+ (-\nu \cdot B_{\bullet,k})^- - \nu_{\bullet,k}} \label{eq 17}
\end{align}

The first equality is the mutation rule \eqref{cluster mutation}, the second is obtained from the definition of $\lif x$ \eqref{formula lifting variables} and the third follows from the following expressions:

\begin{equation}
    \label{eq 15}
    \begin{split}
    \hnu \cdot \lif B^+= 
    \bmat Id & 0\\
    \nu & Id
    \emat  \bmat B^+\\
    (-\nu B)^+ 
    \emat = \bmat B^+ \\
    \nu B^+ + (-\nu B )^+
    \emat \\
     \hnu \cdot \lif B^-= 
    \bmat Id & 0\\
    \nu & Id
    \emat  \bmat B^-\\
    (-\nu B)^-
    \emat = \bmat B^- \\
    \nu B^- + (-\nu B )^-
    \emat
    \end{split}
\end{equation}

Recall that $\lif B^+-\lif B^-=\lif B$, in particular
\begin{equation}
    \label{eq 16}
\hnu \cdot  \lif B^+ - \hnu \cdot \lif B ^-= \bmat Id & 0\\
    \nu & Id
    \emat  \bmat B\\
    -\nu B
    \emat = \bmat B \\
    0
    \emat.
\end{equation}

From \eqref{eq 15} and \eqref{eq 16} we deduce that the exponents of $\hx_D$ in the two monomials of expression \eqref{eq 17} are the same. Hence, using \eqref{cluster mutation}, we deduce from \eqref{eq 17} that 
$$ x_k^*= x_k' \hx_D \strut^{\nu \cdot B^+_{\bullet,k}+ (-\nu \cdot B_{\bullet,k})^+ - \nu_{\bullet,k}}.
$$
Comparing this formula to the expression of $\lif x_k'$, we see that it's sufficient to prove that  $$ \nu B^+_{\bullet, k} + (-\nu B_{\bullet, k})^+ - \nu_{\bullet,k}= \nu'_{\bullet, k}.$$
But this is Formula \eqref{eq 19}.

\bigskip

Next, we consider the exchange matrices, that is we prove that $B^*=\lif B'$. It's clear from \eqref{matrix mutation} and \eqref{formula lifting B} that for any $j \in I$ and $i \in I_{uf}$, $b^*_{j,i}=b'_{j,i}=\lif b'_{j,i}$. For $d \in D$, then $$b^*_{d,k}=- \lif b_{d,k}= \nu_{d, \bullet} B_{\bullet, k}. $$
and $$\lif b_{d,k}'= - \nu'_{d,\bullet} B'_{\bullet,k}= \nu'_{d,\bullet} B_{\bullet,k}.$$
The equality $b^*_{d,k}= \lif b'_{d,k}$ follows from the fact that $b_{k,k}= 0$ and the definition of $\nu'$ \eqref{mutation lifting configuration formula}. Next, if $i \in I_{uf} $, $i \neq k$, using \eqref{matrix mutation} and \eqref{formula lifting B} we get that

\begin{align}
b^*_{d,i}=&  \lif b_{d,i}+ \lif b_{d,k}^+ b_{k,i}^+ - \lif b_{d,k}^-b_{k,i}^- \notag\\[0.2em]
= & \lif b_{d,i} + (-\nu_{d, \bullet} B_{\bullet,k})^+b_{k,i}^+ - (-\nu_{d, \bullet} B_{\bullet, k})^- b_{k,i}^- \notag .
\end{align}

Similarly we have that 

\begin{align}
\lif b'_{d,i} = & -\nu'_{d,\bullet}B'_{\bullet,i} \notag \\[0.2em]
= & -\nu_{d,k}'b_{k,i}' + \sum_{j \neq k}-\nu_{d,j}b_{j,i}' \notag \\[0.2em]
= & -\nu_{d,k}b_{k,i} + \max\{ \nu_{d,\bullet}B^+_{\bullet,k} \, , \, \nu_{d,\bullet}B^-_{\bullet,k} \}b_{k,i} - \sum_{j \in I\setminus \{k\}} \nu_{d,j}(b_{j,i}+ b^+_{j,k}b^+_{k,i}-b^-_{j,k}b^-_{k,i}) \notag \\[0.2em]
= & -\nu_{d,\bullet} B_{\bullet,i}+ \max\{\nu_{d,\bullet} B_{\bullet,k}^+ \, , \, \nu_{d,\bullet} B_{\bullet,k}^-\}b_{k,i} -\nu_{d,\bullet}B_{\bullet,k}^+b_{k,i}^+ +\nu_{d,\bullet}B^-_{\bullet,k}b_{k,i}^- \notag \\[0.2em]
= & \lif b_{d,i} + (-\nu_{d,\bullet}B_{\bullet,k})^+b_{k,i} + \nu_{d,\bullet}B^+_{\bullet,k}b_{k,i}-\nu_{d,\bullet}B^+_{\bullet,k}b_{k,i}^+ + \nu_{d,\bullet}B^-_{\bullet,k}b_{k,i}^- 
\label{eq 20} \end{align}
The first equality follows from formula \eqref{formula lifting B}, the second, third and fourth are obtained from formulas \eqref{matrix mutation}, \eqref{mutation lifting configuration formula} and some elementary manipulations.
The last equality is a consequence of formulas \eqref{formula lifting B} and \eqref{eq 19}. Finally, using that $b_{k,i}=b_{k,i}^+-b_{k,i}^-$ in Formula \eqref{eq 20}, we obtain immediately that $\lif b'_{d,i}=b_{d,i}^*$. 
\end{proof}

The following characterisation of monomial liftings turns out to be useful.

\begin{prop}
    \label{unicity monomial lifting}
   Let  $\widetilde t$ be a seed of $\LL$. Suppose that there exist  $\nu \in \ZZ^{D \times I}$, and $\lambda \in \ZZ^{D \times I_{uf}}$ such that \begin{enumerate}
       \item $\widetilde t$ is a $D$-pointed seed extension of $t$ with $D \subset \widetilde I_{sf}$.
       \item For any $i \in I$, $\widetilde x_i= x_i \,\hx\strut^{\nu_{\bullet,i}}_D.$ 
       \item For any $k \in I_{uf}$, $\mu_k(\widetilde x_k)= \mu_k(x_k)  \,\hx\strut^{\lambda_{\bullet,k}}_D$
   \end{enumerate} Then $\widetilde t \lifa t$ and for any $k \in I_{uf}$, $\lambda_{\bullet,k}= \mu_k(\nu)_{\bullet,k}.$
\end{prop}

\begin{proof}
 Denote by $\lif t$ the monomial lifting of $t$ determined by $(\LL, \nu)$. We have to prove that $\widetilde t= \lif t$. 
 The first assumption implies that $\widetilde t$ and $\lif t$ have the same vertices set. Moreover, assumption 1 and 2 imply that $\widetilde x= \lif x$. By assumption 1, $\widetilde B_{I , \bullet }= B = \lif B_{I, \bullet}.$ So we just have to prove that $\widetilde B_{D , \bullet}= -\nu B.$
 Fix $k \in I_{uf}$. Call $t'=\mu_k(t)$ and $t^*= \mu_k(\widetilde t)$. By assumption 3 and \eqref{exchange relation}, we have that 
$$ \begin{array}{rl}
   x_k^* \widetilde x_k= & x_k'x_k \hx_D\strut^{\nu_{\bullet, k} 
 + \lambda_{\bullet,k}}\\[0.2em]
   = & x\strut^{B^+_{\bullet,k}} \hx_D\strut^{\nu_{\bullet, k} 
 + \lambda_{\bullet,k}} +  x\strut^{B^-_{\bullet,k}} \hx_D\strut^{\nu_{\bullet, k} 
 + \lambda_{\bullet,k}}.
\end{array}$$
Always using \eqref{exchange relation}, we can compute that 

$$\begin{array}{rl}
 x_k^*\widetilde x_k= & \widetilde x \strut^{\widetilde B^+_{\bullet,k} }+ \widetilde x \strut^{\widetilde B^-_{\bullet,k} }\\[0.2em]
 = &  \hx \strut^{\hnu \widetilde B^+_{\bullet,k} }+ \hx \strut^{ \hnu \widetilde B^-_{\bullet,k} }\\[0.2em]
 = &  x\strut^{B^+_{\bullet, k}} \hx_D\strut^{\nu B_{\bullet, k}^+ + \widetilde B^+_{D\times k}}+ x\strut^{B^-_{\bullet, k}} \hx_D\strut^{\nu B_{\bullet, k}^- + \widetilde B^-_{D \times k}}.
 \end{array}$$

We deduce that $$ \nu B_{\bullet, k}^+ + \widetilde B^+_{D \times k}= \nu_{\bullet, k}+ \lambda_{\bullet,k}= \nu B_{\bullet, k}^- + \widetilde B^-_{D \times k}. $$
In particular $$ \widetilde B_{D \times k } = \widetilde B_{D \times k }^+ - \widetilde B_{D \times k }^-=  \nu B_{\bullet, k}^- - \nu B_{\bullet, k}^+ = - \nu B_{\bullet, k}.   $$
Finally, the statement about $\lambda$ can be easily deduced from the two above identities and \eqref{eq 19}.

\end{proof}

We introduce some notation. For any finite sets $J$ and $K$, then $\ZZ^{J \times K}$ and $(\ZZ^J)^K$ are canonically identified by the map that assigns to a matrix the collection of its columns. From now on, we make no difference between these two objects. Let $M \in \ZZ^{J \times K} = (\ZZ^J)^K$. Note that if $I$ is a third finite set and $N \in \ZZ^{K \times I}$, then $MN= M^N$. We write $MN$ when we think about $M$ as a matrix and $M^N$ when we think about $M$ as the collection of its columns.\\
If $H,K$ are groups and $h \in H^J, k \in K^J$, then the notation $h \times k$ stands for the element of $(H \times K)^J$ defined by $(h \times k)_j=(h_j,k_j)$. 

\bigskip

From now on, fix a lifting matrix $\nu \in \ZZ^{D \times I}$. In particular $(\LL, \nu)$ is a fixed $D$-lifting configuration on $t$. 

\bigskip

If the seed $t$ graded, the degree configuration of $t$ can be lifted along with $t$.

\begin{definition}[lifting of degree configuration]
    \label{lifting deg def}
    Let $\lif 0:= (\nu, Id) \in (\ZZ^D)^{I \sqcup D}$. If $\sigma \in H^I$ is a degree configuration on $t$, then we define 
    \begin{equation}
        \label{lifting sigma formula}
        \lif \sigma =  \lif 0 \times (\sigma,0) \in (\ZZ^D \times H)^{I \sqcup D}.
    \end{equation}
     In this case we say that $\nu$ lifts $\sigma$ to $\lif \sigma $ and write: $\lif \sigma \lifa \sigma.$ 
\end{definition}

\begin{lemma}
\label{lifting grad lemma}
    For any degree configuration $\sigma$ on $t$, $\lif \sigma$ is a degree configuration on $\lif t$.
\end{lemma}
\begin{proof}
    Note that  $\lif \sigma\strut^{ \lif \spc B}= \lif 0\strut^{\lif \spc B} \times (\sigma,0)\strut^{\lif \spc B}.$ Moreover $(\sigma,0)\strut^{\lif \spc B}= \sigma^B=0$ and $$\lif 0 \strut^{\lif \spc B}= (\nu, \Id) \bmat B \\
    -\nu B \emat =0.$$
\end{proof}

\begin{remark}
\label{canonical deg conf lifting}
    Note that $\lif t$ has a canonical $\ZZ^D$-degree configuration given by $\lif 0=(\nu, \Id)$, which is the lifting of the trivial degree configuration on $t$. The data of this degree configuration is equivalent to the one of $\nu$. 
    As a warning, the degree configuration $\lif 0$ depends both on $t$ and on $\nu$. 
    To avoid ambiguity, we use the convention that $\lif 0$ is always considered as a degree configuration on a seed which is denoted by $\lif t$ and whose corresponding lifting matrix $\nu$ is clear from the context. If a seed is called $t^\bullet$ (resp. $t_\bullet$), for a certain symbol $\bullet$, and a lifting $\lif t^\bullet$ (resp. $\lif t_\bullet$) of $t^\bullet$ (resp. $t_\bullet$) is defined, then we use the notation $\lif 0^\bullet$ (resp. $\lif 0_\bullet$).
\end{remark}

\begin{lemma}
\label{lifting grad commutes mutation}
   Lifting a degree configuration commutes with mutation. In particular, for any degree configuration $\sigma$ on $t$ and any $k \in I_{uf}$, if $\sigma'= \mu_k(\sigma)$, then we have a commutative diagram

   \[\begin{tikzcd}
	\lif \sigma & \lif \sigma' \\
	\sigma & {\sigma'}
	\arrow["{\nu'}"', from=2-2, to=1-2]
	\arrow["{\mu_k}"', from=2-1, to=2-2]
	\arrow["\nu", from=2-1, to=1-1]
	\arrow["{\mu_k}", from=1-1, to=1-2]
\end{tikzcd}\]
\end{lemma}

\begin{proof}
Let $t'=\mu_k(t)$. From the mutation rule \eqref{mutation graduation} it's clear that 
  $ \mu_k(\lif \sigma)= \mu_k(\lif 0) \times (\sigma',0).$
  Hence it's sufficient to prove that $\mu_k(\lif 0)= \lif 0'$. Call $0^*= \mu_k(\lif 0)$. For $ i \in I \setminus \{k\}$ and $d \in D$, the identities 
  $$0^*_i= \nu_i= \nu'_i= \lif0'_i \quad \text{and} \quad 0^*_d= e_d=\lif 0'_d $$
  follows at once from \eqref{mutation graduation} and \eqref{mutation lifting configuration formula}. From the same formulas we compute that $$\begin{array}{rl}
     0^*_k= &\lif 0\strut^{\lif \spc B^+_{\bullet, k}} - \nu_{\bullet, k} \\[0.15em]
     = & \nu B^+_{\bullet, k} + (-\nu B_{\bullet, k})^+ - \nu_{\bullet,k}.
  \end{array}
  $$
Then the identity $0^*_k=\lif 0'_k$ follows from \eqref{eq 19}.
\end{proof}

\begin{coro}
\label{unicity mutation sequence lifting configuration}
For any $k_1, \dots k_r \in I_{uf}$ and $j_1, \dots j_s \in I_{uf}$ such that $\mu_{k_r}\circ \dots \circ  \mu_{k_1}(t)=\mu_{j_s}\circ \dots \circ \mu_{j_1}(t)$, then $\mu_{k_r} \circ \dots \circ \mu_{k_1}(\nu)=\mu_{j_s} \circ \dots \circ \mu_{j_1}(\nu)$. In particular, any $t' \in \Delta(t)$ has a canonical lifting configuration $(\LL,\nu')$ defined by the composition of \eqref{mutation lifting configuration formula} along any sequence of mutations from $t$ to $t'.$ Moreover, the map sending $t'$ to $\lif t'$ is a bijection between $\Delta(t)$ and $\Delta(\lif t).$ 
\end{coro}

\begin{proof}
    One can recover $\nu$ from $\lif 0$. We know that $\mu_{k_r} \circ \dots \circ \mu_{k_1}(\lif 0)=\mu_{j_s} \circ  \dots \circ \mu_{j_1}(\lif 0)$, then  Lemma \ref{lifting grad commutes mutation} implies that  $\mu_{k_r} \circ \dots \circ \mu_{k_1}(\nu)=\mu_{j_s} \circ \dots \circ \mu_{j_1}(\nu)$. Alternatively, the same result can be obtained in the same way as a direct consequence of Lemma \ref{mutation commutes lifting}. By Lemma \ref{mutation commutes lifting}, the map $\Delta(t) \longto \Delta (\lif t)$ sending $t'$ to $\lif t'$ is well defined. It is surjective since $t$ and $\lif t$ have the same set of mutable vertices. Moreover, it's clear from Definition \ref{monomial lifting defi} that if $\lif t' \lifap t'$, then $\lif t'$ completely determines $\nu'$ and $t'$. In particular, the map is injective. 
\end{proof}

Note that, with little abuse of notation, we can consider the $x_d$ as abstract independent variables over $\KK$. Any $t' \in \Delta$ gives an isomorphism 
\begin{equation}
    \label{lifting variable tensor}
    \begin{array}{r c l }
    \KK[x_d^{\pm 1}]_d & \longto & \CC[x_d^{\pm 1}]_d \otimes \KK\\[0.4em]
    \lif x'_i & \longmapsto & \begin{cases}
        x_d \otimes 1 & \text{if} \quad i=d \in D\\
        (x_d)_d^{\nu'_{\bullet, i}} \otimes x'_i & \text{if} \quad i \in I
    \end{cases}
    \end{array}
\end{equation}
where $\nu'$ is the lifting matrix on $t'$ determined by $\nu$. Because of Lemma \ref{mutation commutes lifting}, the above isomorphism doesn't depend on $t'$ but only on the chosen initial lifting matrix $\nu$ on $t$. So, whenever $\nu$ is clear, we identify the above two rings by the previous isomorphism. In particular, we can think about $\lif t$ as a seed of the fraction field of $\CC[x_d^{\pm 1}]_d \otimes \KK$. With this identification, it's immediate that $\Li(\lif t)= \CC[x_d^{\pm 1}]_d \otimes \Li(t).$

\begin{theo}
\label{monomial lifting AxT}
 We have equality $$ \clu(\lif t) = \clu(t) [x_d^{\pm 1}]_d.$$ Moreover, if $t$ is of maximal rank, then 
 $$ \uclu(\lif t) = \uclu(t) [x_d^{\pm 1}]_d.$$
\end{theo}

\begin{proof}
 By Corollary \ref{unicity mutation sequence lifting configuration}, for any $t^* \in \Delta(\lif t)$,  there exists a unique $t' \in \Delta(t)$, with a well defined lifting configuration $\nu'$, such that $t^*= \lif t'$. The equality involving $\clu(\lif t)$ is then obvious from the definition of $\lif x'$ and the fact that, for any $d \in D$, then $x_d^{\pm 1} \in \clu(\lif t)$ since $D \subseteq \, \, \lif I_{sf}$.

\bigskip

In the following, we think about $\lif t$ as a seed of  the fraction field of $\CC[x_d^{\pm 1}]_d \otimes \KK$. If $t$ is of maximal rank, so is $\lif t$. Then using Theorem \ref{upper cluster equals upper bound} and the convention that $\mu_0(t)=t$, we deduce that (intersections run over $I_{uf} \cup \{ 0 \}$): 
$$ \begin{array}{rl}
   \uclu(\lif t)= & \displaystyle \bigcap \biggl[ \CC[x_d^{\pm 1}]_d \otimes \Li(\mu_k(t))  \biggr] \\[1.5em]
   = &  \displaystyle \CC[x_d^{\pm 1}]_d \otimes \biggl[ \bigcap \Li(\mu_k(t)  \biggr]  \\[1.5em]
   = & \uclu(t) [x_d^{\pm 1}]_d
   \end{array}$$
 where central equality follows because $ \CC[x_d^{\pm 1}]_d \otimes -$ commutes with finite intersections.
\end{proof}

We expect the answer to the following question to be negative in general.

\begin{question}
    Does the equality between $\uclu(\lif t)$ and $\uclu(t) [x_d^{\pm 1}]_d$ hold without the assumption that $t$ is of maximal rank?
\end{question}

\begin{remark}
 In the special case of $t_{prin}$ (the seed obtained from $t$ by adding principal coefficients), the second statement of the previous theorem is a consequence of \cite[Proposition 8.27]{gross2018canonical}. We stress the fact that, the variables $(x_d)_d$ being algebraically independent over $\KK$, $\clu(t) [x_d^{\pm 1}]_d$ and $\uclu(t) [x_d^{\pm 1}]_d$ are Laurent polynomial rings over $\clu(t)$ and $\uclu(t)$ respectively.
\end{remark}
         
\begin{coro}
\label{deletion map surjective lifting}
 The deletion map $s : \uclu(\lif t) \longto \uclu(t)$, that is the morphism defined by $s(\lif x_i)=x_i$ for $i \in I $ and $s(\lif x_d)=1$ for $d \in D$, is surjective. The kernel is the ideal generated by $\lif x_d-1$, for $d \in D$.
\end{coro}

\begin{proof}
This is obvious from Theorem \ref{monomial lifting AxT}.
\end{proof}

Note that $\CC[x_d^{\pm 1}] \otimes \uclu(t)$ is a $\ZZ^D$ graded algebra. For $\lambda \in \ZZ^D$, the homogeneous component of degree $\lambda$ of this graduation is $(x_d)_d^\lambda \otimes \uclu(t)$. The following is an obvious but very useful reformulation of Proposition \ref{unicity monomial lifting}.

 \begin{coro}
         \label{cor:unicity mon lifting homogeneous}  Let  $\widetilde t$ be a seed of the fraction field of $\CC[x_d^{\pm 1}]_d \otimes \uclu(t)$ such that
         \begin{enumerate}
             \item $\widetilde t$ is a $D$-pointed seed extension of $t$ with $D \subset \widetilde I_{sf}$.
             \item For any $i \in I$, $s(\widetilde x_i)= x_i $ and $\widetilde x_i$ is $\ZZ^D$-homogeneous.
       \item For any $k \in I_{uf}$, $s\biggl( \mu_k(\widetilde x_k)\biggr)= \mu_k(x_k) $ and $\mu_k(\widetilde x_k)$ is $\ZZ^D$-homogeneous.
               \end{enumerate}
         For $i \in I$, let $\nu_{\bullet,i }= \deg(x_i).$ Then $\widetilde t = \lif t \lifa t.$
         \end{coro}
\begin{lemma}
Let  $\widetilde t$ be a $D$-seed extension of $t$ such that $D \subseteq \widetilde I_{sf}$. If $B^t : \ZZ^I \longto \ZZ^{I_{uf}}$ is surjective, there exist a lifting matrix $\nu \in \ZZ^{D \times I}$ such that $\lif B= \widetilde B$. 
\end{lemma}

\begin{proof}
Since $B^t$ is surjective, we can find $\nu$ such that $\widetilde B_{D, \bullet}= - \nu B.$
\end{proof}

To the author best knowledge, the next statement is new.

\begin{coro}
    Let  $\widetilde t$ be a $D$-seed extension of $t$ such that $D \subseteq \widetilde I_{sf}$.  If $B^t : \ZZ^I \longto \ZZ^{I_{uf}}$ is surjective, then the deletion map $s : \uclu(\widetilde t) \longto \uclu(t)$ is surjective and its kernel is the ideal generated by $\widetilde x_d -1$, for $d \in D.$
\end{coro}

\begin{proof}
   Note that the isomorphism class of an upper cluster algebra only depends on the $B$-matrix and on the vertices sets of one of its seeds, and so does the desired property of $s$. The statement follows from the previous lemma and Corollary \ref{deletion map surjective lifting}.
\end{proof}

If $\widetilde t$ is a seed extension of $t$, the deletion map $s : \uclu(\widetilde t) \longto \uclu(t)$ is not always surjective.  An example is given in \cite{matherne2015computing}[Section 7.1, Corollary 7.1.2]. The deletion map is always surjective at the level of cluster algebras.

\begin{remark}
\label{rem:analogy monodromy Qin}
    Let $t$ be a seed and set $D:= \{ i' : i \in I \}$. We identify $\ZZ^D$ with $\ZZ^I$ in the obvious way. If we consider $\nu=  \Id \in \ZZ^{I \times I}$, then 
    $$ \lif B= \bmat B \\ -B \emat.$$
   Consider the degree transformation $\psi$ defined in \cite[Definition 3.3.1]{qin2022bases}. By \cite[Remark 3.3.4]{qin2022bases}, for any $k \in I_{uf}$,
   $$e_k -\psi_{t,\mu_k(t)} \psi_{\mu_k(t), t}(e_k) = \lif B_{D \times k}.$$
    In particular, the monomial lifting for this special $\nu$ encodes the non-trivial monodromy of the maps $\psi$ after one mutation. See \cite[Remark 3.3.4]{qin2022bases} for more details. It turns out that, under some hypothesis, the lifting matrix $\nu$ has an interpretation in terms of the Cox ring of a certain partial compactification of the cluster manifold. This is part of an ongoing work.

\bigskip
    
    The results of this section should also work, with the same proofs, for cluster and upper cluster algebras defined over $\ZZ$ instead then over $\CC$.
\end{remark} 

\begin{remark}
    \label{rem:quasi eq seeds}
  If $\nu,\mu \in \ZZ^{D \times I}$ are distinct lifting matrices and $\lif t_\nu \lifa t$, $\lif t_\mu \xlongleftarrow{\mu}$ are the respective monomial liftings, clearly $\lif t_\nu$ and $\lif t_\mu$ are distinct. Nevertheless, they are quasi-equivalent in the sense of \cite{fraser2016quasi} by  \cite[Corollary 4.5]{fraser2016quasi}.
\end{remark}

\subsection{Pole filtration and lifting graduation}
\label{pole filtration section}

It turns out that the graduation induced by $\lif 0$ on $\uclu(\lif t^D)$ has a nice geometric interpretation. To explain this fact, we switch for a moment to a more general setting. This allows us to apply the same ideas in the context of schemes which are suitable for lifting, which are introduced in Section \ref{minimal monomial lifting section}.

In this section, $D$ is a finite set and $T$ is a torus of rank $|D|$. We denote by $X(T)$ the character group of $T$.

\begin{definition}
\label{def: almost polynomial ring}
    Let $A$ be a $\CC$-algebra. An almost polynomial ring over $A$ is a quadruple $(R, \val , X, \phi^*)$ where

    \begin{enumerate}
        \item $R$ is a normal (non necessarily noetherian) $\CC$-algebra which is a domain.
        \item $\val$ is a collection of discrete valuations  $\val_d : R \longto \ZZ_{\geq 0} \cup \{ \infty \}$, indexed by $d \in D$.
        \item $X \subseteq R^D$ is a collection of elements such that $\val_{d_1}(X_{d_2})= \delta_{d_1,d_2}$ and $$R= \bigcap_{d \in D} \val_d\inv( \ZZ_{\geq 0 } \cup \{ \infty \}) \bigcap R_{\prod X_d}.$$
        \item $\phi^*: R_{\prod X_d} \longto \CC[T] \otimes A$ is an isomorphism such that $\phi^*(X_d)=x_d \otimes 1$, where the collection $(x_d)_d$ is a basis of $X(T).$
    \end{enumerate}
\end{definition}

When $\val, X$ and $\phi^*$ are clear, we just write that $R$ is an almost polynomial ring over $A$. If an almost polynomial ring is called $R$, we implicitly assume that the rest of the data is denoted by $\val, X, \phi^*$ and that $R$ is almost polynomial over $A$. Moreover, $R_{\prod{X_d}}$ is canonically identified with $\CC[T] \otimes A$ via $\phi^*$. In particular, we write $X_d= x_d \otimes 1$. 

\begin{example}
    \label{ex: upper cluster almost polynomial}
    Let $t$ be a seed of maximal rank, $\nu$ a $D$-lifting matrix and $\lif t \lifa t$. Then we have an almost polynomial ring over $\uclu(t)$: 
    $$(\uclu(\lif t^D), \cval, \lif x_D, \phi^*).$$ Recall that $\lif x_d=x_d$ and  $\uclu(\lif t^D)_{\prod_d  x_d} = \uclu(\lif t)$ because of Lemma \ref{properties freezing}.
    Moreover, the map $\phi^*: \uclu(\lif t) \longto \CC[ x_d^{\pm 1}]_d \otimes \uclu(t)$ defined by formula \eqref{lifting variable tensor} is an isomorphism because of Theorem \ref{monomial lifting AxT}. The third axiom of Definition \ref{def: almost polynomial ring} holds by statement 8 of Lemma \ref{properties freezing}. The other axioms are trivial to check.
    
    We refer to this as the \textit{standard quasi polynomial structure} of $\uclu(\lif t^D)$. Of course, we consider $(x_d)_d$ as base of characters of the torus $\Spec(\CC[x_d^{\pm 1}])_d$.
\end{example}

Form now on, we fix an almost polynomial ring $R$ over $A$. Let $s: R \longto A$ be the morphism induced by the map  $\CC[T] \otimes A \longto A$ sending $p \otimes a$ to $p(1)a$. Note that, if $R=\uclu(\lif t^D)$, considered with its standard quasi-polynomial structure, then $s$ is the deletion map of Corollary \ref{deletion map surjective lifting}.
We consider $\ZZ^D$ ordered by $\lambda \leq \mu$ if the inequality holds component-wise.

\begin{definition}[Pole filtration]
\label{pole filtration}
The \textit{pole filtration} on the algebra $A$ is the $(\ZZ^D, \leq)$ filtration defined by
$$A_\lambda:= \{ a \in A \, : \, \forall \, d \in D , \val_d(1 \otimes a) \geq - \lambda(d) \}.$$
If $H$ is an abelian group and $A= \bigoplus_{h \in H} A_h$ is a graduation, we set $A_{\lambda,h}= A_\lambda \cap A_h$. We say that the pole filtration is $H$-graded if, for any $\lambda \in \ZZ^d$, $ A_\lambda = \bigoplus_h A_{\lambda, h}.$ 
\end{definition}
It's easy to verify  that, if $\mu \leq \lambda$, then $A_\mu \subseteq A_\lambda$, moreover for any $\lambda, \mu$ we have $A_\lambda \cdot A_\mu \subseteq A_{\lambda+\mu}$ and finally $A_\lambda + A_\mu \subseteq A_{\max\{ \lambda, \mu \}}$. We express this last inclusion saying that the pole filtration is a \textit{tropical filtration}. Note that from the definition of almost polynomial ring, we have that for any $\lambda \in \ZZ^D$

$$A_\lambda= \{ a\in A \, : \, X^\lambda(1 \otimes a) \in R\}.$$

We introduce a notation. Let $G$ and $H$ be abelian groups and $B=  \displaystyle \bigoplus_{(g,h) \in G \times H} B_{g,h}$ be a $G \times H$-graded $\CC$-algebra. The $G$-graduation induced by the $G \times H$ one, on $B$, is defined as follows: 
$$B= \bigoplus_{g \in G} B_{g,\bullet} \quad \text{where, for} \,\, g \in G \quad B_{g, \bullet}= \bigoplus_{h \in H} B_{g,h}.$$ 
Similarly, we have an induced $H$-graduation, for which we use an analogue notation.
\begin{definition}[lifting graduation]
\label{lifting graduation defi}
A $\ZZ^D$-graduation $R= \bigoplus_{\lambda \in \ZZ^D} R_\lambda$ is called \textit{lifting graduation}, of the pole filtration, if the following conditions hold
\begin{enumerate}
    \item For any $d \in D$, $X_d \in R_{e_d}$.
    \item The isomorphism $\phi^*$ identifies $1 \otimes A $ with the $0$-degree component of $R_{\prod X_D}$.
    \item For any $\lambda \in \ZZ^D$, the restriction of $s$ to $R_\lambda$ is an isomorphism with $A_\lambda.$
\end{enumerate}
    If the pole filtration is $H$-graded, we say that a $\ZZ^D \times H$ graduation $R = \bigoplus_{\lambda ,h} R_{\lambda,h}$ is an $H$-graded lifting graduation if \begin{enumerate}
        \item $R= \bigoplus_{\lambda} R_{\lambda, \bullet}$ is a lifting graduation.
        \item For any $(\lambda, h) \in \ZZ^D \times H$, the restriction of $s$ to $R_{\lambda, h}$ is an isomorphism with $A_{\lambda, h}.$
    \end{enumerate}
\end{definition}

   It often happens, in representation theory, that we study the components of the pole filtration on some algebra, because they are the shadows of the homogeneous components of the lifting graduation, which is the real object of interest. For example, some classical expressions for the generalised Littlewood-Richardson coefficients, such as  \cite[Proposition 1.4]{zelevinsky2001littlewood}, can be interpreted in this setting. We explain this in section \ref{branching scheme subsection}.

\bigskip

    The following lemma is obvious but useful.

    \begin{lemma}
        \label{lem: H graded pole filt equivariance}
A $\ZZ^D \times H$ graduation $R = \bigoplus_{\lambda ,h} R_{\lambda,h}$ is an $H$-graded lifting graduation if and only if\begin{enumerate}
        \item $R= \bigoplus_{\lambda} R_{\lambda, \bullet}$ is a lifting graduation.
        \item The map $s$ is $H$-equivariant.
    \end{enumerate}
        
    \end{lemma}

 Recall that if $\widetilde T$ is a torus and $\widetilde R$ is a $\CC$-algebra, there is a bijection between left actions of $ \widetilde T$ on $\Spec( \widetilde R)$ and (multiplicative) graduations $\widetilde R= \bigoplus_{\lambda \in X( \widetilde T)} \widetilde R_\lambda$. Explicitly, if $\alpha :  \widetilde T \times \Spec(\widetilde R) \longto \Spec(\widetilde R)$ is an action, then $\widetilde R_\lambda$ is the set of $r \in \widetilde R$ such that $\alpha^*( r)= \lambda \otimes r$. 
 We adopt the convention that, for any $\CC$-scheme $Z$, $\widetilde T \times Z$ is canonically acted on by $\widetilde T$ by multiplication on the $ \widetilde T$ component. In particular, when $Z= \Spec( \widetilde A)$, for a $\CC$-algebra $\widetilde A $, then $( \CC[\widetilde T] \otimes \widetilde A)_\lambda=  \lambda \otimes \widetilde A $.
 When $\widetilde T$ has a fixed base of characters $\varpi \subseteq X( \widetilde T)^J$, then we identify $\ZZ^J$ with $X( \widetilde T) $ through the map $n \longmapsto \varpi^n$.
 
 In our context, the torus $T$ has a fixed base of characters determined by $R$, that is the base $(x_d)_d.$
 
\begin{prop}
\label{lifting graduation torus action}
    Let  $R=\bigoplus_{\lambda \in \ZZ^D} R_\lambda$  be a graduation and consider the associated $T$ action on $\Spec(R)$. The graduation on $R$ is a lifting graduation if and only if the morphism $\phi: T \times \Spec(A) \longto \Spec(R)$, induced by $\phi^*$, is $T$-equivariant. In particular, if the lifting graduation exists, it is unique.
\end{prop}

\begin{proof}

We identify $R$ with a subring of $\CC[T] \otimes A$, using $\phi^*$. Suppose that $R=\bigoplus_{\lambda \in \ZZ^D} R_\lambda$ is a lifting graduation. We need to prove that $R_\lambda \subseteq x^\lambda \otimes A $. Let $r \in R_\lambda$, by condition 3 in Definition \ref{lifting graduation defi}, $s(r) \in A_\lambda$. Recall that $X_d= x_d \otimes 1$. Since $R= \bigcap_{d \in D} \val_d( \ZZ_{\geq 0 } \cup \{ \infty \}) \bigcap R_{\prod X_d}$ and $\val_{d_1}(X_{d_2})= \delta_{d_1,d_2}$, we have that $ x^\lambda \otimes s(r) \in R$. Moreover, since $X_d \in R_{e_d}$ and $1 \otimes s(r) $ is of degree zero, then $x^\lambda \otimes s(r) \in R_\lambda$. But the restriction of $s$ to $R_\lambda$ is injective. Hence $r=x^\lambda \otimes s(r)$.

\bigskip

Conversely, suppose that $\phi$ is equivariant. Since $X_d=x_d \otimes 1$, we clearly have that $X_d \in R_{e_d}$. Similarly, condition 2 of Definition \ref{lifting graduation defi} is trivial to check. We want to verify condition 3.  
Take $r \in R_\lambda$, then $r \in x^\lambda \otimes  A $, in particular $r= x^\lambda \otimes s(r)$. Again, since
$$R= \bigcap_{d \in D} \val_d\inv( \ZZ_{\geq 0 } \cup \{ \infty \}) \bigcap R_{\prod X_D}$$
and $\val_{d_1}(X_{d_2})= \delta_{d_1,d_2}$, we deduce that $s(r) \in A_\lambda$. Moreover, if $s(r)=0$, then clearly $r=0.$ This proves that the restriction of $s$ to $R_\lambda$ is injective with image in $A_\lambda$. For the surjectivity, notice that if $a \in A_\lambda$, then $x^\lambda \otimes a  \in R_\lambda$ and $s( x^\lambda \otimes a )=a$. 

For the unicity part, it's sufficient to prove that if a torus action as in the statement of the proposition exists, then it is unique. But this is clear since $R$ is a domain, hence $\Spec(R_{\prod X_D})$ is schematically dense in $\Spec(R)$. 
\end{proof}

The following corollary is a useful reformulation of the above proposition.

\begin{coro}
    \label{homogeneous elements lifting graduation}
    A graduation $R= \bigoplus_{\lambda \in \ZZ^D} R_\lambda$ is the lifting graduation if and only if, for any $r \in R_\lambda$, then $\phi^*(r)= x^\lambda \otimes s(r)$.
\end{coro}

\begin{prop}
\label{ker pi lifting graduation}
    If the lifting graduation $R= \bigoplus_{\lambda \in \ZZ^D} R_\lambda$ exists, then $\ker(s)=( X_d-1 \, : \, d \in D ).$
\end{prop}

\begin{proof}
   Let $J=( X_d-1 \, : \, d \in D )$. A standard calculation allows see that if $\lambda \geq 0$, then $X^\lambda-1$ is in $J$.

   We claim that, for any $r \in R$ such that $s(r) \neq 0$ and for any  sufficiently large $\nu$, $r - x^\nu \otimes s(r)  $ is in $J$. Let's see that the claim allows to conclude. Indeed, take $y \in \ker(s)$ and any $r_\lambda \in R_\lambda$ such that $s(r_\lambda) \neq 0$. In particular, $r_\lambda=x^\lambda \otimes s(r_\lambda)$. Then for any $\nu$  we can write $$ y= (y+ r_\lambda) - r_\lambda= (y + r_\lambda - x^\nu \otimes s(r_\lambda) ) + r_\lambda ( x ^{\nu -\lambda} \otimes 1- 1 \otimes 1).$$
  Observe that $s(y+r_\lambda)= s(r_\lambda) \neq 0$. In particular, if $\nu$ is sufficiently large, the claim implies that $(y + r_\lambda - x^\nu \otimes s(r_\lambda) )$ is in $J$, but if $\nu \geq \lambda$, we also have that $ r_\lambda ( x ^{\nu -\lambda} \otimes 1- 1 \otimes 1)= r_\lambda(X^{\nu - \lambda}-1)$ is in $J$.

\bigskip

   Let's prove the claim. Write $r= \sum r_\mu$, where $r_\mu \in R_\mu $ and the sums runs over a finite set $M$, such that $r_\mu \neq 0$ if $\mu \in M$. 
   By the definition of lifting graduation we have that, for any $d \in D$, $\val_d(1 \otimes s(r_\mu) ) \geq - \mu(d)$. Then, if $\nu \geq \max\{\mu \, : \, \mu \in M\} $, we have that  for any $\mu \in M$, $x^\nu \otimes s(r_\mu) \in R_\nu$, hence $x^\nu \otimes s(r)  \in R_\nu$.  Then, for any such $\nu$,
   $$ r - x^\nu \otimes s(r) = \sum_\mu (x^\mu \otimes s(r_\mu))
   ( 1 \otimes 1 -  x^{\nu - \mu} \otimes 1 )= \sum r_\mu(1 - X^{\nu -\mu}).$$
  Since for any $\mu$ in the above sum $\nu - \mu \geq 0$, this proves the claim.
\end{proof}

\begin{prop}
    \label{lifting graduation cluster}
    In the notation of Example \ref{ex: upper cluster almost polynomial}, the $\ZZ^D$ graduation on $\uclu(\lif t^D)$ induced  by $\lif 0$ is the lifting graduation. If $t$ is a graded seed, then the pole filtration is $H$-graded and the graduation induced by $\lif \sigma$, on $\uclu(\lif t^D)$, is the $H$-graded lifting graduation.
\end{prop}

\begin{proof}
For the first statement, we want to apply Corollary \ref{homogeneous elements lifting graduation}. Recall that $\lif 0$ defines a graduation on $\Li(\lif t^D)$, for which the cluster monomials form a basis of homogeneous elements. Moreover, $\uclu(\lif t^D) $ is a graded subring of $\Li(\lif t^D)$. In particular, it's sufficient to prove that for any $i \in \lif \kern-0.15em I$, $\phi^*(\lif x_i)= \hx_D \strut^{\lif \spc \, 0_i} \otimes s(\lif x_i)$. But this is obvious from the definitions. \\
For the second statement, notice that the same argument as before proves that $s$ is $H$-equivariant. Using the first statement of the proposition, we have that for any $\lambda \in \ZZ^D$, $s: \uclu(\lif t^D)_{\lambda, \bullet} \longto \uclu(t)_\lambda$ is an isomorphism. It follows at once that the pole filtration is $H$-graded. Then the second statement follows from Lemma \ref{lem: H graded pole filt equivariance}.
\end{proof}

\begin{coro}
    \label{cor:del map surjective highly freezing}
     The deletion map $s : \uclu(\lif t^D) \longto \uclu(t)$ is surjective with kernel the ideal generated by $\lif x_d-1$, for $d \in D$.
\end{coro}

\begin{proof}
    Apply Propositions \ref{lifting graduation cluster} and \ref{ker pi lifting graduation}.
\end{proof}


\begin{coro}
    \label{cor:quotients of lifting}
    Suppose that $D=D_1 \sqcup D_2$. Denote by $\nu_1= \nu_{D_1, \bullet}$, $\lif t \lifa t$ and $\widetilde{\lif t} \xlongleftarrow{\nu_1} t$. Then $\uclu(\widetilde{\lif t}^{D_1})$ is the quotient of $\uclu(\lif t^D)$ by the ideal $(\lif x_d - 1  \, : d \in D_2).$
\end{coro}

\begin{proof}
    By Lemma \ref{lift in steps}, we have that if $\nu_2 = \nu_{D_2, \bullet}$, then $\lif t^{D_1} \xleftarrow{(\nu_2,0)} \widetilde{\lif t}^{D_1}.$ The statement follows from Corollary \ref{cor:del map surjective highly freezing}.
\end{proof}

\section{The minimal monomial lifting}
\label{minimal monomial lifting section}

We  use the following notation. If $Z$ is a $\CC$-scheme and $f \in \Orb_Z(Z)$, we denote with $V(f)$ the zero locus of $f$ in $Z$, which is a closed subspace of the topological space underlying $Z$. We write $Z^{(1)}$ for the set of codimension one points of $Z$. If $Z$ is normal and integral, for $p \in Z^{(1)}$, we denote with $\val_p : \CC(Z)  \longto \ZZ \cup \{\infty\}$ the valuation induced, on the fraction field of $Z$, by the discrete valuation ring $\Orb_{Z,p}$. If $E \subseteq Z$ is closed, irreducible and of codimension one, it has a unique generic point $p$. Then we set $\val_E= \val_p$ and we say that $\val_p$ is the valuation associated to $E$. We say that $f,g \in \Orb_Z(Z)$ are \textit{coprime} on $Z$ if $ V(f) \cap V(g) \cap Z^{(1)}= \emptyset$. Finally, we say that an open subset $\Omega \subseteq Z$ is \textit{big} if $\Omega^{(1)}= Z^{(1)}.$

\bigskip

From now on, $D$ is a finite set and $T$ is torus of rank $|D|.$

\begin{definition}
\label{def: suitable lifting}A \textit{suitable for $D$-lifting} scheme is a triple $(\lX, \phi, X)$ such that:

\begin{enumerate}
    \item $\lX$ is  a noetherian, normal and integral $\CC$-scheme.
    \item $X \subseteq \Orb_\lX(\lX)^D$ is a collection of global sections such that, for each $d \in D$, $V(X_d)$ is irreducible and if $\val_d$ is the associated valuation, then $\val_{d_1}(X_{d_2})= \delta_{d_1,d_2}.$
    \item $\phi: T \times Y \longto \lX \setminus \cup_d V(X_d)$ is an open embedding such that the image is big and $\phi^*(X_d)= (x_d \otimes 1)$, where the collection of $x_d$ is a base of $X(T)$. Here $Y$ is an irreducible $\CC$-scheme.
\end{enumerate}
    
\end{definition}

When the data of  $\phi$ and $X$ is clear, we just say that $\lX$ is suitable for $D$-lifting. If a scheme suitable for $D$-lifting is called $\lX$, we implicitly assume that the rest of the data defining it is denoted by $\phi$, $Y$ and $X$ and that $\val_d$ is the valuation associated to $V(X_d).$

We use the following convention. If $\lX$ is suitable for $D$-lifting, then we identify $T \times Y$ with an open subset of $\lX$ using $\phi$. Hence, we write $\CC(\lX)= \CC(T \times Y)$ and we identify $\CC(Y)$ with $1 \otimes \CC(Y) \subseteq \CC(\lX)$. Then, $\CC(\lX)$ is canonically a $D$- field extension of $\CC(Y)$: its $D$-uple is $X$. We stress that, with these identifications, if $t$ is a seed of $\CC(Y)$ such that $\uclu(t) = \Orb_Y(Y)$, then $\hx \subseteq (\CC(\lX)^*)^{I \sqcup D}$ is defined by $$ \hx_i= \begin{cases}
    1 \otimes x_i & \text{if} \quad  i \in I\\
    X_d & \text{if} \quad i=d \in D
\end{cases}$$

We can construct affine schemes which are suitable for lifting as follows
\begin{example}
\label{example: affine suitable for lifting}
    Let $R$ be a finite type $\CC$-algebra which is a normal domain. Consider a collection of elements $X \subseteq R^D$ such that the ideals $(X_d)$ are prime and pairwise different.
    Suppose that there is an isomorphism $\psi: R_{\prod_d X_d} \longto \CC[T] \otimes A$ such that $\psi(X_d) = x_d \otimes 1$, where the $x_d$ form a base of $X(T)$. Here $A$ is a $\CC$-algebra. Then $(\Spec(R), \psi^*, X)$ is suitable for $D$-lifting.
\end{example}

We now state the main theorem of this section.

\begin{theo}
    \label{minimal monomial lifting theo}
    Let $\lX$ be a suitable for $D$-lifting scheme  and $t$ be a maximal rank seed  of $\CC(Y)$ such that $\uclu(t)= \Orb_Y (Y).$ Suppose that, for any non-equal $i,j \in I_{uf}$, $x_i$ and $x_j$ are coprime on $Y$ and that, for any $k \in I_{uf}$, $x_k$ and $x_k'= \mu_k(x_k)$ are coprime on $Y$.
   
Consider $(\CC(\lX), \nu)$: the $D$-lifting configuration  on $t$ defined by $$\nu_{d,i}:= - \val_d( 1 \otimes x_i )$$
Then
    \begin{enumerate}
        \item  $ \uclu(\lif t) = \Orb_{\lX}(T \times Y) $.
        \item  $\uclu(\lif t^D) \subseteq \Orb_\lX(X).$
    \end{enumerate}

\end{theo}

\begin{definition}
    \label{minimal monomial lifting defi}
    In the setting of Theorem \ref{minimal monomial lifting theo}, $\nu$ and $\lif t^D$ are called respectively the \textit{minimal lifting matrix} and the \textit{minimal monomial lifting} of $t$, associated to $\lX.$
\end{definition}

The proof is based on the well known algebraic Hartogs' lemma

\begin{lemma}[Hartogs']
\label{Hartogs' lemma}
    Let $Z$ be a locally noetherian normal integral scheme, then $\Orb_Z(Z) = \cap_{p \in Z^{(1)}} \Orb_{Z,p}$, where the intersection takes place in $\CC(Z)$.
\end{lemma}

Note that $\Orb_{Z,p}$ consists of the rational functions with non-negative valuation at $p$.

\bigskip

For the proof, we need the following, which is usually called starfish lemma. Notice that the proof is exactly the same of the affine case \cite[Proposition 3.6]{fomin2016tensor}.

\begin{lemma}[starfish lemma]
\label{starfish lemma}
Let $Z$ be a locally noetherian normal integral  scheme over $\CC$ and  $t$ be a seed of $\CC(Z)$ such that

\begin{enumerate}
    \item For any $i \in I$, $x_i \in \Orb_Z(Z)$ and if $i \in I_{sf}$, $x_i \in \Orb_Z(Z)^*.$
    \item For any $k \in I_{uf}$, $x_k':= \mu_k(x_k) \in \Orb_Z(Z)$.
    \item For any non equal $i,j \in I_{uf}$, $x_i$ and $x_j$ are coprime on $Z$ and for any $i \in I_{uf}$, $x_i$ and $x_i'$ are coprime on $Z$.
\end{enumerate}
Then $\uclu(t) \subseteq \Orb_Z(Z).$
\end{lemma}

\begin{proof}
    Let $f \in \uclu(t)$ and $p \in Z^{(1)}$, since $f \in \Li(t)$ and for any $i \in I_{sf}$ $x_i \in \Orb_Z(Z)^*$, if there is no $k \in I_{uf}$ such that $x_k(p)=0$, then $f \in \Orb_{Z,p}$. If it's not the case, take $k \in I_{uf}$ such that $x_k(p)= 0$. Then, for any $j \in I_{uf} \setminus \{k \}$, $x_j(p) \neq 0$ and $x_k'(p) \neq 0$. But $f \in \Li(\mu_k(t))$, hence we deduce that $f \in \Orb_{Z,p}$. We conclude using Lemma \ref{Hartogs' lemma}. Actually, this proves the stronger statement that the upper bound $\upp(t) $ is included in $\Orb_Z(Z).$
\end{proof}

\begin{proof}[Proof of Theorem \ref{minimal monomial lifting theo}]
For the first statement, because of  Theorem \ref{monomial lifting AxT} and the fact that $\phi$ is an open embedding, we just have to prove that the natural map $\CC[T] \otimes \Orb_Y(Y)  \longto \Orb_{T \times Y}(T \times Y)$ is an isomorphism.

First, if $e \in T(\CC)$ is the identity, identifying $Y$ with $\{ e \} \times Y$ gives an immersion (neither closed nor open) $\iota : Y \longto \lX$. Since $\iota $ is a monomorphism and $\lX$ is quasi-separated, being noetherian, it follows that $Y$ is quasi-separated. Moreover, $\lX$ is locally noetherian and $\iota$ is an immersion, hence $\iota$ is quasi-compact. In particular, $Y$ is quasi-compact since $\lX$ is. Then, we can cover $Y$ with finitely many affine open subsets $Y_i$. Since $Y$ is quasi-separated, we can cover each intersection $ Y_i \cap Y_j$ with finitely many affine open subsets $Y_{i,j,k}$. Consider the natural exact sequence 
$$\begin{tikzcd}
	 0 & {\Orb_Y(Y)} & { \displaystyle \prod_i \Orb_Y(Y_i)} & {\displaystyle \prod_{i,j,k} \Orb_Y(Y_{i,j,k})}
	\arrow[from=1-1, to=1-2]
	\arrow[from=1-2, to=1-3]
	\arrow[from=1-3, to=1-4]
\end{tikzcd} $$

Note that the two products in the exact sequence are finite. Since $\CC[T] \otimes (-)$ is exact, commutes with finite products and the $Y_i, Y_{i,j,k}$ are affine, tensoring with $ \CC[T]$ yields the exact sequence

\[\begin{tikzcd}
	0   & { \CC[T] \otimes \Orb_Y(Y) } 
 & {\prod_{i} \Orb_{T \times Y}(T \times Y_{i} )} & {\prod_{i,j,k} \Orb_{ T \times Y }(T \times Y_{i,j,k} )}
	\arrow[from=1-1, to=1-2]
	\arrow[from=1-2, to=1-3]
	\arrow[from=1-3, to=1-4]
\end{tikzcd}\]

But since the $T \times Y_i $ cover $T \times Y $, and the $T \times Y_{i,j,k}  $ cover $(T \times Y_i ) \cap (T \times Y_j )$, we deduce from the above exact sequence that $\CC[T] \otimes \Orb_Y(Y)  \simeq \Orb_{T \times Y}(T \times Y)$.

\bigskip

For the second statement, we want to apply the starfish lemma. Let $p \in \lX^{(1)}$. Since $\lX$ is suitable for lifting, $p \in (T \times Y)^{(1)}$ or $\overline{p}= V(X_d)$ for a certain $d \in D$. In the last case, $d$ is unique. Note that for any $m \in \ZZ^D$, then $$\val_p( \hx_D \strut^{m})= \begin{cases}
    0 & \text{if} \quad p \in (T \times Y)^{(1)}\\
    m_d &\text{if} \quad \overline{p}= V(X_d).
\end{cases}$$
Then we easily compute that, for any $i \in I$,
\begin{equation}
\label{eq 21}
    \val_p( \lif x_i)= \begin{cases}
    \val_p( 1 \otimes  x_i) & \text{if} \quad p \in (T \times Y)^{(1)}\\
    0 &\text{if} \quad \overline{p}= V(X_d).
\end{cases}
\end{equation}

It follows from the Hartogs' lemma that  $\lif x_i \in \Orb_\lX(\lX)$. Moreover, if $i \in \lif I_{sf}=I_{sf}$ then $\lif x_i \in \Orb_\lX(\lX)^*$ and clearly $\lif x_d= X_d \in \Orb_\lX(\lX).$
Now fix $i, j \in I_{uf}= \lif I^D_{uf}$ with $i \neq j$, we want to prove that $\lif x_i$ and $\lif x_j$ are coprime on $X$. By the previous calculation, it's sufficient to prove that they're coprime on $T \times Y$, which is equivalent to the comprimality of $1 \otimes x_i$ and $1 \otimes x_j $ because, for any $d \in D$, $ X_d \in \Orb_{\lX}(T \times Y)^*$. Take a point $p \in T \times Y$ such that $1 \otimes x_i(p)=0= 1 \otimes x_j (p). $ Take an affine open subset $U \simeq \Spec A$,
of $Y$, such that $ p \in T \times  U $. Note that  $A$ is a normal domain. Then we can identify $p$ with a prime ideal of $\CC[T] \otimes A $ which contains $1 \otimes x_i $ and $1 \otimes x_j$. 
If we look at the natural map $A \longto \CC[T] \otimes  A $ that sends $a$ to $1 \otimes a $, we have that $p^c$ is a prime ideal of $A$ that contains $x_i$ and $x_j$ ($(-)^c$ denotes the contraction of ideals). Since $x_i$ and $x_j$ are coprime on $Y$, there exists a non-zero prime $q $ of $A$ such that $q$ is strictly contained in $p^c$. 
Since $\CC$ is algebraically closed and $\CC[T]$ is a finitely generated $\CC$-algebra which is a domain, then for any $\CC$-algebra $B$ which is a domain, $\CC[T] \otimes B $ is again a domain. In particular, tensoring with $\CC[T]$ sends prime ideals of $A$ to prime ideals of $\CC[T] \otimes A $ and it preserves strict inclusion because of exactness. It follows that $ 0 \subsetneq  \CC[T] \otimes q \subsetneq \CC[T] \otimes p^c \subseteq p$. Hence $p$ has height at least 2, so $1 \otimes x_i $ and $1 \otimes x_j  $ are comprime on $Y$.

Finally, observe that for any $k \in I_{uf}$, $x_k^*:= \mu_k(\lif x_k)$ is regular on $T \times Y $ because of Lemma \ref{mutation commutes lifting}. Moreover, from the exchange relation \eqref{exchange relation} and the fact that, for any $d \in D$, $\val_d(\lif x_k)=0$, it follows that $\val_d(x_k^*) \geq 0$. Hence $x_k^* \in \Orb_X(X)$ by Hartogs' Lemma \ref{Hartogs' lemma}. Finally, call $t'= \mu_k(t)$ and $\nu'=\mu_k(\nu)$. Since $x_k^*=(1 \otimes x_k')\hx_D \strut^{\nu'_{\bullet, k}}$ by Lemma \ref{mutation commutes lifting}, the previous argument also proves that $x_k^*$ and $\lif x_k$ are coprime on $X$. Then, $\uclu(\lif t^D) \subseteq \Orb_\lX(\lX)$ follows from Lemma \ref{starfish lemma}.
\end{proof}

\begin{coro}[Of the proof]
       \label{cor: suitable quasi-polynomial}
    If $\lX$ is suitable for $D$-lifting, then $(\Orb_\lX(\lX), \val, X, \phi^*)$ is an almost polynomial ring over $\Orb_Y(Y)$.
    \end{coro}

    The following lemma allows to produce a nice class of suitable for $D$-lifting schemes.

\begin{lemma}
    \label{lem:very good for lifting} Let $(\lX, \phi, X)$ be a triple that satisfies the first two conditions of Definition \ref{def: suitable lifting} and such that  $\phi: T \times Y \longto \lX \setminus \cup_d V(X_d)$ is an open embedding with big image. 
    If moreover $Y$ is an irreducible $\CC$-scheme such that $\Orb_Y(Y)^*= \CC^*$, and there exist a point $y \in Y$ such that, for any $d \in D$, $\phi^*(X_d)(e,y)=1$, then $(\lX, \phi, X)$ is  suitable for $D$-lifting.
\end{lemma}

\begin{proof}
    We just need to prove that, for any $d \in D$, $\phi^*(X_d)= x_d \otimes 1$ and the collection of $x_d$ is a base of $X(T)$. 
    
    The argument at the beginning of the proof of Theorem \ref{minimal monomial lifting theo} implies that $\Orb_\lX(T \times Y)= \CC[T] \otimes \Orb_Y(Y)$. Since $\Orb_Y(Y)$ is a domain and $\Orb_Y(Y)^*= \CC^*$, it follows that $$\Orb_\lX(T \times Y)^* = \{ \lambda \otimes c \, : \, \lambda \in X(T), c \in \CC^*\}.$$
   But $\phi^*(X_d) \in \Orb_\lX(T \times Y)^*$ by Hartogs' lemma, moreover $\phi^*(X_d)(e,y)=1$. Hence there exist $x_d \in X(T)$ such that $\phi^*(X_d)=x_d \otimes 1$. Let $\lambda \in X(T)^D$ be a base of $X(T)$ and $\Divv \in \ZZ^{D \times D}$ be defined by $$ \Divv_{d_1,d_2}= \val_{d_1}(\lambda_{d_2} \otimes 1).$$
     Since $\lambda$ is a base of $X(T)$, there exists $A \in \ZZ^{D \times D}$ such that $\lambda^A= (x_d)_d$. Hence $\Divv \cdot A = \Id$. This implies that $(x_d)_d$ is a base of $X(T).$
    
    \end{proof}
From now on we suppose to be in the setting of Theorem \ref{minimal monomial lifting theo}. In particular, $\lX$, $t$ are fixed, $\nu$, $\lif t^D$ denote respectively the minimal lifting matrix and the minimal monomial lifting of $t$ associated to $\lX$.

\bigskip

The hypothesis of coprimality of Theorem \ref{minimal monomial lifting theo} often holds. For example:

\begin{lemma}
\label{factorial upper, coprimality}
    If $Y$ is affine and $\Orb_Y(Y)$ is factorial then, for any $i,j \in I_{uf}$ with $i \neq j$, $x_i$ and $x_j$ are coprime on $Y$. Moreover, for any $k \in I_{uf}$, $x_k$ and $x_k':= \mu_k(x_k)$ are coprime on $Y$.
\end{lemma}

\begin{proof}
    No mutable vertex of $t$ is completely disconnected because $t$ is of maximal rank. Then, by Theorem \ref{cluster variables irred}, any non semi-frozen cluster variable $z$ is irreducible in $\Orb_Y(Y)$, hence $(z)$ is prime. In particular, two cluster variables $z$ and $z'$ are coprime if and only if $(z) \neq (z')$. The lemma follows from the second part of Theorem \ref{cluster variables irred}.
\end{proof}

To the author's best knowledge, no general coprimality criteria for cluster variables in an upper cluster algebra is known. In \cite{cao2022valuation}, the authors develop some usefull criteria to study factoriality in upper cluster algebras of maximal rank.

\bigskip

The minimal monomial lifting is the only possible candidate to give a cluster structure on $\Orb_\lX(\lX)$ compatible with the one on $\Orb_Y(Y)$ in the following sense.

\begin{theo}
    \label{unicity minimal monomial lifting}
   Let  $\widetilde t$ be a seed of $\CC(\lX)$, $\mu \in \ZZ^{D \times I}$ and $\lambda \in \ZZ^{D \times I_{uf}}$ such that: \begin{enumerate}
       \item $\widetilde t$ is a $D$-pointed field extension of $t$.
       \item For any $i \in I$, $\widetilde x_i= x_i \,\hx\strut^{\mu_{\bullet,i}}_D.$ 
       \item For any $k \in I_{uf}$, $\mu_k(\widetilde x_k)= \mu_k(x_k)  \,\hx\strut^{\lambda_{\bullet,k}}_D$
   \end{enumerate} If $\uclu(\widetilde t)=\Orb_\lX(\lX)$, then $\mu=\nu$ and $\widetilde t= \lif t^D.$
\end{theo}

\begin{proof}
    For any $d \in D$, $X_d \not \in \Orb_\lX(\lX)^*$. It follows that $D \subseteq \widetilde I_{hf}$. We claim that, for any $d \in D$ and $i \in I$, $\val_d( \widetilde x_i)=0.$ This implies that $\mu=\nu$ because
    $$\val_d(\widetilde x_i)= \val_d(x_i \otimes 1) + \mu_{d,i}= -\nu_{d,i} + \mu_{d,i}.$$
    If $i \in I_{sf}$, since $I_{sf} \subseteq \widetilde I_{sf}$, then $\widetilde x_i \in \Orb_\lX(\lX)^*$. Then the claim follows from the Hartogs' lemma. Similarly, if $i \not \in I_{sf}$, hence $i \not \in \widetilde I_{sf}$, then $\widetilde x_i$ is irreducible in $\Orb_\lX(\lX)$. Then the claim follows from Hartogs' lemma and the second hypothesis in Definition \ref{def: suitable lifting}. For the final statement, it's sufficient to apply Proposition \ref{unicity monomial lifting} to $\widetilde t_D$.
\end{proof}

Theorem \ref{unicity minimal monomial lifting} gets more meaningful if the $T$-action on $T \times Y$ extends to an action on $\lX$. This becomes precise in Theorem \ref{thm:unicity min mon lifting torus}.

\begin{coro}
\label{coro:if equality monomial lifting commutes}
    Let $k \in I_{uf},$ $t'=\mu_k(t)$ and $\nu'=\mu_k(\nu)$. If $\uclu(\lif t^D)=\Orb_\lX(\lX)$, then $\nu'$ and $\mu_k((\lif t')^D)$ are respectively the minimal lifting matrix and the minimal monomial lifting of $t'$ associated to $\lX$.
\end{coro}

\begin{proof}
    Because of statement 5 of Lemma \ref{properties freezing} and Lemma \ref{mutation commutes lifting}, we can apply the previous theorem with $\widetilde t= \mu_k(\lif t^D)$ (considered as an extension of $t'$), $\mu= \nu'$ and $\lambda_{\bullet,i}= \mu_i(\nu')_{\bullet,i}$
\end{proof}

We expect the answer to the following question to be negative, in general, without the assumption that $\uclu(\lif t^D)= \Orb_\lX(\lX)$.

\begin{question}
    If $k \in I_{uf}$, is $\mu_k(\nu)$ the minimal lifting matrix of $\mu_k(t)$ associated to $\lX$?
\end{question}
It's interesting to notice that, reinterpreting the results of \cite{fei2016tensor} in this light, as it is done in Section \ref{sec: prod, tensor prod}, one can construct an example where $\clu(t)=\uclu(t)$, $\uclu(\lif t^D) = \Orb_\lX(\lX)$ but $\clu(\lif t^D) \neq \uclu(\lif t^D).$ See \cite{fei2016tensor}[Example 8.3]. 

\subsection{Equality conditions}
\label{Equality conditions section}

In this section we suppose to be in the setting of Theorem \ref{minimal monomial lifting theo}. In particular, $\lX$ and $t$ are fixed, $\nu$ and $\lif t^D$ denote the minimal lifting matrix and the minimal monomial lifting of $t$, associated to $\lX$. It's easy to see that we don't always have equality between $\uclu(\lif t^D)$ and $\Orb_\lX(\lX)$. Nevertheless, equality holds after localisation at the product of the frozen variables $\lif x_d$, for $d \in D$. Hence, the lack of equality  between between $\uclu(\lif t^D)$ and $\Orb_\lX(\lX)$ should be caused by some bad behaviour along the divisors $V(X_d).$ This is clarified in the following example.

\begin{example}
    We make an affine example in the spirit of Example \ref{example: affine suitable for lifting}. Consider $R=\CC[y,z,X, \frac{y+z}{X}] $ where $y,z$ and $X$ are abstract independent variables. Notice that $R$ is a polynomial ring in the variables $y, X, \frac{y+z}{X}$ and we have an isomorphism $R_X \simeq \CC[y,z]\otimes \CC[X^{\pm 1}]$. The ring $\CC[y,z]$ is the upper cluster algebra of the seed $t$ which has only highly-frozen vertices and has $y,z$ as cluster variables. Here $D= \{d\}$ consists of one element and $X_d=X$. Applying the minimal monomial lifting we get $\uclu(\lif t^D)= \CC[y,z,X]$, which is strictly contained in $R$. In this very simple example we clearly see the problem: $\cval_d(y+z)=0$, where $\cval_d$ is the cluster valuation associated to the frozen vertex $d$. Nevertheless, $\val_d(x+y)=1.$ Here $\val_d$ is the valuation induced by the prime ideal $(X)$ of $R$. If $\uclu(\lif t^D)= R$, the valuations $\cval_d$ and $\val_d$ should coincide.
    \end{example}

\begin{lemma}
    \label{inequality valuations min lifting}
    For any $f \in \Orb_\lX(\lX)$ and $d \in D$, we have $\val_d(f) \geq \cval_d(f)$. Here $\cval_d$ is the cluster valuation defined in \ref{cluster valuation}.
\end{lemma}

\begin{proof}
  It's clear from \eqref{eq 21} that, for any $i \in \, \lif \kern-0.1em I$ and $d \in D$, $\val_d(\lif x_i) = \cval_d(\lif x_i).$ Since $f \in \uclu(\lif t) $ by Theorem \ref{minimal monomial lifting theo}, then we can write $f = \frac{P}{M}$ where $P$ is a polynomial in the cluster variables of $\lif x$, which is not divisible by any unfrozen or semi-frozen variable of $\lif t$, and 
  $$M= \displaystyle \prod_{i \in I_{uf} \sqcup I_{sf} \sqcup D} \lif x_i^{m_i}$$ 
  is a Laurent monomial. Note that, for any $d \in D$, we have $\cval_d(P) = 0$. Moreover, since $P \in \Orb_\lX(\lX)$ because of Theorem \ref{minimal monomial lifting theo} and the Laurent phenomenon, then  $ \val_d(P) \geq 0.$ Then  $$ \cval_d(f)= \cval_d(P)- \cval_d(M)= -m_d = - \val_d( M) \leq \val_d(P) - \val_d(M)= \val_d(f). $$    
    \end{proof}

    \begin{prop}
        \label{conditions equality minimal lifting}
Suppose that $t^* \in \Delta(\lif t)$. The following are equivalent

\begin{enumerate}
    \item  $\uclu(\lif t^D)= \Orb_\lX(\lX).$
    \item For any $d \in D$, $\Orb_\lX(\lX) \subseteq \Li\biggl((t^{*})^{\{d\}} \biggr).$
    \item For any $d \in D$, $\cval_d(\Orb_\lX(\lX)) \subseteq \NN \cup \{ \infty \}$.
    \item For any $d \in D$, $\cval_d \geq \val_d$ over $\Orb_\lX(\lX).$
    \item For any $d \in D$, $\cval_d = \val_d$ over $\Orb_\lX(\lX).$
\end{enumerate}

    \end{prop}

\begin{proof}
  Condition 1 implies 2 and 2 implies 3 are obvious. For 3 implies 4, suppose there exists $f \in \Orb_\lX(\lX)$ and $d \in D$ such that $n_d = \val_d(f) > \cval_d(f)$. Then $\frac{f}{X_d^{n_d}} \in \Orb_\lX(\lX)$ and $\cval_d(f) < 0$, which is a contradiction. Then 4 implies 5 because of Lemma \ref{inequality valuations min lifting}. Finally, 5 implies that $\cval_d= \val_d$ holds over the fraction field of $\Orb_\lX(\lX)$, that we denote by $\FF$. By Theorem \ref{minimal monomial lifting theo}, $\FF$ is also the fraction field of $\uclu(\lif t)$. Then Hartogs' lemma and Theorem \ref{minimal monomial lifting theo} imply that $\Orb_\lX(\lX)= \bigcap_{d \in D}  \val_d\inv ( \ZZ_{\geq 0} \cup \{\infty \}) \bigcap  \uclu(\lif t) $. Hence, using part 8 of Lemma \ref{properties freezing}, we deduce that 5 implies 1.
\end{proof}

\begin{coro}
\label{cor:equality cond at d fixed}
    Conditions 2 to 5 of Proposition \ref{conditions equality minimal lifting} are equivalent at $d $ fixed. That is, for any $d \in D$ and $t^* \in \Delta(\lif t)$, the following are equivalent.
    \begin{enumerate}
          \item  $\Orb_\lX(\lX) \subseteq \Li\biggl((t^{*})^{\{d\}} \biggr).$
    \item  $\cval_d(\Orb_\lX(\lX)) \subseteq \NN \cup \{ \infty \}$.
    \item $\cval_d \geq \val_d$ over $\Orb_\lX(\lX).$
    \item $\cval_d = \val_d$ over $\Orb_\lX(\lX).$
    \end{enumerate}
\end{coro}

\begin{proof}
    The decomposition $$X(T)= \ZZ x_d \bigoplus \biggl( \bigoplus_{d' \in D \setminus \{d\}} \ZZ x_{d'}\biggr)$$ determines an isomorphism $T_d \times T_{\setminus d} \longto T$, where $T_d$ (resp. $T_{\setminus d}$) is a torus with character group $\ZZ x_d $ (resp. $\bigoplus_{d' \in D \setminus \{d\}} \ZZ x_{d'}$).
    Let $$\lX':= \lX \setminus \bigcup_{d' \in D \setminus \{d\}} V(X_{d'}) \quad \text{and} \quad Y'= T_{\setminus d} \times Y.$$
    Moreover, let 
    $$\phi': T_ d \times Y' \longto \lX'$$
    be the map obtained by composing $\phi$ with the natural isomorphism $T_d  \times Y' \longto T \times Y$. The triple $(\lX',\phi', X_d)$ is $\{d\}$-suitable for lifting. 
    Let $\nu_{\setminus d, \bullet}= \nu_{D \setminus \{d\}, \bullet}$ and $\lif t_{\setminus d} \xlongleftarrow{\nu_{\setminus d, \bullet}} t$.
    By Lemma \ref{lift in steps}, we have a commutative diagram 
     \[\begin{tikzcd}
  	& \lif t_{\setminus d} \\
	\lif t && t
	\arrow["{\nu_{\setminus d, \bullet}}"', from=2-3, to=1-2]
	\arrow["{(\nu_{d, \bullet}, 0)}"', from=1-2, to=2-1]
	\arrow["\nu"', from=2-3, to=2-1]
\end{tikzcd}\]
As in the proof of Theorem \ref{minimal monomial lifting theo}, we have that $\Orb_{Y'}(Y') \simeq\CC[T_{\setminus d}] \otimes \Orb_Y(Y)= \uclu(\lif t_{\setminus d}). $ Moreover, $(\nu_{d,\bullet})$ is the minimal lifting matrix of $\lif t_{\setminus d}$ with respect to $(\lX', \phi', X_d)$. Finally, note that 
$$\Orb_{\lX'}(\lX')= \Orb_\lX(\lX)_{\prod_{d' \in D \setminus \{d\}}X_{d'}}$$ because of the Hartogs' lemma. The corollary is then a direct consequence of  Proposition \ref{conditions equality minimal lifting}. Indeed, the 4 conditions of the corollary are equivalent to $\uclu(\lif t^{\{d\}}) = \Orb_{\lX'}(\lX').$
\end{proof}

\begin{coro}
    Let $D' \subseteq D$, then $\Orb_\lX(\lX) \subseteq \uclu(\lif t^{D'})$ if and only if $$D' \subseteq \{ d \, : \, \val_d =\cval_d \, \, \text{over} \, \,\Orb_\lX(\lX)\}.$$
\end{coro}

\begin{proof}
    Note that $\Orb_\lX(\lX) \subseteq \uclu(\lif t^{D'})$ if and only if for any $t^* \in \Delta(\lif t)$ and $d \in D'$, $\Orb_\lX(\lX) \subseteq \Li((t^*)^{\{d\}})$. Then the statement follows from Corollary \ref{cor:equality cond at d fixed}.
\end{proof}
Observe that $\Orb_Y(Y)$ has two pole filtrations: one  coming from the almost polynomial ring $\bigl(\Orb_\lX(\lX), \val , X, \phi^*\bigr)$, briefly denoted $\Orb_\lX(\lX)$, and the other one coming from the standard almost polynomial structure of $\uclu(\lif t^D).$ The one coming from $\Orb_\lX(\lX)$ may not admit a lifting graduation.

\begin{coro}
\label{equality minimal lift torus action}
   If $\uclu(\lif t^D)= \Orb_\lX(\lX)$, then the two pole filtrations on $\Orb_Y(Y)$ coincide. In particular there exists an action of $T$ on $\Spec(\Orb_\lX(\lX))$ such that the map 
   $$T \times \Spec(\Orb_Y(Y))  \longto \Spec(\Orb_\lX(\lX))$$ 
   induced by $\phi^*$ is equivariant.
\end{coro}

\begin{proof}
     Because of Proposition \ref{conditions equality minimal lifting}, $\val= \cval$. Hence the two pole filtrations on $\Orb_Y(Y)$ coincide. The rest of the corollary is a direct consequence of Propositions \ref{lifting graduation cluster} and \ref{lifting graduation torus action}. 
\end{proof}
Thus, when $\lX$ is affine, the existence of a torus action on $\lX$ extending the one on $T \times Y$ is a necessary condition for having equality between $\uclu(\lif t^D)$ and $\Orb_\lX(\lX)$. 
In general, we have the following lemma.

\begin{lemma}
\label{lifting subgraded}
Let $\hT$ be a torus acting on $Y$. Suppose that the natural $X(\hT)$-graduation on $\Orb_Y(Y)= \uclu(t)$ is induced from a $X(\hT)$ degree configuration $\sigma$ on $t$ (eventually, $\hT= \{e\}$ and $\sigma = 0$). Suppose that there exists an action of $T \times \hT$, on $\lX$, extending the natural action on $T \times Y$. Then $\uclu(\lif t^D)$, with the graduation induced by $\lif \sigma$, is a $(X(T) \times X(\hT))$-graded subalgebra of $\Orb_\lX(\lX)$.
\end{lemma}

\begin{proof}
Note that the pole filtration on $\Orb_Y(Y)$ coming from $\lX$, that is from the almost polynomial ring $(\Orb_\lX(\lX), \val, X, \phi^*)$, is $X(\hT)$-graded. Indeed, if $\lambda \in \ZZ^D= X(T)$, for $f \in \Orb_Y(Y)$ we have that 
\begin{equation}
    \label{eq: 4 1} f \in \Orb_Y(Y)_\lambda \ifff X^\lambda (1 \otimes f) \in \Orb_\lX(\lX).
\end{equation} 
Here, $\Orb_Y(Y)_\lambda$ is the $\lambda$-component of the pole filtration on $\Orb_Y(Y).$ Since, for any $d \in D$, the function $X_d$ is $T \times \hT$ semi-invariant, we deduce from \eqref{eq: 4 1} that $\Orb_Y(Y)_\lambda$ is $\hT$-stable which implies that the pole filtration is $X(\hT)$-graded.

  Consider the induced $T \times \hT$-action on $\Spec(\Orb_\lX(\lX)).$ 
  We have an open embedding $ T \times \Spec(\Orb_Y(Y)) \longto \Spec(\Orb_\lX(\lX)) $ induced by $\phi^*$, which is clearly $T \times \hT$-equivariant. By Lemmas \ref{lifting graduation torus action} and \ref{lem: H graded pole filt equivariance}, the graduation on $\Orb_\lX(\lX)$ is a $X(\hT)$-graded lifting graduation.  We denote by $\deg(f) \in X(T) \times X(\hT)$ the degree of an homogeneous element $f \in \Orb_\lX(\lX).$ From the definition of lifting graduation we have that, for $d \in D$, 
  $$\deg(\lif x_d)=\deg(X_d)=(e_d, 0)= \lif \sigma_d.$$ 
  Moreover, for $i \in I$, 
  $$\deg(\lif x_i)= \deg(1 \otimes x_i) + \deg(X^{\nu_{\bullet,i}})= (\nu_{\bullet, i}, \sigma_i)= \lif \sigma_i.$$
Hence, for any $i \in \lif \kern-0.15em I$, the degree of $\lif x_i$ as an element of $\Orb_\lX(\lX)$ is the same as its degree as an element of $\uclu(\lif t^D)$ with respect to the graduation induced by $\lif \sigma$. In particular, the two graduations coincide on $\Li(\lif t)$. The lemma follows because both $\uclu(\lif t^D)$ and $\Orb_\lX(\lX)$ are graded subalgebras of $\Li(\lif t)$.
\end{proof}

\begin{definition}
    \label{def: optimal for lifting} A suitable for $D$-lifting scheme is said to be \textit{homogeneously suitable for $D$-lifting} if the $T$-action of $T$ on $T \times Y$ extends to an action on $\lX$.
\end{definition}
 Corollary \ref{equality minimal lift torus action} and Lemma \ref{lifting subgraded} suggest that, homogeneous suitable for $D$-lifting schemes, should provide a good environment where the minimal monomial lifting technique could give important information on the graded ring $\Orb_\lX(\lX)$. 
   Indeed, the minimal monomial lifting is the only possible candidate for extending the cluster structure on $\Orb(Y)$ to a $X(T)$-graded cluster structure on $\Orb_\lX(\lX)$. More precisely, we have the following theorem.
     \begin{theo}
         \label{thm:unicity min mon lifting torus} Let $\lX$ be an homogeneously suitable for $D$-lifting scheme and $\widetilde t$ be $D$-pointed seed extension of $t$, in the field $\CC(\lX)$. Let $s: \Orb(\lX) \longto \Orb(Y)$ be the restriction map obtained by identifying $Y$ with $\{e\} \times Y$. Suppose that
         \begin{enumerate}
             \item For any $i \in I$, $s(\widetilde x_i)= x_i $ and $\widetilde x_i$ is $X(T)$-homogeneous, with respect to the $T$-action on $\lX$.
       \item For any $k \in I_{uf}$, $s\biggl( \mu_k(\widetilde x_k)\biggr)= \mu_k(x_k) $ and $\mu_k(\widetilde x_k)$ is $X(T)$-homogeneous.
       \item  $\uclu(\widetilde t)= \Orb(\lX)$.
         \end{enumerate}
         Then $\widetilde t=\lif t^D$.
         \end{theo}

         \begin{proof}
             By the proof of Lemma \ref{lifting subgraded}, it follows that the $X(T)$-graduation on $\Orb(\lX)$ is a lifting graduation. Corollary \ref{homogeneous elements lifting graduation} implies that assumptions 1 and 2 of the theorem are respectively equivalent to assumptions 2 and 3 of Theorem \ref{unicity minimal monomial lifting}. Hence, the statement is a reformulation of Theorem \ref{unicity minimal monomial lifting}.
         \end{proof}

Finally, we improve Proposition \ref{conditions equality minimal lifting} in the affine case. That is: $(\lX, \phi, X)$ is constructed as in Example \ref{example: affine suitable for lifting} from a given $(R, \psi, X)$. Let $t^* \in \Delta(\lif t)$ and recall that $\CC[t^*]$ denotes the polynomial ring in the cluster variables of the seed $t^*$. By Theorem \ref{minimal monomial lifting theo}, we have inclusions $\CC[t^*] \subseteq R \subseteq \Li(t^*)$, which correspond to maps $$ \Spec(\Li(t^*)) \xrightarrow{\iota_{t^*}}  \lX \xrightarrow{\pi_{t^*}} \Spec (\CC[t^*])$$
where $\iota_{t^*}$ is an open embedding whose image is a principal open subset. 

\begin{prop}
    \label{condition equality min lifting affine} The following are equivalent

    \begin{enumerate}
        \item $\uclu(\lif t^D)=R$.
        \item For any $d \in D$ there exists a map $\iota_{t^*}^d: \Spec\biggl(\Li\biggl((t^{*})^{\{d\}}\biggr)\biggr) \longto \lX$ extending $\iota_{t^*}$. In this case  $\iota_{t^*}^d$ is an open embedding.
        \item If $p_d$ (resp. $q_d^*) $ is the point of $\lX$ (resp. $\Spec(\CC[t^*])$ corresponding to the ideal generated by $X_d$ in $R$ (resp. $\CC[t^*])$, then $\pi_{t^*}(p_d)=q_d^*$. That is $\pi_{t^*}$ does not contract the divisor $\overline{\{p_d\}}$.
    \end{enumerate}
\end{prop}

\begin{proof}
    Condition 1 implies 2 is a reformulation of the same implication in Proposition \ref{conditions equality minimal lifting}. The fact that, if $\iota_{t^*}^d$ exists, then it is an open embedding, is because if $R \subseteq \Li\biggl((t^{*})^{\{d\}}\biggr)$, then the localisation of $R$ at all the unfrozen and semi-frozen variables of $(t^{*})^{\{d\}}$ is $\Li\biggl((t^{*})^{\{d\}}\biggr)$ since $\CC[t^*] \subseteq R.$
    If 2 holds, then $\Spec \biggl( \Li\biggl((t^{*})^{\{d\}}\biggr)\biggr) $ can be identified with an open subscheme of $\lX$ containing $p_d$. Since $\pi_{t^*} \circ \iota_{t^*}^d$ corresponds to the natural inclusion $\CC[t^*] \subseteq  \Li\biggl((t^{*})^{\{d\}}\biggr)$, it clearly sends $p_d$ to $q_d^*$.
   Finally, if condition 3 holds, then we have an inclusion $\Orb_{\Spec(\CC[t^*]), q_d^*} \subseteq \Orb_{\lX,p_d}$. But these two rings are discrete valuation  domains  with the same fraction field, hence they have to be equal. In particular, $\cval_d= \val_d$ on $\CC(\lX)$, hence condition 3 implies 1 by Proposition \ref{conditions equality minimal lifting}.
\end{proof}

\newpage

\section{Preliminaries on algebraic groups}
\label{Preliminaries alg groups section}
\subsection{Some classical facts and notation}
\label{Some classical facts and notation section}
In this section we introduce the notation and recall some classical facts about algebraic groups, we mostly follow \cite{fomin1999double}.
From now on, unless explicitly stated, $G$ is a complex, semisimple, simply connected algebraic group of rank $r$. Moreover, $B$ and $B^-$ are opposite Borel subgroups with unipotent radicals $U$ and $U^-$ respectively. Let $T= B \cap B^-$, which is a maximal torus of $G$, and $W=N_G(T)/T$ be the Weyl group.
Denote by $\lg$, $\lb$, $\lu$, $\lb^-$, $\lu^-$ and $\lh$  the Lie algebras of $G,B,U,B^-,U^-$ and $T$ respectively. We denote by $X(T)$ and $X(T)^\vee$ the set of characters and cocharacters of $T$ respectively. We write $\langle - , - \rangle $ for the natural pairing between $X(T)$ and $X(T)^\vee$.
Let $\Phi \subseteq X(T)$ be the root system of $G$ and $\Delta \subseteq \Phi$ be the set of simple roots determined by $B$. We denote by $\Phi^+$ the positive roots and $\Phi^-=-\Phi^+$ the negative ones, so that the root space $\lg_\alpha \subset \lu$ for $\alpha \in \Phi^+$. For $\beta \in \Phi$, $\beta^\vee \in X(T)^\vee$ is the corresponding coroot. We denote by $A=(a_{\alpha,\beta})_{\alpha,\beta \in\Delta}$  the \textit{Cartan matrix}, which is defined by $a_{\alpha,\beta}= \langle \beta , \alpha^\vee \rangle$. 
The character group $X(T)$ is free on the set of fundamental weights $\varpi_\alpha$, indexed by the simple roots $\alpha \in \Delta$. For $\beta \in \Phi$, we denote by $s_\beta \in W$ the associated reflection. The set $\{ s_\alpha : \alpha \in \Delta\} $ is a set of simple reflection of $W$ as a Coxeter group. 

\bigskip

Given a representation $V$ of $G$ and $\mu \in X(T)$, we denote  the associated \textit{weight space} by $$V_\mu=\{ v \in V \, : \,  t \cdot v=\mu(t)v \,  \, \text{for} \, t \in T \}.$$ Recall that $V= \bigoplus_{\mu \in X(T)}V_\mu$. We consider $V^*$ as a representation of $G$ as follows: if $\varphi \in V^*$, $g \in G$ and $v \in V$, then $(g\varphi)(v)=\varphi( 
g^{-1} v)$.
Let $X(T)^+$ be the set of dominant characters of $T$. For $\lambda \in X(T)^+$, the associated irreducible representation of $G$ is denoted by $V(\lambda)$. If $\lambda \in X(T)$, we denote by $\lambda^*= - w_0\lambda$. This defines a linear automorphism of $X(T)$ that stabilises $X(T)^+$. In particular, $V(\lambda)^*=V(\lambda^*).$

\bigskip

The following discussion holds for a general connected reductive group $G$. Consider the action of $G \times G$ on $G$ defined by $(g_1,g_2)\cdot g= g_1 g g_2\inv$. This gives to $\CC[G]$ the structure of a $G \times G$ module. We have an isomorphism of $G \times G$ representations: \begin{equation}
\begin{array}{r c l}
  \label{Peter Weyl thm}
   \displaystyle  \bigoplus_{\lambda \in X(T)^+} V(\lambda)^* \otimes V(\lambda) &  \longto & \CC[G] \\
    (\varphi_\lambda \otimes v_\lambda)_\lambda & \longmapsto & ( g \longmapsto \sum_\lambda \varphi_\lambda(g v_\lambda)).
\end{array}
\end{equation}

Consider the action of $U^-$ (resp. $U$) on $G$ by left (resp. right) multiplication. Then, $$ \begin{array}{c c}
     \CC[G]^{U^-}=  \CC[U^- \setminus G] &
     \CC[G]^{U}=\CC[G/U].
\end{array} 
$$ 

The two rings above inherit a natural structure of $G$-module. For any $\lambda \in X(T)^+$, we fix elements $\varphi_\lambda^- \in (V(\lambda)^*)^{U^-}=V(\lambda)^*_{-\lambda}$ and $v_\lambda^+ \in V(\lambda)^U=V(\lambda)_\lambda$ such that $\varphi_\lambda^-(v_\lambda^+)=1.$ Then, we have isomorphisms of $G$-modules:
\begin{equation}
\label{ C (G/U)}
\begin{array}{r c l }
  \displaystyle    
\bigoplus_{\lambda \in X(T)^+} V(\lambda) & \longto & \CC[G]^{U^-}  \\  
(x_\lambda)_\lambda & \longmapsto & (g \longmapsto \sum_\lambda \varphi_\lambda^-(g x_\lambda)) \\[1em]  
\displaystyle \bigoplus_{\lambda \in X(T)^+} V(\lambda)^* &\longto & \CC[G]^{U}\\
(\psi_\lambda)_\lambda & \longmapsto & ( g \longmapsto \sum_\lambda \psi_\lambda(g v_\lambda^+)).
\end{array}
\end{equation} 
Moreover, since $T$ normalises $U$ and $U^-$,  $T$ acts on $U^- \setminus G$ (resp. $G/U)$ by left (resp. right) multiplication. It's easy to notice that the homogeneous components of the $X(T)$-graduations induced by these actions, on the two rings, coincide with the decomposition given above into $G$-modules. In particular, for $\lambda \in X(T)$,

\begin{equation}
\label{G/U homogeneous components}
     (\CC[G]^{U^-})_\lambda = V(\lambda) \quad 
     (\CC[G]^{U})_\lambda = V(\lambda)^*,
     \end{equation}
were of course we set $V(\lambda)=V(\lambda)^*=0$ if $\lambda $ is not a dominant character. To avoid confusion, we recall that, accordingly to the notation introduced in Section \ref{pole filtration section} we have that 
$$
\begin{array}{c}
 (\CC[G]^{U^-})_\lambda = \{ f \in \CC[G]^{U^-}  : f(tg)=\lambda(t)f(g) \, \,\text{for any} \, t \in T, g \in G \} \\
 (\CC[G]^{U})_\lambda=\{ f \in \CC[G]^U \, : f(gt)=\lambda(t)f(g) \, \, \text{for any} \, t \in T, g \in G \}.
 \end{array}
$$

\subsection{Generalised minors}
\label{generalised minors section}
For any $\beta \in \Phi^+$, we chose an $\mathfrak{sl}_2$-triple $(X_\beta, H_\beta, X_{-\beta}) $ such that $X_{\pm \beta} \in \lg_{\pm \beta}$, $H_\beta \in \lh$, and $[X_\beta, X_{-\beta}]=H_\beta$. For $t \in \CC$ and $\beta \in \Phi$, let 
\begin{equation}
\label{x beta}
    x_{ \beta}(t)= \exp(t X_{\beta}).
\end{equation}

For $\alpha \in \Delta$, denote by $\varphi_\alpha : \SL_2 \longto G$ the morphism determined by the $\mathfrak{sl}_2$-triple associated to $\alpha$, so that 

\begin{equation*}
   \varphi_\alpha \begin{pmatrix} 
   1 & t \\
   0 & 1
    \end{pmatrix}=x_\alpha(t)
   \quad  \varphi_\alpha \begin{pmatrix} 
   t& 0 \\
   0 & t\inv
\end{pmatrix}= \alpha^\vee(t) \quad 
   \varphi_\alpha \begin{pmatrix} 
   1 & 0 \\
   t & 1
    \end{pmatrix}=x_{-\alpha}(t).
\end{equation*}
We define
\begin{equation}
    \label{s bar}
\sba:= \varphi_\alpha \begin{pmatrix} 
    0 & -1\\
    1 & 0
   \end{pmatrix}.
   \end{equation}
A sequence $ \ii=(i_l, \dots , i_i )$, of elements of $\Delta$, is called \textit{reduced expression} for $w \in W$ if $s_{i_l} \dots s_{i_1}=w$ and $l$ is minimal. In this case, $l=\ell(w)$ where $\ell$ is the  \textit{length} function of the Coxeter group $W$. We denote by $R(w)$ the set of reduced expressions of $w$, and by $w_0$ the longest element of $W$. Recall that, by a result of Matsumoto and Tits, any two reduced expressions of the same element can be obtained from each other by applying a sequence of braid moves.
The \textit{support} of $w$, denoted by $\supp(w)$, is the set of simple reflections appearing in a fixed (equivalently in any) reduced expression of $w$. Since the family $\{ \sba : \alpha \in \Delta \}$ satisfies the braid relations, if $\ii=(i_l, \dots , i_i) \in R(w)$, the element $$\overline{w}= \overline{s}_{i_l} \dots \overline{s}_{i_1}$$
doesn't depend on $\ii$.
In the following, we make an extensive use of the following fundamental commutation relations. If $\alpha \in \Delta$ and $t \in \CC^*$, then
\begin{equation}
    \label{x alpha s alpha = ..}
    \begin{array}{l}
        x_\alpha(t) \sba  = x_{-\alpha}(t\inv) \alpha^\vee(t)x_\alpha(-t\inv)\\
         \sba^{-1} x_{-\alpha}(t)= x_{-\alpha}(-t\inv)\alpha^\vee(t)x_\alpha(t\inv).
    \end{array}
\end{equation}

We consider the Bruhat decomposition $$ \displaystyle G= \bigsqcup_{w \in W} B^-wB$$ and denote by $G_0=B^-B$ the open cell. Recall that the product induces an isomorphism of varieties between $U^- \times T \times U$ and $B^-B$. In particular, an element $x \in G_0$ can be written uniquely as $x=[x]_-[x]_0[x]_+$ where $[x]_- \in U^-$, $[x]_0 \in T$ and $[x]_+ \in U$. Moreover, $G \setminus G_0= \bigcup_{\alpha \in \Delta} \overline{B^- s_\alpha B}$ and the set of $\overline{B^- s_\alpha B}$ consists of pairwise distinct divisors of $G$. 

\bigskip

For any $\alpha \in \Delta$, we define the \textit{generalised minor} $\delaee$ as the element of $\CC[G]$ corresponding to $\varphi_{\pia}^- \otimes v_{\pia}^+$ via \eqref{Peter Weyl thm}. Namely 
\begin{equation}
    \label{principal minor e e defi}
    \delaee(g)=\phiam(g \vap).
\end{equation}
The generalised minor $\delaee$ is the only regular function on $G$, such that for any $x \in G_0$, $\delaee(x)= \varpi_\alpha([x]_0).$
Then, if $v,w \in W$, we define the generalised minor $\delavw$ by $\delavw= (\overline{v},\overline{w})\cdot \delaee$. Explicitly, if $g \in G$, then
\begin{equation}
    \label{principal minor v,w defi}
    \delavw(g)=\delaee(\overline{v}\inv g \overline{w}).
\end{equation}
Note that our notation is slightly different from the classical one where $\delavw$ is denoted by \\$\Delta_{v\varpi_\alpha, w \varpi_\alpha}$.
We denote by $\leq_R$ the right weak order on $W$: $v \leq_R w$ if and only if $\ell(w)=\ell(v) + \ell(v\inv w).$ 
We recall a bunch of well known properties of the generalised minors that are either trivial or can be found in \cite{fomin1999double} (\cite{fomin1999double}[Propositions 2.2,2.4, Section 2.7]) and that are largely used in this text. 

\begin{lemma}[Basic properties of minors]
\label{lem:basic prop minors}
    Let $\alpha \in \Delta$, $v,w,v',w' \in W$ and $h,h' \in T$, then 
    \begin{enumerate}
        \item If $v\pia=v' \pia$ and $w\pia=w'\pia$, then $\delavw=\Delta^{\pia}_{v',w'}$.
        \item  $\delavw$ is $vU^-v\inv$-invariant on the left and $wUw\inv$ invariant on the right.
    \item $\delavw(hgh')=v\pia(h)w\pia(h')\delavw(g)$.
    \item The zero locus of $\delaee$ is $\overline{B^-s_\alpha B}$.
    \item $\delavw$ vanishes on $U$ if and only if it vanishes on $U \cap v U v\inv \cap w U^- w\inv.$ 
    \item If $v \leq_R w$, then $\delavw$ doesn't vanish on $U$.
    \item $\delaew$ doesn't vanish on $U$.
    \end{enumerate}
\end{lemma}

 Fomin and Zelevinsky proved the following fundamental identity \cite[Theorem 1.17]{fomin1999double}, which stands at the very base of the known cluster algebra structures related to $G$.

 \begin{theo}
 \label{relations minors FZ}
 For any $\alpha \in \Delta$ and $v,w \in W$, such that $\ell(v s_\alpha)=\ell(v)+1$ and $\ell(w s_\alpha)=\ell(w)+1$, then 
 $$\delavw \Delta^{\varpi_\alpha}_{vs_\alpha,ws_\alpha}= \Delta^{\pia}_{vs_\alpha, w}\Delta^{\pia}_{v,ws_\alpha}+ \prod_{\beta \in \Delta \setminus\{\alpha\}}(\delbvw)^{-a_{\beta,\alpha}}. $$
\end{theo}
 The following lemma seems well known but we add a proof because of a lack of a precise reference.
 \begin{lemma}
 \label{lem:minors irred}
     If $\alpha \in \Delta$ and $v,w \in W$, the generalised minor $\delavw$ is irreducible in $\CC[G]$. Moreover, the principal ideals generated by $\delaee$ are pairwise different with respect to $\alpha$.
 \end{lemma}

 The proof relies on the following well known fact, which can be found in \cite{popov1994invariant} (Theorem 3.7 and proof of Theorem 3.1) and which is frequently used in this text.

\begin{lemma}
\label{valuation on invariants or above}
    Let $R$ be a factorial $\CC$-algebra of finite type and $H$ be a connected algebraic group acting on $R$. If $X(H)=\{e\}$ then
    \begin{itemize}
        \item $R^H$ is factorial.
        \item An element $f \in R^H$ is irreducible, in $R^H$, if and only if it is irreducible in $R$.
        \item If $f \in R^H$, its factorisations into irreducibles in $R^H$ and in $R$ coincide.
    \end{itemize}
\end{lemma}

   \begin{proof}[Proof of Lemma \ref{lem:minors irred}]
   From the definitions, for proving that $\delavw$ is irreducible, it's sufficient to prove that $\delaee$ is irreducible. Since $G$ is semi-simple and simply connected, then $\CC[G]$ is factorial. Consider the $U^- \times U$-action on $G$ defined by $(u^-,u)\cdot g= u^- g u\inv$. Since $U^- \times U$ is connected and has no non-trivial character, Lemma \ref{valuation on invariants or above} applies. Let $\CC[T_\beta]_{\beta \in \Delta}$ be a polynomial ring in the independent variables $T_\beta$. From \eqref{Peter Weyl thm}, we easily deduce that the map $$ \begin{array}{r c l}
     \CC[T_\beta]_\beta  & \longto & \CC[G]^{U^- \times U} \\
      T_\beta & \longmapsto & \delbee
   \end{array}$$
   is an isomorphism. Since $T_\alpha$ is irreducible in $\CC[T_\beta]_\beta$, we deduce from Lemma \ref{valuation on invariants or above} that $\delaee$ is irreducible in $\CC[G]$. The last part of the statement follows from the fact that $V(\delavw)=\overline{B^- s_\alpha B}$ and these divisors are pairwise distinct.
   \end{proof}

   \subsection{Cluster structure on U(w)}
\label{Cluster structure on U(w) section}

Let $w \in W$, we set $\Phi(w)= \Phi^+ \cap w\inv \Phi^-$ the set of \textit{inversions} of $w$. Let $$U(w):= U \cap w\inv U^- w.$$
This is the $T$-stable subgroup of $U$ whose Lie algebra is the direct sum of the root spaces corresponding to elements of $\Phi(w).$ We recall the cluster algebra structure constructed by Goodearl and Yakimov on $\CC[U(w)]$, in \cite[Theorem 7.3]{goodearl2021integral}. This coincides with the one found by Geiss, Leclerc and Schröer, in \cite{geiss2011kac}, in the simply laced case for $G$ simple. Actually, in  \cite{goodearl2021integral} the construction is done in the quantum setting. The specialisation to the classical setting is made possible by \cite[Theorem 1.6]{davison2021strong}. Note that \cite{goodearl2021integral} and \cite{geiss2011kac} adopt different conventions for the meaning of $N(w)$. In what follows, we slightly modify the notation of Goodearl and Yakimov. Our notation is almost identical to the one used by Geiss, Leclerc and Schröer.

\bigskip

Fix $\ii = (i_l, \dots , i_1) \in R(w)$. 
    For convenience we set $s_{i_0}=s_{i_{l+1}}=e.$
  We denote 
\begin{equation}
\label{w_k}
    w_{\leq k}^{-1}= s_{i_1} \dots s_{i_k}, \quad w_{\leq k}= s_{i_k} \dots s_{i_1}= (w_{\leq k}^{-1})\inv.
\end{equation}
For 
    $k \in \{0, 1, \dots, l, l+1 \}$
  we use the notation 
\begin{equation}
    \label{k+ k-}
    \begin{array}{l }
    k^+= \begin{cases}
        \min \{ \{ k < j \leq l \, : i_j= i_k \} \cup \{l+1 \}\}  & \text{if} \quad k \not \in \{0,l+1\}\\
        k  & \text{if} \quad k \in \{0,l+1\}
    \end{cases} \\[1.5 em]
       k^-= \begin{cases}
           \max \{ \{ 1 \leq j < k \, : \, i_j=i_k \} \cup \{ 0 \} \} & \text{if} \quad k \not \in \{0,l+1\}\\
           k &  \text{if} \quad k  \in \{0,l+1\}
       \end{cases} 
    \end{array}
\end{equation}

Moreover, for a simple root $\alpha$ such that $s_\alpha \in \supp(w)$, we denote 
\begin{equation}
\label{alpha }
    \alpha^{\min}= \min \{  1 \leq k \leq l \, : \, i_k= \alpha \} \quad \alpha^{\max}= \max \{  1 \leq k \leq l \, : \, i_k= \alpha \}.
\end{equation}
Recall that $A= (a_{\alpha,\beta})_{\alpha, \beta \in \Delta}$ is the Cartan matrix of $\Phi$. For an integer $N \in \NN$, we use the notation $$[N]= \{ 1 , \dots , N \}. $$
We consider the seed $t=t_\ii$, of $\CC(U(w))$, defined as follows: 

\begin{enumerate}
    \item[-] The vertex set is $I= [l]$. Moreover, $I_{sf}=\emptyset$ and $I_{hf}= \{ \alpha^{\max} \, : \, s_\alpha \in \supp(w) \}.$
    \item[-] For $j \in I, k \in I_{uf}$, the coefficient $b_{j,k}$ of the generalised exchange matrix $B$ is defined by \begin{equation}
        \label{B for U(w)}
b_{j,k}= \begin{aligned}
    \begin{cases}
        1 & \text{if} \quad j=k^-\\
        -1 & \text{if} \quad j=k^+ \\
        a_{i_j,i_k} & \text{if} \quad j < k < j^+ < k^+\\
        -  a_{i_j,i_k} & \text{if} \quad k<j<k^+<j^+\\
        0 & \text{otherwise}
    \end{cases}
\end{aligned}
    \end{equation} 
     
        It is practical to set, for $k \in I_{uf}$ $$b_{0,k}=b_{l+1,k}=0.$$
    
\item[-] For $k \in I$, the cluster variable $x_k \in \CC[U(w)]$ is the restriction at $U(w)$ of a generalised minor, in particular: 

\begin{equation}
    \label{x for U(w)}
    x_k= \bigl(\Delta^{\varpi_{i_k}}_{e, \wkm}\bigr)_{|U(w)} 
\end{equation}

\end{enumerate}
We refer to the generalised minor in \eqref{x for U(w)} as the \textit{minor defining} $x_k$.

\bigskip

\begin{example}
    Let $G=\SL_3$, $\Delta= \{1,2\}$ and $\ii=(1,2,1) \in R(w_0)$. Here $U=U(w_0)$ and
    $$U(\CC)= \Biggl \{ \bmat 
    1 & a & b \\
    0 & 1 & c\\
    0 & 0 & 1
    \emat \, : \, a,b,c \in \CC \Biggr \}.$$
    The seed $t=t_\ii$ is graphically represented by the following quiver 
    
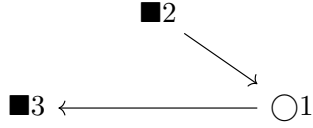
\[\begin{tikzcd}
	& {\blacksquare 2} \\
	{\blacksquare 3} && {\bigcirc 1}
	\arrow[from=1-2, to=2-3]
	\arrow[from=2-3, to=2-1]
\end{tikzcd}\]
and its cluster variables are
$$x_1=a \quad x_2=ac-b \quad x_3=b.$$
Note that $\mu_1(x_1)=c.$
\end{example}

The following is well known but we add a proof for completeness.

\begin{lemma}
\label{maximal rank U(w)}
   The seed $t$ is of maximal rank.
\end{lemma}

\begin{proof}
This is the argument of \cite[Proposition 2.6]{berenstein2005cluster3}. Consider the minor of $B$ whose rows are indexed by $I_{uf}^+=I \setminus \{ \alpha^{\min} \, : \, s_\alpha \in \supp(w) \}$ and columns indexed by $I_{uf}$. The map $I_{uf}^+ \longto I_{uf}$ sending $i$ to $i^-$ is a bijection. Order the rows of the minor according to the natural order, induced by $\NN$. Then, order the columns accordingly to the above bijection, that is: $i^- < j^-$ if and only if $i < j.$ It's immediate to check that the minor considered is upper triangular with $-1$'s on the diagonal.
\end{proof}

\begin{theo}
    \label{cluster U(w)}
    We have equalities $\clu(t)=\uclu(t)=\CC[U(w)].$
\end{theo}

The following proposition seems new to the author.
\begin{prop}
    \label{prop:t i graded}  The collection of weights $\sigma \in X(T)^{I}$ defined by $\sigma_{k}= w_{\leq k}^{-1} \varpi_{i_k} - \varpi_{i_k}$ is a degree configuration on $t$. For convenience, we set $\sigma_0=\sigma_{l+1}=0.$
\end{prop}
    
  \begin{remark}
      Note that $T$ acts on $U(w)$ by conjugation and the weight of $x_k$, with respect to the graduation induced by this action, is precisely $\sigma_k$. It follows that the equality of Theorem \ref{cluster U(w)} is of graded algebras, where $\clu(t)=\uclu(t)$ is graded by the degree configuration $\sigma$. The importance of this remark is explained at the end of Section \ref{sec:on question }.
  \end{remark} 

For the proof of Proposition \ref{prop:t i graded} we introduce the following notation: if $k \in I_{uf}$ and $ \gamma \in \Delta \setminus \{i_k\}$ we set
\begin{equation}
    \label{eq: k max min}
    \begin{array}{l}
    \gamma_{k\min}= \min \{ \{ j \, : \, k < j < k^+, \,  i_j=\gamma\} \cup \{0\}\} \\[0.5em]  \gamma_{k\max}= \max \{ \{ j \, : \, k < j < k^+, \,  i_j=\gamma\} \cup \{l+1\}\}.
    \end{array}
\end{equation}
The following property of the matrix $B$ is obvious from its definition, but useful.

\begin{lemma}
    \label{lem:B i alternativa}
    Let $k \in I_{uf}$ and $j \in I$ such that $b_{j,k} \neq 0$. One of the following conditions hold:
    \begin{enumerate}
        \item $j =k^- \neq 0$ and $b_{j,k}=1.$
        \item $j=k^+$ and $b_{j,k}=-1.$
        \item $i_j = \gamma \neq i_k$ and $j = \gamma_{k\min}^- \neq 0.$ Then, $b_{j,k}= \langle i_k, \gamma^\vee \rangle.$
        \item $i_j= \gamma \neq i_k $ and $j= \gamma_{k \max} $. Then $b_{j,k}= - \langle i_k, \gamma^\vee \rangle.$
    \end{enumerate}
\end{lemma}

   \begin{proof}[Proof of Proposition \ref{prop:t i graded}.]
Let $k \in I_{uf}$ and $ j \in I$. Form the definition of $\sigma$, using that $s_{i_j} \varpi_{i_j}= \varpi_{i_j} - i_j$, we deduce that 
\begin{equation}
    \label{eq: sigma i ricorrente}
    \sigma_j= \begin{cases}
        \sigma_{j^-}
- s_{i_1} \dots s_{i_{j-1}} i_j & \text{if} \quad j>1\\
\sigma_0 - i_1 & \text{if} \quad j=1,
    \end{cases}
\end{equation}
where recall that we set  $\sigma_0=0$.
Set 
$$ _k\Delta_{k^+}= \{ \gamma \in \Delta \setminus \{ i_k\} \, : \, \gamma_{k \max } \neq l+1\}.$$
Using Lemma \ref{lem:B i alternativa}, the statement of Proposition \ref{prop:t i graded} is equivalent to the identity 
\begin{equation}
    \label{eq: eq:6 1}
    \sigma_{k^-}- \sigma_{k^+}= \sum_{\gamma \in  _k\Delta_{k^+}} \biggl( \sigma_{\gamma_{k\max}}- \sigma_{\gamma_{k \min}^-} \biggr) \langle i_k , \gamma^\vee \rangle.
\end{equation}

Iterating \eqref{eq: sigma i ricorrente} we deduce that, for $\gamma \in _k \Delta_{k^+},$

$$ \sigma_{\gamma_{k \max}}- \sigma_{\gamma_{k \min}^-}= - \sum_{ \substack{k < j < k^+ \\ i_j= \gamma}} s_{i_1} \dots s_{i_{j-1}} i_j.$$
Hence, the right and side of \eqref{eq: eq:6 1} is equal to the following expression:

\begin{equation}
    - \sum_{k < j < k^+} \langle i_k , i_j^\vee \rangle s_{i_i} \dots s_{i_{j-1}} i_j.
\end{equation}
Moreover, (in the following computation, if $k=1$ we should replace the terms $ s_{i_1} \dots s_{i_{k-1}} i_k$ with $i_k$)  we have that 
$$ 
\begin{array}{r l }
  \sigma_{k^-}- \sigma_{k^+}  = & s_{i_1} \dots s_{i_{k^+-1}} i_k  + s_{i_1} \dots s_{i_{k-1}} i_k\\
  = & - \displaystyle \biggl(\sum_{k < j < k^+}\langle i_k , i_j^\vee \rangle s_{i_1} \dots s_{i_{j-1}} i_j  \biggr) + s_{i_1} \dots s_{i_{k}} i_{k} +s_{i_1} \dots s_{i_{k-1}} i_{k}\\
  = & - \displaystyle \biggl(\sum_{k < j < k^+}\langle i_k , i_j^\vee \rangle s_{i_1} \dots s_{i_{j-1}} i_j  \biggr).
  \end{array}$$
The first of the previous identities follows again from \eqref{eq: sigma i ricorrente}, while the second one is deduced easily by iterating the formula
$$s_{\alpha} i_k= i_k - \langle i_k, \alpha^\vee \rangle \alpha \quad \text{for} \quad \alpha \in \Delta.$$
The third identity is obvious. This completes the proof.
   \end{proof}

\begin{example}
    Let $G=\SL_4$. We use the notation of \cite{Bourb}. Let $\ii=( 2 , 3 , 1 , 2) \in R(z)$ where $z$ is the permutation $(1,3)(2,4)$. We can compute that

    
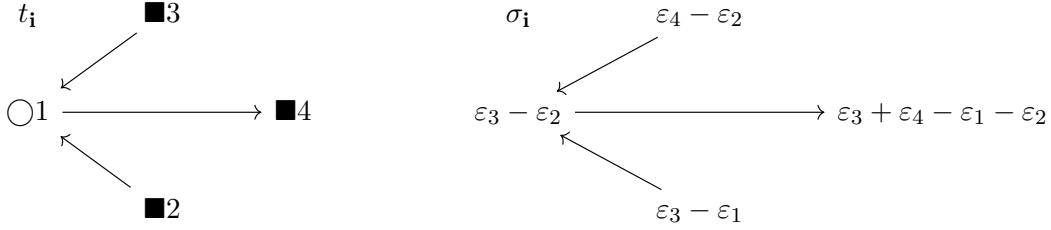
\[\begin{tikzcd}
	{t_\ii} & {\blacksquare 3} &&& {\sigma_\ii} & {\varepsilon_4 - \varepsilon_2} \\
	{\bigcirc 1} && {\blacksquare 4} && {\varepsilon_3-\varepsilon_2} && {\varepsilon_3 + \varepsilon_4 - \varepsilon_1 - \varepsilon_2} \\
	& \blacksquare2 &&&& {\varepsilon_3- \varepsilon_1}
	\arrow[from=1-2, to=2-1]
	\arrow[from=3-2, to=2-1]
	\arrow[from=2-1, to=2-3]
	\arrow[from=2-5, to=2-7]
	\arrow[from=1-6, to=2-5]
	\arrow[from=3-6, to=2-5]
\end{tikzcd}\]

    The fact that $\sigma_\ii$ is a degree configuration, is equivalent to the equality: $\sigma_{\ii,4}=\sigma_{\ii,2}+ \sigma_{\ii,3}$.

\end{example}

It's known that, for $G$ simple and simply laced
and for any $\ii \in R(w)$, the seed $t_\ii$ is injective reachable. See for instance \cite{qin2022bases} and \cite{geiss2011kac}.
Moreover, it's known and not difficult to verify that, in the simply laced case, the cluster structure described above doesn't depend on the reduced expression for $w$. In particular if $\ii,\ii' \in R(w)$, then $t_\ii \sim t_{\ii'}$. To the author's best knowledge, no similar result is known in the non simply-laced case (see \cite[Remark 2.14]{berenstein2005cluster3}). 

\bigskip

\subsection{Technical properties of minors I: vanishing and twist}
 \label{sec:Technical properties of minors}

 This section and the next one are quite technical. The reader can skip them for a first reading, even though the results obtained here are crucial for proving Theorems \ref{equality Levi}, \ref{equality tensor product}.

\bigskip
 
In this section we fix $\alpha, \beta \in \Delta$ and $v,w \in W$.

\begin{lemma}
    \label{vanishing minor s,w}
    The generalised minor $\Delta^{\pia}_{s_\alpha,w}$ vanishes on $U$ if and only if $s_\alpha$ is not in the support of $w$.
\end{lemma}

\begin{proof}
    If $s_\alpha $ is not in the support of $w$, then $\Delta^{\pia}_{s_\alpha,w}=\Delta^{\pia}_{s_\alpha,e}$ because of Lemma \ref{lem:basic prop minors}. The zero locus of this last function is $\overline{s_\alpha B^- s_\alpha B}$, which clearly contains $U$. \\
    Next, suppose that $\Delta^{\pia}_{s_\alpha,w}$ vanishes on $U$. This means that $U \subseteq \overline{s_\alpha B^- s_\alpha B w^\inv}$. Note that $\overline{s_\alpha B^- s_\alpha B w^\inv}$ is a closed, irreducible subvariety of $G$ of codimension 1, and it is $s_\alpha B^- s_\alpha$-stable by left multiplication. Since also $\overline{s_\alpha B^- s_\alpha U} = \overline{s_\alpha B^- s_\alpha B}$ is closed, irredcuible and of codimension 1 in $G$, we have that $U \subseteq \overline{s_\alpha B^- s_\alpha B w^\inv}$ is equivalent to  $\overline{s_\alpha B^- s_\alpha B} = \overline{s_\alpha B^- s_\alpha B w^\inv}$. Then 
    \begin{equation}
    \label{eq1}
        \begin{array}{rl}
             \overline{s_\alpha B^- s_\alpha B} = \overline{s_\alpha B^- s_\alpha B w^\inv} \ifff  &   \overline{s_\alpha w_0 B w_0 s_\alpha B} = \overline{s_\alpha w_0 B w_0 s_\alpha B w^\inv}\\
              \ifff & \overline{ B w_0 s_\alpha B} = \overline{ B w_0 s_\alpha B w^\inv}\\
              \ifff & \overline{ B s_\alpha w_o B}/B = \overline{w B s_\alpha w_0 B }/B.
        \end{array}
    \end{equation}

The last equality is between two closed, irreducible $T$-stable, subvarieties of $G/B$. Recall that the $T$-fixed points of $G/B$ are canonically identified with $W$, and that $$(\overline{ B v B}/B)^T= \{ x \in W \, : \, x \leq v\},$$ where $\leq$ denotes the Bruhat order on $W$. Note also that by the sub-word property of Coxeter groups, $s_\alpha \in \supp(w)$ if and only if $s_\alpha \leq w$. Moreover, $s_\alpha \in \supp(w) \ifff s_\alpha \in \supp(w\inv).$ Finally, recall that right multiplication by $w_0$ reverses the Bruhat order. In particular, by looking at $T$- fixed points, we have that the last equality in \eqref{eq1} implies the first of the following equalities 

\begin{equation*}
    \begin{array}{rl}
       \{v \, : \, v \leq s_\alpha w_0 \}= \{wv \, : \, v \leq s_\alpha w_0 \} \ifff  &  \{ v \, : \, vw_0 \geq s_\alpha \}= \{ wv \, : \, vw_0 \geq s_\alpha \}\\
       \ifff  & \{ vw_0 \, : v \geq s_\alpha \} = \{ wvw_0 \, : \, v \geq s_\alpha \}.
    \end{array}
\end{equation*}

By contradiction, if $s_\alpha \in \supp(w)$, then $w\inv \geq s_\alpha$, hence $w_0 $ belongs to the RHS of the last equality. Looking at the LHS we deduce that $e \geq s_\alpha$, which is a contradiction.

\end{proof}

 The following is a well known result about representation theory of $\SL_2$. See for example \cite[Lemma 2.8]{fomin1999double}. Recall that, if $\beta \in \Phi$ and $n $ is a positive integer, then the \textit{divided power} $X_\beta^{(n)}$ denotes $\frac{X_\beta^n}{n!}$, which is an element of the envelopping algebra of $\lg$.

\begin{lemma}
    \label{E_i, s_i}
    Let $V$ be a representation of $\sl_\alpha$ and $v^- \in V$ (resp. $v^+ \in V$) a lowest (resp. highest) weight vector  of weight $- m$ (resp. $m$) for $m \in \NN$. Then 
    \begin{equation*}
    X_\alpha^{(n)}v^-= \begin{cases}
        0 & \text{if} \quad n >m\\
        \sba^{-1} v^- & \text{if} \quad n =m.
    \end{cases} \quad 
    X_{-\alpha}^{(n)}v^+= \begin{cases}
        0 & \text{if} \quad n >m\\
        \sba v^+ & \text{if} \quad n =m.
        \end{cases}
    \end{equation*}
\end{lemma}

Next, we do some calculations that will be crucial in the following.

\begin{lemma}
    \label{minors, developpement right translation}
    Let  $t \in \CC$, $g \in G$ and $N=\langle - w \pia , \beta^\vee \rangle$. Then $$ \delavw(gx_\beta(t))=    \begin{aligned}
           \begin{cases} 
           \delavw(g) & \text{if} \quad w\inv \beta \in \Phi^+ \\
        \sum_{n=0}^N t^n Q_{N-n}(g) & \text{if} \quad w\inv \beta \in \Phi^-
           \end{cases}
       \end{aligned}$$
       for some $Q_i \in \CC[G]$, such that $Q_0= \Delta^{\pia}_{v, s_\beta w} $ and $Q_N= \delavw$.
\end{lemma}

\begin{proof}
    If $w\inv \beta \in \Phi^+$, the statement follows from Lemma \ref{lem:basic prop minors}. Otherwise,  we have that 
   
          $$ \delavw(gx_\beta(t)) =  \sum_{n \in \NN} t^n \overline{v} \phiam (g X_\beta^{(n)} \overline{w} \vap ).$$
If $w\inv \beta \in  \Phi^-$, then $\overline{w}\vap$ is a lowest weight vector for $\sl_\beta$ of weight $\langle  w \pia , \beta^\vee \rangle $, hence Lemma \ref{E_i, s_i} applies. In particular, for $n > N$ the corresponding factor of the above sum is zero. For $n \in \{0,\dots , N\}$ we set $Q_{N-n}(g)= \overline{v} \phiam (g X_\beta^{(n)} \overline{w} \vap)$. Clearly, $Q_N=\delavw$. Moreover, $X_\beta^{(N)} \overline{w} \vap= \sbb\inv \overline{w} \vap.$ But since $w\inv \beta \in \Phi^-$, then $\overline{w}= \sbb \overline{s_\beta w}$. The lemma follows.
\end{proof}

We have a variation of the previous lemma relative to the action of elements of $U^-$.

\begin{lemma}
    \label{minors, developpement right translation minus}
    Let  $t \in \CC$, $g \in G$ and $N=\langle  w \pia , \beta^\vee \rangle$. Then $$ \delavw(gx_{-\beta}(t))=    \begin{aligned}
           \begin{cases} 
           \delavw(g) & \text{if} \quad w\inv \beta \in \Phi^- \\
        \sum_{n=0}^N t^n Q_{N-n}(g) & \text{if} \quad w\inv \beta \in \Phi^+
           \end{cases}
       \end{aligned}$$
       for some $Q_i \in \CC[G]$, such that $Q_0= \Delta^{\pia}_{v, s_\beta w} $ and $Q_N= \delavw$.
\end{lemma}
\begin{proof}
    Literally the same proof of Lemma \ref{minors, developpement right translation}. Just notice that if $w \inv \beta \in \Phi^+$, then $\overline{w}v_{\pia}^+$ is a highest weight vector for $\sl_\beta$ of weight $\langle w \pia , \beta^\vee \rangle.$
\end{proof}
\begin{lemma}
    \label{minors, developpement left translation plus}
 Let $t \in \CC$, $g \in G$ and $N=\langle v \pia , \beta^\vee \rangle$. Then $$ \delavw( x_\beta(t) g)=    \begin{aligned}
           \begin{cases} 
           \delavw(g) & \text{if} \quad v\inv \beta \in \Phi^- \\
        \sum_{n=0}^N t^n P_{N-n}(g) & \text{if} \quad v\inv \beta \in \Phi^+
           \end{cases}
       \end{aligned}$$
       for some $P_i \in \CC[G]$, such that $P_0= \Delta^{\pia}_{s_\beta v,  w} $ and $P_N= \delavw$.
    
\end{lemma}

\begin{proof}
    One can prove the statement by direct calculation, otherwise we can reduce to Lemma \ref{minors, developpement right translation}. Following \cite{fomin1999double}, let $(-)^\iota $ be the involutive anti-authomorphism of the group $G$ defined by

    $$x_\gamma(t)^\iota=x_\gamma(t) \quad x_{-\gamma}(t)^\iota= x_{-\gamma}(t) \quad h^\iota= h\inv \quad \text{for} \quad \gamma \in \Delta, \,  \, h \in T.$$
    By \cite[Proposition 2.7]{fomin1999double}, we have that for any $\gamma \in \Delta$ and $g \in G$
    
    \begin{equation}
        \label{eq 10}
   \Delta^{\varpi_\gamma}_{v,w}(g)= \Delta^{\varpi_\gamma^*}_{ww_0, vw_0}(g^\iota)
    \end{equation}

    In particular, form Lemma \ref{minors, developpement right translation} we deduce that 

    $$ \delavw(x_\beta(t)g)= \Delta^{\pia^*}_{ww_0, vw_0}(g^\iota x_\beta(t)) =    \begin{aligned}
           \begin{cases} 
           \Delta^{\pia^*}_{ww_0, vw_0}(g^\iota)=\delavw(g) & \text{if} \quad w_0v\inv \beta \in \Phi^+ \\
        \sum_{n=0}^M t^n Q_{M-n}(g^\iota) & \text{if} \quad w_0v\inv \beta \in \Phi^-
           \end{cases}
       \end{aligned}$$
       where $M= \langle -v w_0 \pia^* , \beta^\vee \rangle$. Since $\pia^*=- w_0 \pia$ we have that $M=N$. Then we set $P_n(g)=Q_n(g^\iota)$. Using again \cite{fomin1999double}[Proposition 2.7], we deduce the desired expression for $P_0$ and $P_N$ from the expression of $Q_0$ and $Q_N$ given in Lemma \ref{minors, developpement right translation}. Finally, notice that $$w_0 v\inv \beta \in \Phi^+ \ifff v\inv \beta \in \Phi^-.$$
\end{proof}

\begin{remark}
    If $\langle v \pia , \beta^\vee \rangle=0$, then $s_\beta v \pia= v \pia - \langle v \pia , \beta^\vee \rangle \beta= v \pia$. Hence $P_0$ and $P_N$ agree in this case because of Lemma \ref{lem:basic prop minors}, so there's no ambiguity. For $Q_0$ and $Q_N$ we have the same phenomenon. The convention about the enumeration of the coefficients $P$ is due to the fact that we want it to be compatible with the expansion of $t^N \delavw( x_\beta(t\inv) g)$. The same applies to the coefficients $Q$. 
\end{remark}

The next statement is a special case of the previous lemma.

\begin{lemma}
    \label{minors, developpement left translation} Let $t \in \CC$ and $g \in G$. We have that $$ \delaew(x_\beta(t)g)=    \begin{aligned}
           \begin{cases} 
           \delaew(g) & \text{if} \quad \alpha\neq \beta \\
         \delaew(g) + t \Delta^{\pia}_{s_\alpha,w}(g) & \text{if} \quad \alpha=\beta.
           \end{cases}
       \end{aligned}$$
\end{lemma}

\begin{proof}
    Obvious from Lemma \ref{minors, developpement left translation plus}.
\end{proof}

\begin{remark}[A remark on products]
\label{remark on products}
    We fix a quite obvious convention relative to products. Suppose that two complex, semisimple, simply connected groups $G_1$ and $G_2$ are given. Moreover, suppose that, for $i=1,2$ and $\alpha_i \in \Delta_i$, the data of $T_i$, $B_i \subseteq G_i$
 and of an $\sl_2$-triple $(X_{\alpha_i}, H_{\alpha_i}, X_{-\alpha_i})$ are also given. Then, we automatically make the following choices in $G=G_1 \times G_2$: the torus $T=T_1 \times T_2$, the Borel $B=B_1 \times B_2$. Moreover, for $\alpha \in \Delta=\Delta_1 \sqcup \Delta_2$, the $\sl_2$ triple associated to $\alpha$, in $\lg=\lg_1 \oplus \lg_2$, is the one in $\lg_{1/2}$ accordingly if $\alpha \in \Delta_{1/2}$. 
 
 With this choice, if $\alpha = \alpha_1 \in \Delta_1$, then $\overline{s_{\alpha}}= (\overline{s_{\alpha_1}} ,e)$. If $\alpha = \alpha_2 \in \Delta_2$, then $\overline{s_{\alpha}}= (e, \overline{s_{\alpha_2}})$. Hence, for any $w=(w_1,w_2) \in W=W_1 \times W_2$, we have that $$\overline{w}= (\overline{w_1}, \overline{w_2}).$$The fundamental weights of $T_1 \times T_2$ are of the form $(\varpi_{\alpha_1},0)$ and $(0, \varpi_{\alpha_2})$  and we have 
 \begin{equation}
     \Delta^{(\varpi_{\alpha_1}, 0)}_{(v_1,v_2),(w_1,w_2)}= \Delta^{\varpi_{\alpha_1}}_{v_1,w_1} \otimes 1 \quad \text{and} \quad 
     \Delta^{(0, \varpi_{\alpha_2})}_{(v_1,v_2),(w_1,w_2)}= 1 \otimes \Delta^{\varpi_{\alpha_2}}_{v_2,w_2}.
 \end{equation}

 \end{remark}

\subsection{Technical properties of minors II: algebraic independence}
\label{Technical lemmas for the study of  valuations section}
We study algebraic independence of, the restriction, of some families of generalised minors to some $T$-stable subgroup of $U$. 

\bigskip

Fix $z \in W$ and $\ii \in R(z)$. Let $l=\ell(z)$ and $t=t_\ii$. For $1 \leq j \leq l$, $x_j \in \CC[U(z)]$ denotes the $j$-th cluster variable of the seed $t$. We refer to Section \ref{cluster U(w)} for the definitions and the notation of this section.

\subsection*{Left twist: the case of U(z)}
 We fix $\alpha \in \Delta$ such that $s_\alpha \in \supp(z)$. In this subsection, for $\beta \in \Delta$ and $v,w \in W$, we denote by $D^{\pib}_{v,w}$ the restriction of $\Delta^{\pib}_{v,w}$ to $U(z)$.

\begin{lemma}
\label{delta alpha min polinomio nei precedenti U(w)}
    The following equality holds $$ \displaystyle D^{\pia}_{s_\alpha, z_{\leq \alpha^{\min}}^{-1}}= \prod_{j=1}^{\alpha^{\min}-1} (D^{\varpi_{i_j}}_{s_\alpha, z_{\leq j }^{-1}})^{c_{j}} $$
   for some non-negative integers $c_j$. Moreover, for $j < \alpha^{\min}$, we have that $ D^{\varpi_{i_j}}_{s_\alpha, z_{\leq j }^{-1}}= D^{\varpi_{i_j}}_{e, z_{\leq j }^{-1}}$.
\end{lemma}

\begin{proof}
    Note that $z_{\leq \alpha ^{min}}^{-1}= z' s_\alpha$ where $z'=s_{i_1} \dots s_{(i_{\alpha^{\min}-1})}$, or $z'=e$ if $\alpha^{\min}=1$. We apply Theorem \ref{relations minors FZ} with $w=z'$ and $v=e$.
    
    Note that $s_\alpha \not \in \supp(z')$, hence $z'\pia=\pia$ and so $\Delta^{\varpi_\alpha}_{e,z'}=\delaee$ by Lemma \ref{lem:basic prop minors}. Moreover, $D^{\pia}_{e,e}=1$. Similarly, from Lemma \ref{vanishing minor s,w} we deduce that $D^{\pia}_{s_\alpha , z'} = D^{\pia}_{s_\alpha , e}= 0 $. So, Theorem \ref{relations minors FZ} specialises to the following identity:
    $$ \displaystyle D^{\pia}_{s_\alpha, z_{\leq \alpha ^{min}}^{-1}}= \prod_{\beta \in \Delta \setminus \{\alpha\}} (D^{\pib}_{e,z'})^{-a_{\beta,\alpha}}.$$
    But for $\beta \neq \alpha$, $(D^{\pib}_{e,z'})^{-a_{\beta,\alpha}}=(D^{\pib}_{s_\alpha,z'})^{-a_{\beta,\alpha}}$ because of Lemma \ref{lem:basic prop minors}. Moreover, if $s_\beta \notin \supp(z')$, $(D^{\pib}_{s_\alpha,z'})^{-a_{\beta,\alpha}}=1$.
    If $s_\beta \in \supp(z')$ and 
    $$k_\beta= \max \{ j \in \{1 ,\dots , \alpha^{\min}-1\} \, : \, i_j=\beta \},$$ by Lemma \ref{lem:basic prop minors} we have that $$(D^{\pib}_{s_\alpha,z'})^{-a_{\beta,\alpha}}= (D^{\pib}_{s_\alpha,z_{\leq k_\beta}^{-1}})^{-a_{\beta,\alpha}}.$$
    In particular, for $j < \alpha^{\min}$, setting \begin{equation*}
        c_j= \begin{cases}
            -a_{\beta,\alpha} & \text{if} \quad j=k_\beta \quad \text{for some} \quad \beta \neq \alpha.\\
            0 & \text{otherwise}
             \end{cases}
    \end{equation*}
    we have proved the first part of the statement. The last part is obvious from Lemma \ref{lem:basic prop minors} and the definition of $\alpha^{\min}.$
\end{proof}

\begin{lemma}
    \label{delta s w alg indip U(w)}
    For any $k \leq l $, the following equality of subfields of $\CC(U(w))$ holds:
    $$\CC\biggl(x_{\alpha^{\min}} \, , \, D^{\varpi_{i_j}}_{s_\alpha, z_{\leq j}^{-1}} \, : \, j \leq k, j \neq \alpha^{\min}\biggr)=\CC\biggl(x_{\alpha^{\min}} \, , \, x_j \, : \, j \leq k \biggr). $$
    In particular, the functions $\{D^{\varpi_{i_j}}_{s_\alpha, z_{\leq j}^{-1}} \, : \, j \leq l, j \neq \alpha^{\min} \}$ are algebraically independent.
\end{lemma}

\begin{proof}
    Recall that $x_j=D^{\varpi_{i_j}}_{e,z_{\leq j}^{-1}}$. We prove the equality by induction on $k$. For $k=0$ there is nothing to prove. Suppose that the equality has been proved for $k$ and call $$\FF_k=\CC\biggl(x_{\alpha^{\min}} \, , \, D^{\varpi_{i_j}}_{s_\alpha, z_{\leq j}^{-1}} \, : \, j \leq k, j \neq \alpha^{\min}\biggr)=\CC\biggl(x_{\alpha^{\min}} \, , \, x_j \, : \, j \leq k \biggr).$$
    If $k+1= \alpha^{\min}$, there is nothing to prove. Similarly, if $i_{k+1} \neq \alpha$, then by Lemma \ref{lem:basic prop minors} 
    $$ D^{\varpi_{i_{k+1}}}_{s_\alpha,z_{\leq k+1}^{-1}}=D^{\varpi_{i_{k+1}}}_{e,z_{\leq k+1}^{-1}},$$ so the desired equality is trivial. Next, suppose that $i_{k+1}= \alpha$ and $k+1 \neq \alpha^{\min}$. Applying Theorem \ref{relations minors FZ} with $v=e$ and $w= z_{\leq k}^{-1}$,  we get an equality that can be read as 
    
    \begin{equation}
        \label{eq 11}
    f D^{\varpi_{i_{k+1}}}_{s_\alpha,z_{\leq k+1}^{-1}} = g D^{\varpi_{i_{k+1}}}_{e,z_{\leq k+1}^{-1}} + \varphi.
    \end{equation}
    We analyse the three terms $\varphi, f, g$. 
    
    Note that $$ \varphi= \prod_{\beta \in \Delta \setminus \{ \alpha \} } (D^{\pib}_{e, z_{\leq k}^{-1}})^{-a_{\beta,\alpha}}.$$
    For $\beta \in \Delta \setminus \{ \alpha \}$, if $s_\beta \not \in \supp (z_{\leq k}^{-1})$, then $D^{\pib}_{e, z_{\leq k}^{-1}}=1$. Otherwise, let $k_\beta= \max  \{ 1 \leq j \leq k \, : \, i_j=\beta \}$. Then, 
    
    $$D^{\pib}_{e, z_{\leq k}^{-1}}= D^{\pib}_{e, z_{\leq k_\beta}^{-1}}=x_{k_\beta}. $$
Hence $\varphi \in \FF_k$.

Next, we look at $g$. Let $t= (k+1)^-$, which is not zero since $k+1 \neq \alpha^{\min}$. Lemma \ref{lem:basic prop minors} implies that:
$$g= D^{\varpi_{i_{k+1}}}_{s_\alpha,z_{\leq k}^\inv}=D^{\varpi_{i_t}}_{s_\alpha,z_{\leq t}^\inv}.$$
Since  $s_\alpha \in \supp(z_{\leq t}^{-1})$, then $g \neq 0$ by Lemma \ref{vanishing minor s,w} and statement $5$ of Lemma \ref{lem:basic prop minors}. Moreover, if $t\neq \alpha^{\min}$, we clearly have $g \in \FF_k$. If $t=\alpha^{\min}$, we deduce that $g \in \FF_k$ from Lemma \ref{delta alpha min polinomio nei precedenti U(w)}.

Finally, we consider $f$. Notice that 
$$f=D^{\varpi_{i_{k+1}}}_{e, z_{\leq k}\inv}=D^{\varpi_{i_t}}_{e,z_{\leq t}^\inv}=x_t.$$
In particular, if $t=\alpha^{\min}$, then $f=x_{\alpha^{\min}} \in \FF_k$. In the other case, it's obvious that $f \in \FF_k$. Moreover, $f$ is clearly non-zero since it is a cluster variable of $t$.

Since $f,g, \varphi \in \FF_k$ and $f,g$ are non-zero, the equality of fields for $k+1$ follows by induction from \eqref{eq 11}.

\bigskip

 For the algebraic independence, note that the functions $\{ x_j \, : \, j \leq l\}$ are the cluster variables of $t$, which are algebraically independent. The statement follows from the well definiteness of the transcendence degree.  
\end{proof}

\subsection*{Left twist: the case of $U(z)_{\setminus \alpha}$ }

For a $T$-stable subgroup $H$ of $U$ and $\alpha \in \Delta$, we set $$H_{\setminus \alpha} = H \cap s_\alpha U s_\alpha.$$ 
We fix $\alpha \in \Delta \cap \Phi(z)$. In particular, $s_\alpha \in \supp(z)$. In this subsection, for $\beta \in \Delta$ and $v,w \in W$, we denote by $D^{\pib}_{v,w}$ the restriction of $\Delta^{\pib}_{v,w}$ to $U(z)_{\setminus \alpha}$.

\begin{lemma}
\label{delta alpha min polinomio nei precedenti U(w) alpha}
    The following equality holds $$ \displaystyle D^{\pia}_{s_\alpha, z_{\leq \alpha^{\min}}^{-1}}= \prod_{j=1}^{\alpha^{\min}-1} (D^{\varpi_{i_j}}_{s_\alpha, z_{\leq j }^{-1}})^{c_{j}}$$
    for some non-negative integers $c_j$. Moreover, for $j < \alpha^{\min}$, we have that $ D^{\varpi_{i_j}}_{s_\alpha, z_{\leq j }^{-1}} = D^{\varpi_{i_j}}_{e, z_{\leq j }^{-1}}$.
\end{lemma}

\begin{proof}
    This is an obvious consequence of Lemma \ref{delta alpha min polinomio nei precedenti U(w)}.
\end{proof}

\begin{lemma}
     \label{delta s w alg indip U(w) alpha}
   The functions $\{D^{\varpi_{i_j}}_{s_\alpha, z_{\leq j}^{-1}} \, : \, j \leq l, j \neq \alpha^{\min} \}$ are algebraically independent.
\end{lemma}

\begin{proof}
   Lemma \ref{minors, developpement left translation plus} implies that  $\Delta^{\pib}_{s_\alpha,w}$ is invariant by left multiplication of $U(s_\alpha)$. Since the product induces an isomorphism $U(s_\alpha) \times U(z)_{\setminus \alpha} \simeq U(z)$ by \cite{Hum:refgroup}[Proposition 28.1], we deduce the statement from Lemma \ref{delta s w alg indip U(w)}.
\end{proof}

\subsection*{Right twist}
We fix $\alpha \in \Delta \cap \Phi(z)$. We set \begin{equation}
\label{eq:alpha minus}
    \alpha(-)= \min \{ k \, : \, z_{\leq k} \alpha \in \Phi^- \}.
\end{equation} 

\begin{lemma}
    \label{right translation alpha(-)}
   We have that
    \begin{enumerate}
        \item $\{ k \, : \, z_{\leq k}\alpha \in \Phi^- \}= \{ \alpha(-) , \dots , l\}.$
        \item $s_\alpha z_{\leq \alpha(-)}^{-1}= z_{\leq( \alpha(-) -1)}^{-1}$.
        \item  $\langle - z_{\leq \alpha(-)}^{-1} \varpi_{i_{\alpha(-)}}, \alpha^\vee \rangle = 1.$
        \item The expression $\ii'$ obtained from $\ii$ by deleting the term $i_{\alpha(-)}$ is a reduced expression for $z s_\alpha$.
     \end{enumerate}

 \end{lemma}
    
\begin{proof}
\begin{enumerate}
    \item By the definition of $\alpha(-)$, we clearly have that $z_{\leq k}\alpha \in \Phi^+$ for $k < \alpha(-)$. Moreover, since $\ii$ is a reduced expression, from the definition of $z_{\leq k}$ it follows that $\Phi(z_{\leq k}) \subseteq \Phi(z_{\leq k+1})$. Hence, for $k \geq \alpha(-)$ we have that $z_{\leq k}
    \alpha \in \Phi^-.$
    
    \item Since $z_{\leq (\alpha(-)-1)} \alpha \in \Phi^+$ and $s_{i_{\alpha(-)}}z_{\leq (\alpha(-)-1)} \alpha \in \Phi^-$, we have that 
    \begin{equation}
        \label{eq 8}
    z_{\leq (\alpha(-)-1)} \alpha = i_{\alpha(-)},
    \end{equation}
    where we assume that $z_{\leq (\alpha(-)-1)}=e$ if $\alpha(-)=1$. Then the identity follows from the fact that, for any $w \in W$ and $\beta \in \Phi$, we have that $w s_\beta w\inv =s_{w \beta}.$
    \item Using \eqref{eq 8}, we can compute that 
    \begin{equation*}
        \begin{array}{r l}
              \langle - z_{\leq \alpha(-)}^{-1} \varpi_{i_{\alpha(-)}}, \alpha^\vee \rangle = &  \langle -  \varpi_{i_{\alpha(-)}}, (z_{\leq \alpha(-)}\alpha)^\vee \rangle \\
             = & \langle -  \varpi_{i_{\alpha(-)}}, (s_{i_{\alpha(-)}}i_{\alpha(-)})^\vee \rangle \\
             = &  \langle -  \varpi_{i_{\alpha(-)}}, (-i_{\alpha(-)})^\vee \rangle \\
             = & 1.
        \end{array}
        \end{equation*}

        \item Using the second statement (and paying attention if $\alpha(-)=1$ or $\alpha(-)=l$),  we have that  $$ \displaystyle s_{i_l} \dots s_{i_{\alpha(-)+1}} s_{i_{\alpha(-)-1}} \dots s_{i_1} = s_{i_l} \dots s_{i_{\alpha(-)+1}}s_{i_{\alpha(-)}} s_{i_{\alpha(-)-1}} \dots s_{i_1} s_{\alpha}= z s_{\alpha}. $$
   Hence, $\ii'$ is an expression for $zs_\alpha$. Since $\ell(zs_\alpha)=\ell(z)-1$, we deduce that $\ii'$ is reduced.
\end{enumerate}
\end{proof}

Here, for $\beta \in \Delta$ and $v,w \in W$, we denote by $D^{\pib}_{v,w}$ the restriction of $\delbvw$ to $U(z)$. For $k \in [l]$, we define $f_k \in \CC[U(z)]$ as \begin{equation*}
    f_k= \begin{cases}
        D^{\varpi_{i_k}}_{e, z_{\leq k}^{-1}} & \text{if} \quad (z_{\leq k}) \alpha \in \Phi^+\\
         D^{\varpi_{i_k}}_{e, s_\alpha z_{\leq k}^{-1}} & \text{if} \quad (z_{\leq k}) \alpha \in \Phi^-
    \end{cases}
\end{equation*}

\begin{lemma}
    \label{f alpha(-) = }
    In the above setting, $f_{\alpha(-)}= f_{\alpha(-)^-}$. We use the convention that $f_0=1.$
\end{lemma}

\begin{proof}
    Let $k= \alpha(-)$. By statement 2 of Lemma \ref{right translation alpha(-)} we have that 

     \begin{equation*}
        \begin{array}{r l}
             s_\alpha z_{\leq k}^{-1} \varpi_{i_k}= & z_{\leq k-1}^{-1} \varpi_{i_k} \\
             = & s_{i_1} \dots s_{i_{(k-1)}} \varpi_{i_k}\\
             = & s_{i_i} \dots s_{i_{k^-}} \varpi_{i_k}\\
             = & z_{\leq k^-} \varpi_{i_k}
        \end{array}
    \end{equation*}
    where $z_{\leq k^-}=e$ if $k^-=0.$ The lemma follows form Lemma \ref{lem:basic prop minors}.
\end{proof}

\begin{lemma}
    \label{f alg indip}
If $- z \alpha \in  \Delta$, then $\{f_k \, : \, k \neq \alpha(-) \}$ are algebraically independent.
\end{lemma}
\begin{proof}
    If $z \alpha = - \beta $, with $\beta \in \Delta$, it means that $z s_\alpha = s_\beta z \leq_L z$, where $\leq_L$ denotes the \textit{left} weak order on $W$. 
    In particular, $\Phi(z s_\alpha) \subseteq \Phi(z)$, which implies that $U(zs_\alpha) \subseteq U(z).$ For $k \in \{1, \dots, l-1\}$ let $x'_k \in \CC[U(z s_\alpha)]$ be the $k$-th cluster variable of the seed $t_{\ii'}$, where $\ii'$ is the reduced expression in Lemma \ref{right translation alpha(-)}. Clearly, for $k < \alpha(-)$, $f_k$ restricts to $x'_k$, while using the second statement of Lemma \ref{right translation alpha(-)} one deduces (analogously to the proof of the fourth statement of Lemma \ref{right translation alpha(-)}) that, for $k > \alpha(-)$, $f_k$ restricts to $x'_{k-1}$. Since the cluster variables $x_k'$ are algebraically independent in $\CC[U(zs_\alpha)]$, the lemma follows.
\end{proof}

\section{Application to G}
\label{application to G section}
Let $G$ be a semi-simple, simply connected, complex algebraic group. We show here a simple application, of the minimal monomial lifting, to $G$. This example is designed to make the equality between $\CC[G]$ and the upper cluster algebra constructed via the minimal monomial lifting to fail (see Remark \ref{rem:equality fail G}). It may be useful, for the reader, to have a look at this example to understand the strategy of Section \ref{equality Levi} and \ref{equality tensor product}.
 We use the notation of Section \ref{Preliminaries alg groups section}. 

\bigskip
 
Let $(-)^T : G \longto G$ be the \textit{transpose}, which is the involutive anti-automorphism  considered in \cite[Section 2.1]{fomin1999double}. In particular, it is the only anti-automorphism of algebraic groups defined by
$$ x_\alpha(t)^T=x_{-\alpha}(t) \quad  h^T=h \quad \text{for any} \, \, \alpha \in \Delta, \, \, t \in \CC, \, \, h \in T.$$
This restricts to an anti-isomorphism between $U^-$ and $U$. If $\ii \in R(w_0)$, let $t_\ii^-$ be the seed of $\CC[U^-]$ obtained by "transposing" the seed $t_\ii$ of $\CC[U]$. Clearly, $\clu(t_\ii^-) = \uclu(t_\ii^-)=\CC[U^-]$. Moreover, by \cite[Proposition 2.7]{fomin1999double}, the cluster variables of $t_\ii^-$ are all obtained by restricting functions of the form $\Delta^{\pia}_{w,e}$, with $s_\alpha \in \supp(w)$,  to $U^-$. 

\bigskip

Fix $\ii',\ii \in R(w_0)$ and call 
$$t_\ii=(I_{uf}, I_{sf}, I_{hf}, B, x) \quad t_{\ii'}^-=(I^-_{uf}, I^-_{sf}, I^-_{hf}, B^-, x^-) \quad t_{\ii'}^-|t_\ii= (J_{uf}, J_{sf},J_{hf}, C, z),$$
where $t_{\ii'}^-|t_\ii$ is the disjoint union of $t_{\ii'}^-$ and $t_\ii$, as defined in Section \ref{sec:disj union}. By Lemma \ref{dijoint union tensor product}, we have that $\uclu(t_{\ii'}^-|t_{\ii})= \CC[U^- \times U]$. This is a factorial $\CC$-algebra of finite type. 

Let $D= \Delta$, $Y=U^- \times U$ and 
$$\begin{array}{r c l}
   \phi: T \times U^- \times U & \longto & G\\
   (h, u^- , u ) & \longmapsto & u^-hu.
\end{array}
$$
Moreover, for $\alpha \in D$, we set $X_\alpha= \delaee$.

\begin{lemma}
    \label{lem:G suitable lift}
    The triple $(G,\phi, X)$ is suitable for $D$-lifting.
\end{lemma}

\begin{proof}
    Since $G$ is semi-simple and simply connected, $\CC[G]$ is a factorial domain, hence normal. Moreover, $G$ is of finite type. Condition 1 of Definition \ref{def: suitable lifting} then clearly holds. Condition 2 holds because of Lemma \ref{lem:minors irred}. Moreover, $\phi$ is an open embedding with image $G_0=B^-B$ and, for any $\alpha \in D$, $\phi^*(X_\alpha)= \varpi_\alpha \otimes 1$ by Lemma \ref{lem:basic prop minors}. Hence condition 3 of Definition \ref{def: suitable lifting} holds.
\end{proof}
 \begin{prop}
     \label{minimal mon lifting G}
     For any $\ii',\ii \in R(w_0)$ we have $$ \uclu\bigl(\lif(t_{\ii'}|t_\ii)\bigr)= \CC[G_0] \quad \text{and} \quad \uclu\bigl(\lif(t_{\ii'}|t_\ii)^D\bigr) \subseteq \CC[G]. $$
 \end{prop}
\begin{proof}
   Because of Lemma \ref{lem:G suitable lift} and Lemma \ref{factorial upper, coprimality} we can apply Theorem \ref{minimal monomial lifting theo}.
\end{proof}

It's easy to compute the minimal lifting matrix $\nu$ and the corresponding cluster $\lif z$ in this example. Let $\chi : J \longto \Delta$ be the map defined as follows: for $k \in I$ (resp. $k \in I^-$), $\chi(k)= \alpha$ if and only if the minor defining $x_k$ (See \eqref{x for U(w)}) is of the form $\delaew$ (resp. $\Delta^{\pia}_{w,e})$.

\begin{lemma}
    \label{computation z(nu) for G}
    For any $j \in J$ and $\alpha \in D$, we have $\nu_{\alpha,j}= \delta_{\alpha, \chi(j)}$. Moreover, the variable $\lif z_j$ is the generalized minor defining $z_j$.
\end{lemma}

\begin{proof}
    Suppose that $j \in I$. If $j \in I^-$ the proof is similar. Let $\delbew$ be the minor defining $x_j$. In particular, $\chi(j)=\beta$. Using Lemma \ref{lem:basic prop minors} and the definition of $z_j$ we have that, for $x=(h,u^-,u) \in T \times Y$
    $$\delbew \bigl(\phi(x)\bigr)=\pib(h)\delbew(u)= \delbee\bigl(\phi(x)\bigr) \bigl(1 \otimes z_j(\phi(x)\bigr).$$
    In particular, identifying $1 \otimes z_j$ as a rational function on $G$ via $\phi$, we have that 
    \begin{equation}
    \label{eq:app G1}
        \delbew=X_\beta (1 \otimes z_j).
    \end{equation}
    We claim that for any $\alpha \in D$, $\val_\alpha(\delbew)=0$. Thanks to \eqref{eq:app G1}, this completes the proof. Suppose by contradiction that there exists $\alpha \in D$ such that $\val_\alpha(\delbew)>0.$
   Since $\delbew$ is irreducible in $\CC[G]$ because of Lemma \ref{lem:minors irred}, and so is $\delaee$, this means that $\delbew= \psi \delaee$, for some $\psi \in \CC[G]^*$. Since $G$ is semi-simple, $\psi$ is constant.  By statement three of Lemma \ref{lem:basic prop minors}, we deduce that 
   $$(\pib, w\pib)=(\pia, \pia).$$
   So $\alpha=\beta$ and $w \pib=\pib.$ The latter equality gives a contradiction since $s_\beta \in \supp(w).$
\end{proof}

From the previous lemma, we deduce that for $j \in I$, $\lif z_j$ is a generalised minor of the form $\delaew$. Similarly, if $j \in I^-$, $\lif z_j$ is a generalised minor of the form $\Delta^{\pia}_{w,e}$. In both cases, $s_\alpha \in \supp(w)$. 

\begin{prop}
\label{not equality G}
    Suppose that $\ii$ and $\ii'$ start with the same simple root, that is $i_1=\alpha=i'_1$. Then, the upper cluster algebra $\uclu( \lif (t_{\ii'}^-|t_\ii)^D)$ is strictly contained in $\CC[G]$.
\end{prop}

\begin{proof}
    
 Form the definition of the seeds $t_\ii$, $t_{\ii'}^-$ and the previous lemma, we have that $\Delta^{\pia}_{e,s_\alpha}$ and $\Delta^{\pia}_{s_\alpha,e}$ are cluster variables of the initial seed $\lif (t_{\ii'}^-|t_\ii)$. But then, by Theorem \ref{relations minors FZ}, we have that

\begin{equation}
    \label{eq:31}\Delta^{\varpi_\alpha}_{s_\alpha,s_\alpha}= \frac{ \Delta^{\pia}_{s_\alpha, e}\Delta^{\pia}_{e,s_\alpha}+ \prod_{\beta \in \Delta \setminus\{\alpha\}}(\delbee)^{-a_{\beta,\alpha}} }{\delaee}
\end{equation}

which proves that $ \Delta^{\varpi_\alpha}_{s_\alpha,s_\alpha}$ is an element of $\CC[G] \setminus  \uclu\bigl(\lif (t_{\ii'}^-|t_\ii)^D\bigr)$. In fact, on the RHS of \eqref{eq:31} we have an element of $\Li(\lif (t_{\ii'}^-|t_\ii))$ which is not in $\Li(\lif (t_{\ii'}^-|t_\ii)^D)$.  Alternatively, on the LHS of \eqref{eq:31} we have an irreducible element of $\CC[G]$, in particular $\val_\alpha(\Delta^{\varpi_\alpha}_{s_\alpha,s_\alpha}) \geq 0$ (it's easy to see that it's actually zero). But by looking at the RHS, we see that $\cval_\alpha(\Delta^{\varpi_\alpha}_{s_\alpha,s_\alpha})=-1$. The statement follows from Proposition \ref{conditions equality minimal lifting}.

\end{proof}

\begin{remark}
    \label{rem:equality fail G}

Note that, for any $\ii,\ii' \in R(w_0)$, the seed $\lif (t_{\ii'}^-|t_\ii)$ is not one of the seeds constructed in \cite{berenstein2005cluster3} for the open double Bruhat cell. Moreover, there is no $T$-action on $G$ that makes the map $\phi$ equivariant. Hence, the equality between $\uclu(\lif (t_{\ii'}^-|t_\ii)^D)$ and $\CC[G]$ is never reached, for any pair $\ii,\ii' \in R(w_0)$, because of Corollary \ref{equality minimal lift torus action}. Finally, a similar construction can be done for the spherical homogeneous space $G/\hG$ if and only if the pair $\hG \subseteq G$ is of minimal rank. This will be developed in the future.

\end{remark}

\section{Monomial lifting for branching problems}
\label{monomial lift branchig section}
    In this section we explain how to apply the minimal monomial lifting technique to study some branching problems.

\subsection{The branching scheme}
\label{branching scheme subsection}

Let $G$ be a semisimple, simply connected complex algebraic group and let $\hG$ be a connected, reductive subgroup of $G$. Fix maximal tori and Borel subgroups $T$, $B$ of $G$ and $\hT, \hB$ of $\hG$ such that $T \cap \hG= \hT$ and $B \cap \hG=\hB$. Let $\rho: X(T) \longto X(\hT)$ be the restriction map.
Consider the $U^- \times \hU \times \hG$ action on $G \times \hG$  defined by $$(u^-, \wh u , \wh s) \cdot (g, \hg)=(u^-g \wh s \inv, \wh s \hg \wh u \inv).$$
By \eqref{Peter Weyl thm}, we clearly have that 
\begin{equation}
\label{Petre Weyl Branchin}
    \displaystyle \CC[G \times \hG]^{U^- \times \hU \times \hG} \simeq \bigoplus ( V(\lambda) \otimes V(\hlambda)^*)^\hG \simeq \bigoplus \Hom(V(\hlambda), V(\lambda))^\hG,
\end{equation} 
where the sums run over $(\lambda, \hlambda) \in X(T)^+ \times X(\hT)^+.$ Recall that, if $\lambda \in X(T) \setminus X(T)^+$ (resp. $\hlambda \in X(\hT) \setminus X(\hT)^+$), then we set $V(\lambda)=0$ (resp. $V(\hlambda)=0$). Consider the $T \times \hT$ action on $G \times \hG$ defined by $$ (h , \wh h) \cdot (g , \hg)= (hg, \hg \wh h).$$
The induced action on $\CC[G \times \hG]$ stabilises $\CC[G \times \hG]^{U^- \times \hU \times \hG}$ which is then $X(T) \times X(\hT)$-graded. It's easy to verify that the homogeneous components of this graduation correspond to the decomposition given in \eqref{Petre Weyl Branchin}. 
In particular, for any $(\lambda, \hlambda) \in X(T) \times X(\hT)$, we have
$$ \CC[G \times \hG]^{U^- \times \hU \times \hG}_{\lambda,\hlambda}= (V(\lambda) \otimes V(\hlambda)^*)^\hG.$$ 
If  $\lambda$ and $\hlambda$ are dominant, then $(V(\lambda) \otimes V(\hlambda)^*)^\hG$ is the \textit{multiplicity space} of the representation $V(\hlambda)$ in $V(\lambda)$, the latter considered as a representation of $\hG$.

\bigskip

Consider now the action of $U^- \times \hU$ on $G$ given by $(u^-, \wh u)\cdot g=(u^-g \wh u \inv)$. The action of $T \times \hT$ on $G$, given by $(h, \wh h) \cdot g= h g \wh h$, stabilised $\CC[G]^{U^- \times \hU}$ and thus induces a $X(T)\times X(\hT)$-graduation on it.

\begin{lemma}
    \label{two models for branching}. 
The product $G \times \hG \longto G$ induces a $T \times \hT$-equivariant isomorphism between $\CC[G \times \hG]^{U^- \times \hU \times \hG}$ and $\CC[G]^{U^- \times \hU}$.
\end{lemma}

\begin{proof}
    The inverse is the map $G \longto G \times \hG$ sending $g $ to $(g,e)$. The proof is straightforward.
\end{proof}

From now on, we identify the two graded algebras appearing in the previous lemma. We define the \textit{branching algebra} $\Br(G,\hG)$ to be the $X(T) \times X(\hT)$-graded algebra

\begin{equation}
\label{branching algebra}
\Br(G,\hG):=\CC[G \times \hG]^{U^- \times \hU \times \hG} = \CC[G]^{U^- \times \hU}.
\end{equation}

\begin{lemma}
\label{branching algebra factorial f.type}
    The branching algebra $\Br(G, \hG)$ is factorial and of finite type. An element $f \in \Br(G, \hG)$ is irreducible if and only if it is irreducible as an element of $\CC[G].$
\end{lemma}

\begin{proof}
    The group $U^- \times \hU $ is connected, has no non-trivial character. Moreover, $\CC[G]$ is factorial since $G$ is semisimple and simply connected. Then, the factoriality of $\Br(G,\hG)$ and the last statement follow from Lemma \ref{valuation on invariants or above}.
     Moreover, it is well known that $\CC[G \times \hG]^{U^- \times \hU} = \CC[G]^{U^-} \otimes \CC[\hG]^\hU$ is of finite type. Since $\hG$ is reductive, then $\Br(G , \hG)$ is of finite type by a well known theorem of Hilbert.
\end{proof}

Then, we define the \textit{branching scheme} $\lX(G, \hG)$ as the normal, integral and affine $T \times \hT$-variety \begin{equation}
    \label{def X for branching}
    \lX(G, \hG):= \Spec (\Br(G, \hG)).
\end{equation}

When $G$ and $\hG$ are fixed, we drop the dependence on $G, \hG$ and just write $\lX$ for $\lX(G, \hG)$ and $\Br$ for $\Br(G,\hG).$

\subsection{A first suitable for lifting structure}

In the setting of the previous section, we want to prove that $\lX=\lX(G, \hG)$ is homogeneously suitable for lifting. We fix \begin{equation}
\label{D lifting branching}
    D = \Delta
\end{equation}
and for any $\alpha \in D$, we set
\begin{equation}
\label{T lifting branching}
    X_\alpha = \delaee \in  \CC[G]^{U^- \times \hU} = \Br(G,\hG).
\end{equation}
Let $$\Omega= \lX_{\prod_{\alpha \in \Delta} X_\alpha}= \lX \setminus \bigcup_{\alpha \in D} V(X_\alpha)$$ be the principal open subscheme of $\lX$ defined by the non-vanishing the $X_\alpha$.

\bigskip

Moreover, let $\widetilde Y:= \Spec (\CC[U]^\hU)$, where $\hU$ acts on $U$ by right multiplication. The $\hT$ action on $U$ defined by $\wh h \cdot u= \wh h \inv u \wh h$ defines a structure of $\hT$-scheme on $\widetilde Y$. Similarly, the action of $T \times \hT $ on $T \times U$ defined by 
\begin{equation}
    \label{ T x hT action on T x U}
    (h , \wh h)\cdot (t,u)= (ht \widehat h \, , \, \wh h\inv u \wh h)
\end{equation}
factors through a $T \times \hT$ action on $T \times \widetilde Y$.
 By definition, the induced $T $ action on $T \times \widetilde Y$ is the left multiplication of $T$ on the $T$-component.

\begin{lemma}
    \label{T x Y tilde open embedding branching}
The product map $T \times U \longto G$ induces an open embedding $\widetilde \phi : T \times \widetilde Y \longto \lX$, with image $\Omega$. Moreover, the following hold.
\begin{enumerate}
    \item For any $\alpha \in D$, $\widetilde \phi^* (X_\alpha)= \pia \otimes 1$.
    \item The map $\widetilde \phi$ is $T \times \hT$-equivariant.
    \end{enumerate}
\end{lemma}

\begin{proof}
   By \cite[Cororllary 2.5]{fomin1999double}, the non vanishing locus in $G$ of all the generalised minors $\delaee$ is $G_0$. In particular, $\Omega = \Spec (\CC[G_0]^{U^- \times \hU})$. Recall that the product induces an isomorphism $U^- \times T \times U \simeq G_0$. Then, the proof consists of some immediate verifications. Note that the inverse of $\widetilde \phi : T \times \widetilde Y \longto \Omega$ is induced by the map $G_0 \longto T \times U$ sending $x$ to $([x]_0, [x]_+).$
\end{proof}

\begin{prop}
    \label{prop:branch scheme suitable tilde}
    The triple $(\lX, \widetilde \phi , X)$ is homogeneously-suitable for $D$-lifting.
\end{prop}

\begin{proof}
    It is a direct consequence of Lemmas \ref{branching algebra factorial f.type}, \ref{lem:minors irred} and \ref{T x Y tilde open embedding branching}. 
\end{proof}

The following question naturally arises.

\begin{question}
    \label{question: Y tilde cluster} For which pairs $( G, \hG)$ the algebra $\CC[U]^\hU$ has a cluster structure?
\end{question}

\subsection{Lifting graduation for the Branching algebra}
\label{sec:lifting grad branching}

Before describing some pairs $(G, \hG)$ for which we can answer Question \ref{question: Y tilde cluster}, we take a closer look at the information given by the lifting graduation on $\Br$.

\bigskip

Note that, by Lemma \ref{lifting graduation torus action}, the $X(T)$-graduation on $\Br$ is a lifting graduation of the pole filtration on $\CC[U]^\hU$ induced by the almost polynomial ring $(\lX, \val, X, \widetilde \phi^*)$. 
    Note that the conjugation of $T$ on $U$: 
    $$ t \cdot u= t\inv u t \quad \text{for} \quad t \in T, \, u \in U$$
    factors through an action of $\hT$ on $\widetilde Y$.

\begin{lemma}
    The pole filtration on $\CC[U]^\hU$ is $X(\hT)$-graded.
\end{lemma}

\begin{proof}
    Identify $\CC(\lX)$ and $\CC(T \times \widetilde Y)$ using $\widetilde \phi$. Recall that $\ZZ^D$ is canonically identified with $X(T)$ via the $\ZZ$-linear map defined by $e_\alpha \longmapsto \pia.$ Fix $\lambda \in \ZZ^D$, then 
    \begin{equation}
    \label{eq:30}
        f \in \CC[\widetilde Y ]_\lambda \ifff X^\lambda (1 \otimes f) \in \Orb_\lX(\lX),
    \end{equation}
    where $\CC[\widetilde Y ]_\lambda$ is the $\lambda$-component of the lifting graduation. The map $\CC(\widetilde Y) \longmapsto \CC(\lX)$ defined by $f \longmapsto 1 \otimes f$ is $\hT$-equivariant. Since, for any $\alpha \in  D$, $X_\alpha $ is $\hT$ semi-invariant, we deduce from \eqref{eq:30}  that $\CC[\widetilde Y]_\lambda$ is $\hT$-stable.
\end{proof}

The $X(T) \times X(\hT)$-graduation on $\Br$ is NOT a $X(\hT)$-graded lifting graduation. Indeed, the map $\widetilde Y \longto \lX$ defined by $\widetilde y \longmapsto \widetilde \phi ( (e, \widetilde y))$ is not $\hT$-equivariant (see Lemma \ref{lem: H graded pole filt equivariance}). We can fix the situation by twisting the $T\times \hT$-action on $\lX.$

\bigskip

Let $\lX^*$ be the scheme $\lX$ endowed with the twisted $T \times \hT$-action $*$ defined by $$(h, \wh h) * x= (h \wh h \inv, \wh h) \cdot x,$$
where $\cdot$ denotes the standard $T \times \hT$-action on $\lX.$ We denote by $\Br^*=\Orb_{\lX^*} (\lX^*)$, considered with its $X(T) \times X(\hT)$-graduation induced by the $*$-action. As $\CC$-algebras, $\Br^*= \Br.$
Note that a subset of $\lX$ is stable for the $\cdot$ action if and only if it is stable for the $*$ action. 

\begin{coro}
    \label{* suitable lifting}
    The triple $(\lX^*, \widetilde \phi, X)$ is homogeneously suitable for $D$-lifting. Moreover, the $X(T) \times X(\hT)$-graduation on $\Br^*$ is a $X(\hT)$-graded lifting graduation. 
\end{coro}

\begin{proof}
    The first part of the statement is obvious from Proposition \ref{prop:branch scheme suitable tilde}. Moreover, using \eqref{ T x hT action on T x U}, we easily deduce that, for any $h,t \in T$, $\wh h \in \hT$ and $\widetilde y \in \widetilde Y$,  then
    $$(h,\wh h) * \widetilde \phi \bigl((t, \widetilde y)\bigr)= \widetilde \phi\bigl((h t , \wh h \cdot \widetilde y)\bigr)$$
    where $h \cdot \widetilde y$ denotes the conjugation action. In particular, the map $\iota : \widetilde Y \longto \lX^*$ defined by $\widetilde y \longmapsto \widetilde \phi ((e, \widetilde y))$ is $\hT$-equivariant. We conclude by applying Lemma \ref{lem: H graded pole filt equivariance}.
\end{proof}

Let $\iota : \widetilde Y \longto \lX^*$ be as in the proof of the previous corollary.

\begin{prop}
\label{prop:geom realisation coeff general}
    For any $(\lambda, \hlambda) \in X(T)^+ \times X(\hT)^+$, $\iota^*$ induces an isomorphism $$\Hom(V(\lambda), V(\hlambda))^\hG \longto \CC[U]^\hU_{\lambda, \hlambda -\rho(\lambda)}$$
where
\begin{equation*}
\begin{array}{r l l }
\CC[U]^\hU_{\lambda , \hlambda - \rho(\lambda)} = \{ f \in \CC[U]^\hU \, : &  f(\wh h^\inv u \wh h)= (\hlambda-\lambda)( \wh h) f(u) &  \text{for} \quad \wh h \in \hT, u \in U \ , \, \text{and} \\
& \val_\alpha (1 \otimes f) \geq -\langle \lambda, \alpha^\vee \rangle & \text{for} \quad  \alpha \in D\}.
\end{array}
\end{equation*}
\end{prop}

\begin{proof}
    It's a simple computation that $\Br_{\lambda, \hlambda}= \Br^*_{\lambda, \hlambda - \rho(\lambda)}$. Then the proposition follows from Corollary \ref{* suitable lifting} and Definition \ref{lifting graduation defi}.
\end{proof}

\begin{remark}
\label{geometric realisation branching coefficients remark}
 The above proposition gives a geometric realisation of the multiplicity spaces for the branching problem of $\hG$ in $G$. Consider the tensor product case, that is when $G=\hG \times \hG$ and $\hG$ is diagonally embedded in $G$. By identifying $U/ \hU= (\hU \times \hU) / \hU $ with $\hU$ via the map $(\wh u_1, \wh u_2) \longmapsto \wh u_1 \wh u_2 \inv$, we can easily deduce Proposition 1.4 of \cite{zelevinsky2001littlewood}. Actually, \cite{zelevinsky}[Proposition 1.4] is the starting point for the interest in perfect bases, which have classically been one of the most used items to study tensor product decomposition. See \cite{kamnitzer2022perfect} for a beautiful survey on perfect bases. We wonder if Proposition \ref{prop:geom realisation coeff general} can lead to the development of a similar theory for other branching problems.
\end{remark}

\subsection{On question \ref{question: Y tilde cluster}}
\label{sec:on question }

 Question \ref{question: Y tilde cluster} being quite vague, the author feels the following question more interesting.

 \begin{question}
     \label{ques: U=U(z) x hU?}
     For which pairs $(G,\hG)$, there exist a $z \in W$ such that the product $U(z) \times \hU \longto U$ is an isomorphism?
 \end{question}

If the answer to the previous question is positive for a certain $z \in W$, then the natural map $U(z) \longto \widetilde Y$ is an isomorphism. Thus, $\widetilde Y$ has a natural cluster structure obtained throughout this identification. In the following, we exhibit two families for which we can answer positively to Question \ref{ques: U=U(z) x hU?}. In general, the answer to the previous question is positive in many other interesting cases, as we will see in some forthcoming works.

\subsubsection{Levi subgroups}
\label{Levi subgroup setting}

Let $I \subsetneq \Delta$ be a strict subset of $\Delta$, and $\Phi_I= \Phi \cap \ZZ I$. This is a root system with $I$ as simple set of roots and Weyl group $W_I= \langle s_\alpha \, : \, \alpha \in I \rangle \subseteq W$. Consider $\hG= G_I$ the \textit{Levi subgroup} of $G$, which is a reductive group with root system $\hPhi=\Phi_I$. Let $\hT=T=\hG \cap T$ and let $\hB= B \cap \hG$. Note that $\hW=W_I$. Let $w_{0,I}$ be the longest element of $\hW$ and $z= w_{0,I}^\vee= w_0 w_{0,I}$. We have that $\hU= U(w_{0,I})$. Then we define $$ Y= U(z)=U(w_{0,I}^\vee).$$
As a special case, when $I=\emptyset$, then $\hG=T$ and $Y=U.$

\subsubsection{Products}
\label{Products setting}

Let $\hG$ and $H$ be semisimple, simply connected, complex algebraic groups such that $\hG \subseteq H$. Let $T_H$, $B_H$ (resp. $\hT$, $\hB$) be a maximal torus  and a Borel subgroup of $H$  (resp. $\hG$) such that $T_H \subseteq B_H$ (resp. $\hT \subseteq \hB)$. Denote by $U_H$ (resp $\hU$) the unipotent radical of $B_H$ resp $\hB$. Assume furthermore that $\hT \subseteq T_H$ and $\hB \subseteq B_H$.

Let $G= H \times \hG$, and consider $\hG$ diagonally embedded in $G$.
Then, $T= T_H \times \hT$, $B=B_H \times \hB$ and $U= U_H \times \hU$ (recall our convention on products \ref{remark on products}). 

Let $z=(w_{0,H},e) \in W =W_H \times \hW$, where $w_{0,H}$ is the longest element of $W_H$. Then we set $$Y=U(z)=U_H \times \{e\}.$$

Note that, the case of $\hG$ diagonally embedded into $\hG^n$, which is the product of $n$ copies of $\hG$, is a special case of this construction. The corresponding branching problem consists in decomposing the tensor product of $n$ irreducible representations of $\hG$, under the diagonal action. Also, the branching problem of $\hG \subseteq H$ is completely determined by the one of $\hG \subseteq G$. 

\begin{remark}
    \label{remark cluster structure products}
Note that the reduced expressions of $z \in W$ are in obvious bijection with the reduced expressions of $w_{0,H} \in W_H$. The reduced expression $\ii \in R(z)$ gives a seed $t_\ii$ of $ \CC[U(z)]$. Thinking about the same reduced expression as an expression for $w_{0,H}$, we obtain a seed $t_\ii'$ of $ \CC[U_H]$. It's trivial to verify, using the last identity in Remark \ref{remark on products}, that $t_\ii=t_\ii'$ under the obvious identifiction between $U(z)$ and $U_H.$ So, we make no difference between these two seeds.

\end{remark}

From now on, we suppose that the pair $\hG \subseteq G$ belongs to one of the two cases discussed above. The element $z \in W$, for a given pair, is the one previously described.

\begin{lemma}
\label{T x Y open embedding branching}
  The product $U(z) \times \hU \longto U$ is a $\hT$-equivariant isomorphism. Thus, the natural inclusion $Y=U(z) \longto U$ induces a $\hT$-equivariant isomorphism between $Y$ and $\widetilde Y= \Spec(\CC[U]^\hU)$.
   
 \end{lemma}

\begin{proof}
    For the Levi case, the claim follows from \cite{Hum}[Proposition 28.1]. For the product case, the above map is given explicitly by 
    $$
    \begin{array}{rcl}
        U_H \times \hU & \longto & U_H \times \hU \\
        (u_h , \wh u) & \longmapsto & (u_h \wh u, \wh u)
    \end{array}
    $$
    which is obviously an isomorphism.
    \end{proof}

Let $\phi : T \times Y \longto \lX$ be the map obtained from $\widetilde \phi$, by identifying $Y$ and $\widetilde Y$ as in Lemma \ref{T x Y open embedding branching}. 
\begin{coro}
    \label{X phi suitable lifting}
    The triple $(\lX, \phi, X)$ is homogeneously suitable for $D$-lifting.
\end{coro}
\begin{proof}
    It is obvious from Corollary \ref{* suitable lifting} and Lemma \ref{T x Y open embedding branching}.
\end{proof}

    Recall that $\Br^*$ is the algebra $\Br$ with its twisted $X(T) \times X(\hT)$ graduation, as defined in Section \ref{sec:lifting grad branching}.

\begin{theo}
    \label{application minimal mon lifting branching}
     Let $\ii \in R(z)$ and $\sigma_\ii $ be the $X(T)$-degree configuration on $t_\ii$ defined in Proposition \ref{prop:t i graded}. We have that
     $$\displaystyle \uclu(\lif t_\ii)= \Br_{\prod_{\alpha \in D} X_\alpha} \quad \text{and} \quad  \uclu(\lif t_\ii^D) \subseteq \Br.$$
     Moreover, $\uclu(\lif t_\ii^D)$ with the graduation induced by $\lif \rho( \sigma_\ii)$, is a graded subalgebra of $\Br^*$.
\end{theo}
\begin{remark}
    In the notation of Theorem \ref{application minimal mon lifting branching}, $\rho(\sigma_\ii) \in X(\hT) ^{I_{\ii}}$ is the $X(\hT)$-degree configuration obtained by applying component-wise the restriction map $\rho : X(T) \longto X(\hT)$ to $\sigma_\ii.$
\end{remark}

\begin{proof}
Because of Corollary \ref{X phi suitable lifting} and Lemma \ref{factorial upper, coprimality} we can apply Theorem \ref{minimal monomial lifting theo}. The part on the graduation is a consequence of Lemma \ref{lifting subgraded}.
\end{proof}
From now on, fix $\ii \in R(z)$ as above. Let $t=t_\ii$ and $\nu \in \ZZ^{D \times I}$ be the minimal lifting matrix of the seed $t$ associated to $\lX$. Recall that, for any $k \in I$, $\nu_{\bullet, k}$ is canonically an element of $X(T).$ 
\begin{lemma}
\label{lem:degree lifting variables}
  For any $k \in I$, $\lif x_{k}$ is a $X(T) \times X(\hT)$-homogeneous element of $\Br$ of degree $$\bigl(\nu_{\bullet, k}\, , \, \rho(\nu_{\bullet, k})+ \rho(z_{\leq k}^{-1} \varpi_{i_k} - \varpi_{i_k})\bigr).$$
\end{lemma}
\begin{proof}
    Because of \eqref{ T x hT action on T x U} and Lemma \ref{lem:basic prop minors}, we easily deduce that $1 \otimes x_{k}$ is a $X(T) \times X(\hT)$-homogeneous rational function of weight $( 0 \, , \,  \rho(z_{\leq k}^{-1} \varpi_{i_k} - \varpi_{i_k}))$. Moreover, for any $\alpha \in D$, $X_\alpha$ is homogeneous of degree $(\pia \, , \, \rho(\pia)).$ The lemma follows from the definition of $\lif x_{k}.$
\end{proof}

The previous lemma gives some explicit components of the branching from $G$ to $\hG$. Moreover, using the mutation formula for degree configurations \eqref{mutation graduation}, we obtain a (a priory infinite) set of non-zero components for the branching from $G$ to $\hG$. These non-zero components can be computed by explicit recursive formulas. 
Also, if $t^*$ is a graded seed mutation equivalent to $\lif t^D$, for any $k \in I_{uf}$ we have that 
$$\sigma^*_k + \mu_k(\sigma_k^*) \in X(T) \times X(\hT)$$
identifies a component of the branching from $G$ to $\hG$ of multiplicity at least $2$. Indeed, $\sigma_k^* + \mu_k(\sigma_k^*)$ is the degree of $x_k^*\mu_k(x_k^*)$ which, by the exchange relation \eqref{exchange relation}, can be expressed as a sum of two linearly independent homogeneous elements of degree $\sigma_k^* + \mu_k(\sigma_k^*).$ Similarly, we easily deduce that for any $n \geq 1$, $n \bigl(\sigma^*_k +  \mu_k(\sigma_k^*) \bigr)$ identifies a component of multiplicity at least $n+1.$
 The author is curious about the following questions.

 \begin{question}
     \label{ques: on weights}
     \begin{enumerate}
         \item Let $k \in I_{uf}$. Define $\mathfrak{D}(k)= \{ \deg(x_k^*) \in X(T) \times X(\hT)$ \, : \, $t^* \sim \lif t^D \}$.
         Is the set $\mathfrak{D}(k)$ infinite?.
         \item  For which $t^* \sim \lif t^D$ and $k \in I_{uf}$, if $\deg(x_k^*)=(\lambda, \hlambda)$, then $\dim\Hom(V(\lambda), V(\hlambda))^\hG=1.$
     \end{enumerate}
 \end{question}

\section{Study of equality}
\label{study of equality branching section}
In this section we  study whether there is equality, between $\uclu(\lif t_\ii^D)$ and $\Br$, in Theorem \ref{application minimal mon lifting branching}. By Proposition \ref{conditions equality minimal lifting}, for any $\alpha \in D$, we need to compare the valuations $\val_\alpha$ and the cluster valuations $\cval_\alpha$. 

\bigskip

Recall that $\CC(\lX)$ is naturally identified with a subfield of $\CC(G)$. If a rational function $f \in \CC(G)$ belongs to $\CC(\lX)$, then $f$ is invariant for left multiplication by $U^-$ and for right multiplication by $\hU$.

\begin{lemma}
    \label{lem:valuations in G not X}
    For any $f \in \CC(\lX)$ and $\alpha \in D$, $\val_\alpha(f)= \val_{\overline{ B^- s_\alpha B}}(f).$
\end{lemma}

\begin{proof}
    It's sufficient to prove the case of $f \in \Orb_\lX(\lX)$. Then, $\val_\alpha(f)$ is the multiplicity of $X_\alpha$, in the decomposition of $f$ into irreducible factors in $\Orb_\lX(\lX).$ By Lemma \ref{valuation on invariants or above}, this is the same as the multiplicity of $X_\alpha$ in the decomposition of $f$ into irreducible factors in $\CC[G]$. The statement follows by Lemma \ref{lem:basic prop minors}, since $V(X_\alpha)=\overline{B^- s_\alpha B}$.
\end{proof}

\subsection{The Levi case}
The notation of this section is as in Subsection \ref{Levi subgroup setting}. We want to prove the following theorem.

\begin{theo}
\label{equality Levi}
    Suppose that $G$ is simple and $\hG=G_I$ is the Levi subgroup corresponding to $I \subsetneq \Delta$. For any $\ii \in R(z)$, there's equality in Theorem \ref{application minimal mon lifting branching}. That is $\uclu(\lif t_\ii^D)=\Br(G, G_I)$.
\end{theo}
 Note that, the branching problem from a semisimple group to a Levi subgroup, can always be reduced to  the study of a finite number of branchings from a simple group, to a Levi subgroup.
\begin{coro}
    If $G$ is simply laced and $\ii, \ii' \in R(z)$, then $\lif t_\ii^D \simeq \lif t_{\ii'}^D$.
\end{coro}

\begin{proof}
    Since $G$ is simply laced,  $t_\ii \simeq t_{\ii'}$. Then $\lif t_\ii^D \simeq \lif t_{\ii'}^D$ by Theorem \ref{equality Levi}, Corollary \ref{coro:if equality monomial lifting commutes} and Lemma \ref{mutation commutes lifting}.
\end{proof}

From now on, $\hG$ and $G$ are as in the hypothesis of Theorem \ref{equality Levi} and $\ii \in R(z)$ is fixed. Recall that $z=w_{0,I}^\vee= w_0w_{0,I}$. We denote $t=t_\ii$ and $\Br=\Br(G,\hG)$. Here $D= \Delta$ and we have a cartesian square
\begin{equation}
  \label{eq:cartesian square Levi}  
 \begin{tikzcd}
U^- \times T \times Y \times \hU \arrow["\psi", r] \arrow[d]  \arrow[dr, phantom, "\lrcorner", very near start] 
& G \arrow["\pi" , d] \\
T \times Y \arrow["\phi", r]
& X
\end{tikzcd}
\end{equation}
where the leftmost vertical map is defined by $(u^-, h, y, \wh u) \longmapsto (h , y)$, $\psi$ by $(u^-, h, y, \wh u) \longmapsto u^-h y \wh u$ and $\pi$ is the natural projection. Note that the image of $\psi$ is $G_0.$ 

\bigskip

Recall that, for a $T$-stable subgroup $H$ of $U$ and $\alpha \in \Delta$, $H_{\setminus \alpha} = H \cap s_\alpha U s_\alpha$.

\begin{lemma}
\label{chart Levi case}
For $\alpha \in D$, let $\iota_\alpha$ be the map defined by 
\begin{enumerate}
    \item If $\alpha \in I$
    \begin{equation*}
        \begin{array}{c c l}
         \iota_\alpha :  U^- \times T \times \CC \times Y \times \hU_{\setminus \alpha}  & \longto   & G\\
         x=(u^-,h,t,y,\widehat u) & \longmapsto & \iota_\alpha(x)= u^-h \sba^{-1} x_{-\alpha}(t)y \widehat u.
        \end{array}
    \end{equation*}

    \item If $\alpha \not \in I$
    \begin{equation*}
        \begin{array}{r c l}
         \iota_\alpha :  U^- \times T \times \CC \times Y_{\setminus \alpha} \times \hU  & \longto   & G\\
         x=(u^-,h,t,y, \widehat u) & \longmapsto & \iota_\alpha(x)=u^-h \sba^{-1} x_{-\alpha}(t)y \widehat u.
        \end{array}
    \end{equation*}
\end{enumerate}
Then, $\iota_\alpha$ is an open embedding whose image intersects $\overline{B^- s_\alpha B}$. Moreover, $\iota_\alpha(x) \in \overline{B^- s_\alpha B}$ if and only if $t = 0$, otherwivse $\iota_\alpha(x) \in G_0$.
\end{lemma}

\begin{proof}
    It's clear that, for any $\alpha \in \Delta$, the map 
    \begin{equation}
    \label{eq 2}
         \begin{array}{r c l}
        \widetilde \iota_\alpha :  U^- \times T \times U & \longto   & G\\
         (u^-,h,u) & \longmapsto & u^-h u \sba^{-1}
        \end{array}
    \end{equation}
is an open embedding whose image intersects $\overline{B^- s_\alpha B}$. Moreover, the map 
 \begin{equation*}
         \begin{array}{r c l}
        \CC \times U_{\setminus \alpha } & \longto   & U\\
         (t,u) & \longmapsto & x_\alpha(-t)u
        \end{array}
    \end{equation*}
is an isomorphism and conjugation by $\sba$ induces an automorphism of $U_{\setminus \alpha }$. 
To conclude that $\iota_\alpha$ is an open embedding, note first that the product induces an isomorphism between $Y \times \hU_{\setminus \alpha}$ (resp. $Y_{\setminus \alpha} \times \hU$) and $U_{\setminus \alpha}$, if $\alpha \in I$ (resp $\alpha \not \in I)$, because of \cite[Proposition 28.1]{Hum:refgroup}. Finally, we use that for any $t \in \CC$, $$ x_\alpha(-t) \sba^{-1}= \sba^{-1} x_{-\alpha}(t), $$
which is an easy identity that can be checked in $\SL_2$. It's clear from the definition of $\iota_\alpha$ that $\iota_\alpha(x) \in \overline{ B^- s_\alpha B} $ if the $t$-coordinate of $x$ is zero. If $t \neq 0$, $\iota_\alpha(x) \in G_0$ because of the second identity in \eqref{x alpha s alpha = ..}.
\end{proof}

By the previous lemma, the divisor $\overline{B^- s_\alpha B}$, in the chart given by $\iota_\alpha$, corresponds to $\{ t=0 \}.$ Then we make a very important computation.

\begin{lemma}
    \label{fundamental computation Levi case}
Let $\lambda \in X(T)$, $\alpha \in D$ and $f \in \CC(\lX)_{\lambda,\bullet}.$ For a generic $x$ as in the notation of Lemma \ref{chart Levi case}, such that $t \neq 0$, we have $$f(\iota_\alpha(x))= \lambda(h)t^{\langle \lambda, \alpha^\vee \rangle} f(x_\alpha(t\inv) y).$$
Moreover, if $\beta \in D$, then 
$$ X_\beta(\iota_\alpha (x))= \pib(h) t^{\langle \pib, \alpha^\vee \rangle}.$$
\end{lemma}

\begin{remark}
    With little abuse, for $\lambda \in X(T)$, we denote by $\CC(\lX)_{\lambda,\bullet}$ the set of $T$-equivariant rational functions of weight $\lambda.$
\end{remark}
\begin{proof}
If $x=(u^-,h,t,y,\widehat u)$, using the second formula in \eqref{x alpha s alpha = ..} and the fact that $f$ is $U^- \times \hU$-invariant, we have that
 \begin{equation*}
     \begin{array}{rl}
          f(\iota_\alpha(x)) = & f(h \sba^{-1} x_{-\alpha}(t) y)  \\
         =  & \lambda(h) f( x_{-\alpha}(-t\inv)\alpha^\vee(t)x_\alpha(t\inv) y)\\
         = & \lambda(h)t^{\langle \lambda, \alpha^\vee \rangle} f(x_\alpha(t\inv) y). 
     \end{array}
 \end{equation*}
 The statement on $X_\beta=\delbee$ follows immediately from Lemma \ref{lem:basic prop minors}.
 
 \end{proof}
We use the convention that, whenever $\alpha \in D$ is fixed, for any $v,w \in W$ and $\beta \in \Delta$, we denote
\begin{equation}
    \label{definizione D v w}
    D^{\pib}_{v,w}= \begin{cases}
        (\delbvw)_{|Y} & \text{if} \, \, \alpha \in I \\
        (\delbvw)_{|Y_{\setminus \alpha }} & \text{if} \, \, \alpha \not \in I. 
    \end{cases}
\end{equation}

\begin{prop}
    \label{valuation and formula cluster variable Levi}
    Let $k \in I$ and $\alpha \in D$. Then $\nu_{\alpha,k}=\delta_{\alpha,i_k}.$ Moreover for $x$ as in Lemma \ref{chart Levi case} 
    $$\lif x_k\bigl(\iota_\alpha(x)\bigr) = \varpi_{i_k}(h)\biggl(p_{k,0}(y) + t p_{k,1}(y)\biggr) \quad \text{where} \quad p_{k,0}=D^{\varpi_{i_k}}_{s_\alpha, z_{\leq k}^{-1}}.$$
   Moreover:
       $$p_{k,1}= \begin{cases}
           0 & \textit{if} \quad \alpha \neq i_k\\
       D^{\varpi_{i_k}}_{e, z_{\leq k}^{-1}} & \text{otherwise}.
   \end{cases}$$
\end{prop}

\begin{remark}
Note that $p_{k,0}$ and $p_{k,1}$ depend on $\alpha.$ So, when we use such a notation, we assume that it refers to an $\alpha$ which is sufficiently clear from the context. 
\end{remark}

\begin{proof}
    First, we compute $\nu.$
    
    Suppose that $\alpha \in I$ and that $x$ is such that $t \neq 0$. Since $1 \otimes x_k \in \CC(\lX)_{0,\bullet}$, from Lemma \ref{fundamental computation Levi case} we deduce that $x_k(\iota_\alpha(x))=1 \otimes x_k(x_\alpha(t\inv) y)$. Recall that, in this case, $y \in Y=U(z).$ Note that conjugation by $x_\alpha(t\inv)$ stabilises $Y$.
    In fact, if $\delta \in \Phi(z)=\Phi^+ \setminus \Phi^+_I$ and $\delta + \alpha \in \Phi$, then $\delta + \alpha \in \Phi(z).$ Hence 
    $$
    \begin{array}{rl }
        1 \otimes x_k(x_\alpha(t\inv) y) & 
        =1 \otimes x_k(x_\alpha(t\inv) y x_\alpha(- t\inv ) x_\alpha(t\inv))\\
        & = 1 \otimes x_k(x_\alpha(t\inv) y x_\alpha (-t\inv))
    \end{array}$$
    where the last equality follows from the fact that, if $\alpha \in I$, then $x_\alpha(t\inv) \in \hU$. Note that (see \eqref{eq:cartesian square Levi})
    $$\phi \biggl((e, x_\alpha(t\inv) y x_\alpha (-t\inv))\biggr) = \pi \bigl(x_\alpha(t\inv) y x_\alpha (-t\inv)\bigr).$$
    In particular, from the definition of $1 \otimes x_k$ we deduce the first of the following equalities 
    \begin{equation*}
        \begin{array}{rl}
  1 \otimes x_k(x_\alpha(t\inv) y x_\alpha (-t\inv))= &  x_k(x_\alpha(t\inv) y x_\alpha (-t\inv))\\[0.5em]
  = & \Delta^{\varpi_{i_k}}_{e, z_{\leq k}^{-1}}((x_\alpha(t\inv) y x_\alpha (-t\inv))\\[0.9em]
  = & \Delta^{\varpi_{i_k}}_{e, z_{\leq k}^{-1}}((x_\alpha(t\inv) y ).
  \end{array}
    \end{equation*}
The second equality is the definition of $x_k$ and the last one follows from Lemma \ref{minors, developpement right translation} and the fact that $z_{\leq k } \alpha \in \Phi^+$. In fact, from the definition of $z_{\leq k } $, we have that $z_{\leq k }  \leq_L z$, where $\leq_L$ denotes the left weak order. Hence, $\Phi(z_{\leq k } ) \subseteq \Phi(z)$ and $\alpha \not \in \Phi(z).$ Now we apply Lemma \ref{minors, developpement left translation}. \\
If $i_k \neq \alpha$, we deduce that 
$$  \Delta^{\varpi_{i_k}}_{e, z_{\leq k}^{-1}}((x_\alpha(t\inv) y )=  \Delta^{\varpi_{i_k}}_{e, z_{\leq k}^{-1}}( y ) = \Delta^{\varpi_{i_k}}_{s_\alpha , z_{\leq k}^{-1}}( y ).$$
The last equality follows from Lemma \ref{lem:basic prop minors}, since $s_\alpha \varpi_{i_k}=\varpi_{i_k}$. In particular, we deduce that $\nu_{\alpha,k}=0.$\\
If $i_k=\alpha$, we have that 
$$ \Delta^{\varpi_{i_k}}_{e, z_{\leq k}^{-1}}((x_\alpha(t\inv) y ) = \Delta^{\varpi_{i_k}}_{e, z_{\leq k}^{-1}}( y ) + t\inv \Delta^{\varpi_{i_k}}_{s_\alpha, z_{\leq k}^{-1}}( y ).$$
By Lemma \ref{vanishing minor s,w}, we have that $\Delta^{\varpi_{i_k}}_{s_\alpha, z_{\leq k}^{-1}}$ doesn't vanish on $U$. In fact, $s_{i_k}=s_\alpha \in \supp( z_{\leq k}^{-1})$. Then, by Lemma \ref{lem:basic prop minors}, $\Delta^{\varpi_{i_k}}_{s_\alpha, z_{\leq k}^{-1}}$ doesn't vanish on 
$$U \cap s_\alpha U s_\alpha \cap  z_{\leq k}^{-1} U^-  z_{\leq k} \subseteq U \cap z\inv U^- z =Y. $$
We deduce that $\nu_{\alpha,k}=1.$

\bigskip

We assume now that $\alpha \not \in I$. In this case, $x_\alpha(t\inv) \in Y$, so for $y \in Y_{\setminus \alpha}$ we have that $x_\alpha(t^\inv)y \in Y$. In particular, by a very similar but simpler argument then the previous one, we have that 

$$ 1\otimes x(k)(\iota_\alpha(x))=1 \otimes x_k(x_\alpha(t\inv) y)= \Delta^{\varpi_{i_k}}_{e, z_{\leq k}^{-1}}(x_\alpha(t\inv) y).$$ 
Then, literally the same proof of the previous case, up to replacing $Y$ with $Y_{\setminus \alpha }$, implies that $\nu_{\alpha,k}=\delta_{\alpha,i_k}$.

\bigskip

From the explicit expression of $\nu$, it follows that $\lif x_k= X_{i_k}(1 \otimes x_k).$ Then, the last part of the statement follows from the previous calculations and Lemma \ref{fundamental computation Levi case}.
\end{proof}

\begin{coro}
\label{cor:components cluster variables Levi}
    For any $k \in [l]$, $\rho(z_{\leq k}^{-1} \varpi_{i_k})$ is a dominant weight and the $\hG$-representation $V\bigl(\rho(z_{\leq k}^{-1} \varpi_{i_k})\bigr)$ is a sub-representation of $V(\varpi_{i_k}).$
\end{coro}

\begin{proof}
   The cluster variable $\lif x_k$ is homogeneous of degree $( \varpi_{i_k} \, , \, \rho(z_{\leq k}^{-1} \varpi_{i_k}))$, because of Lemma \ref{lem:degree lifting variables} and Proposition \ref{valuation and formula cluster variable Levi}.
\end{proof}

\begin{lemma}
    \label{support z Levi}
   For any $I \subsetneq \Delta$, $\supp(z)=\Delta$.
\end{lemma}

\begin{remark}
\label{rem: simple Levi}
    This is the only lemma in which we use that $G$ is simple.
\end{remark}

\begin{proof}
    Let $\alpha \in \Delta$ and $\Phi^+_\alpha=\{ \sum_{\beta \in \Delta} n_\beta \beta \in \Phi^+ \, : \, n_\alpha > 0 \} $. If by contradiction $s_\alpha \not \in \supp(z)$, then $z \Phi^+_\alpha \subseteq \Phi^+$. In fact, for any $\delta \in \Phi$, the $\alpha$-coefficients of $\delta$ and $z \delta$ (in the base $\Delta$) are the same. 
    
    Since $G$ is simple, we can consider $\gamma \in \Phi^+$ the longest root of $\Phi$. We have that $\gamma \in \Phi^+_\beta $ for any $\beta \in \Delta$, hence $\gamma \not \in \Phi(z)$. But then $\gamma \in \Phi(w_{0,I})$. This is a contradiction since $\gamma \not \in \Phi_I.$
\end{proof}
By the previous lemma, for any $\alpha \in \Delta$, $\alpha^{\min}$ is well defined.

\begin{lemma}
\label{p alpha min 0 polinomio nei precedenti Levi}
   In the notation of Proposition \ref{valuation and formula cluster variable Levi} 
   $$ \displaystyle p_{\alpha^{\min},0}= \prod_{j=1}^{\alpha^{\min}-1} p_{j,0}^{c_{j}}$$
    for some non-negative integers $c_j$.
\end{lemma}

    Again, the exponents $c_j$ depend on $\alpha$ (as the functions $p_{j,0})$, so we use such a notation when it refers to an $\alpha$ which is sufficiently clear from the context.

\begin{proof}
   Accordingly to if $\alpha \in I$ or not, this is a reformulation of Lemma \ref{delta alpha min polinomio nei precedenti U(w)} or of Lemma \ref{delta alpha min polinomio nei precedenti U(w) alpha} respectively.
\end{proof}

\begin{lemma}
\label{p k 0 algebraically independent Levi}
    The functions $p_{k,0}$, for $k \neq \alpha^{\min}$ and $k \in [l]$, are algebraically independent.
\end{lemma}

\begin{proof}
    Again this is a reformulation of Lemma \ref{delta s w alg indip U(w)} or Lemma \ref{delta s w alg indip U(w) alpha}.
\end{proof}

We are ready to prove Theorem \ref{equality Levi}.

\begin{proof}[Proof of Theorem \ref{equality Levi}]
    By contradiction suppose that $\uclu(\lif t^D)$ is strictly contained in $\Br$. By Proposition \ref{conditions equality minimal lifting}, there exists an $\alpha \in D$ and $f \in \Br$ such that $\cval_\alpha(f) < 0$. Since $\Br$ is $X(T)$-graded, we can suppose $f $ homogeneous.
    Recall that $\uclu(\lif t^D)$ is a graded subalgebra of $\Br$, which in turn is a $X(T)$-graded subalgebra of $\uclu(\lif t)$. In particular, up to multiplying for a monomial in the $X_\beta$, $\beta \in D$, and in the unfrozen variables of $\lif x$,  we can suppose that there exist $\lambda \in X(T)$ and $f \in \Br_{\lambda,\bullet}$ such that 

    \begin{equation}
    \label{eq 3}
        f= \frac{P}{\lif x_\alpha} \quad \text{where} \quad P= \sum_{n  \in \NN^{\lif \spc I}}a_n \lif x^n  \, \, : \, \, \lif x_\alpha \not | \, P.
    \end{equation}
Hence, $P$ is a polynomial in the variables of the cluster $\lif x$, which is not divisible by $\lif x_\alpha$. Here divisibility is intended in the polynomial ring. Up to changing $f$, we can assume that if $n_\alpha > 0 $, then $a_n=0$, so that the sum defining $P$ runs over $\NN^{\lif \spc I \setminus \{\alpha \}}.$ 
We use the convention that, if $\beta \in D$, then $i_\beta= \beta$ so that $i_j$ makes sense for any $j \in \lif I$. 
The fact that $f$ is $X(T)$-homogeneous of degree $\lambda$ imposes  (using Proposition \ref{valuation and formula cluster variable Levi}) that, for any $n$ such that $a_n \neq 0$

\begin{equation}
    \label{eq 4}
    \sum_{j \in \lif \spc I \setminus \{\alpha\} } n_j \varpi_{i_j}= \lambda + \pia.
\end{equation}

Next, let $x$ as in Lemma \ref{chart Levi case}. Using Proposition \ref{valuation and formula cluster variable Levi} and Lemma \ref{fundamental computation Levi case} (recall that $\lif x_\beta=X_\beta)$ we deduce that $$f(\iota_\alpha(x))=\lambda(h) t^\inv \sum_{k\geq 0} t^k f_k(y)$$
for certain $f_k \in \CC[Y]$ (resp. $\CC[Y_{\setminus \alpha}]$) if $\alpha \in I$ (resp. $\alpha \not \in I$). 
For $\beta \in \Delta$, we set $p_{\beta,0}=1$ so that, for $\beta \neq \alpha$, $$\lif x_\beta(\iota_\alpha(x))= \pib(h)=\pib(h)p_{\beta,0}.$$ 
Using the last expression, the fact that $n_\alpha=0$ if $a_n \neq 0$ and Proposition \ref{valuation and formula cluster variable Levi}, we get that
$$f_0= \sum_{n \in \NN^{\lif \spc I \setminus \{ \alpha \}}}a_n p_{\bullet,0}^n \quad \text{where} \quad p_{\bullet,0}^n= \prod_{j \in \lif \spc I \setminus \{\alpha \}}p_{j,0}^{n_j}.$$ 
But since $f \in \Br$, $\iota_\alpha^*(f)$ is regular, hence $f_0=0.$

\bigskip

Consider the linear map $\widetilde \pi : \ZZ^{\lif \spc I \setminus \{\alpha \}} \longto \ZZ^{(I \setminus \alpha^{\min})}$ defined in the following way.
For $\beta \in D \setminus \{\alpha \}$, $\widetilde \pi(e_\beta)=0$. For $j \in I \setminus \{\alpha^{\min} \}$, $\widetilde \pi(e_j)=e_j$ and $\widetilde \pi(e_{\alpha^{\min}})= \sum_{j < \alpha^{\min}} c_j e_j$, where the coefficients $c_j$ are the ones of Lemma \ref{p alpha min 0 polinomio nei precedenti Levi}. Since $p_{\beta,0}=1$ for $\beta \in D$ and because of Lemma \ref{p alpha min 0 polinomio nei precedenti Levi}, we have that for any $n \in \NN^{\lif \spc I \setminus \{\alpha \}}$
$$p_{\bullet,0}^n=p_{\bullet,0}^{\widetilde \pi(n) }.$$

In particular $$f_0 = \sum_{ m \in \NN^{I \setminus \{\alpha^{\min} \}}}b_m p_{\bullet,0}^m \quad \text{where} \quad b_m= \sum_{n \in \NN^{{\lif \spc  I \setminus \{\alpha \}}} \, : \, \widetilde \pi(n)=m}a_n.$$
By Lemma \ref{p k 0 algebraically independent Levi}, the functions $p_{j,0}$ with $j \in I \setminus \{ \alpha^{\min} \}$ are algebraically independent.
Since $f_0=0$, then for any $m \in  \NN^{I \setminus \{\alpha^{\min} \}}$, $b_m=0$. We claim that, for any $m \in  \NN^{I \setminus \{\alpha^{\min} \}}$, there exists at most one  $n \in \NN^{{\lif \spc I \setminus \{\alpha \}}}$ such that $a_n \neq 0$ and $\widetilde \pi(n)=m$. This implies at ones that, for any $n \in \NN^{{\lif \spc I \setminus \{\alpha \}}} $, $a_n= 0$, which is a contradiction.

\bigskip

To prove the claim, just notice that any $n$ such that $a_n \neq 0$ satisfies the weight condition \eqref{eq 4}. In particular, consider the linear map $\pi : \ZZ^{{\lif \spc I \setminus \{\alpha \}}} \longto  X(T) \times \ZZ^{I \setminus \{\alpha^{\min} \}}$ defined by $\pi(e_j)=(\varpi_{i_j}, \widetilde \pi(e_j))$. Remember that, for $\beta \in \Delta$, $i_\beta= \beta$. If $n$ is such that $a_n \neq 0$ and $\widetilde \pi(n)=m $, then $\pi(n)= (\lambda + \pia, m)$. But the map $\pi$ is injective. To prove this, consider $\pi$ as a matrix with columns indexed by $\lif I \setminus \{ \alpha \}$ and rows indexed by $\lif I \setminus \{ \alpha^{\min} \}.$ Recall that $X(T)$ is identified with $\ZZ^D$ by means of the fundamental weights. Enumerate the roots of $\Delta \setminus \{ \alpha \}$ from $1$ to $r-1$. Then order the columns of the matrix representing $\pi$ accordingly to the order 
$$\beta_1 < \dots < \beta_{r-1} < 1 < \dots < l$$
and the rows according to the order
$$\beta_1 < \dots < \beta_{r-1} < 1 < \dots < \alpha^{\min}-1 < \alpha < \alpha^{\min} + 1 < \dots < l.$$
Then, the matrix representing $\pi$ is upper triangular with ones on the diagonal. In particular $\pi$ is invertible.
\end{proof}

\subsection{The product case: tensor product}
\label{sec: prod, tensor prod}
To simplify the notation, in this section we switch the role of $G$ and $\hG$. In particular, $G$ is a subgroup of $\hG$. We assume that $G$ is semisimple, simply connected  and we consider $G$ as a subgroup of $\hG=G \times G$ by the diagonal embedding. This is the simplest case of Subsection \ref{Products setting}: when $H=G$. In this section we want to prove the following

\begin{theo}
\label{equality tensor product}

For any $\ii \in R(w_0)$, there's equality in Theorem \ref{application minimal mon lifting branching}. That is: $$\uclu(\lif t_\ii^D)= \Br(G \times G, G).$$
    
\end{theo}

\begin{coro}
\label{cor:independence reduced exp tensor}
    If $G$ is simply laced and $\ii, \ii' \in R(w_0)$, then $\lif t_\ii^D \simeq \lif t_{\ii'}^D$.
\end{coro}

\begin{proof}
    Since $G$ is simply laced, then $t_\ii \simeq t_{\ii'}$. Then $\lif t_\ii^D \simeq \lif t_{\ii'}^D$ by Theorem \ref{equality tensor product}, Corollary \ref{coro:if equality monomial lifting commutes} and Lemma \ref{mutation commutes lifting}.
\end{proof}
Here we chose a maximal torus $T$ contained in a Borel $B$, of $G$, and we consider $\hT=T \times T$ and $\hB=B \times B$. Because of Remark \ref{remark cluster structure products}, we make no difference between $z=(w_0,e)$ and $w_0$. Moreover, we identify $\hU (z)$ and $U=U(w_0)$ in the obvious way. Recall that, for any $\ii \in R(w_0)=R(z)$, the seed $t_\ii$ gives a cluster structure to $Y=U$. From now on, we fix $\ii=(i_l, \dots , i_1) \in R(w_0)$ and set $t=t_\ii$. Here, $D=D_l \sqcup D_r$ is the disjoint union of two copies of $\Delta$, namely: the \textit{left copy} is $D_l= \{ \alpha_l \, : \, \alpha \in \Delta \}$ and the \textit{right copy} is $D_r= \{ \alpha_r \, : \, \alpha \in \Delta \}$. For $\alpha \in \Delta$ (see Remark \ref{remark on products}) 

\begin{equation}
    X_{\alpha_l}= \delaee \otimes 1 \quad X_{\alpha_r}=1 \otimes \delaee.
\end{equation}
We denote $\Br=\Br(G \times G, G)$. Note that $X_{\alpha_l} \in \Br_{(\pia, 0), \pia}$ and $X_{\alpha_r} \in \Br_{(0, \pia), \pia}$.
The  map 

\begin{equation*}
    \begin{array}{r c l}
        T \times T \times U & \longto & G \times G \\
        (h_l,h_r,u) & \longmapsto & (h_l u, h_r) 
    \end{array}   
\end{equation*}

induces the $T \times T$-equivariant open embedding $ \phi: T \times T \times Y \longto \lX:= \lX(G \times G,G)$. We have a cartesian square

\begin{equation}
    \label{eq: cartesian tensor}
 \begin{tikzcd}
U^- \times U^- \times T \times T \times U \times U \arrow["\psi", r] \arrow[d] \arrow[dr, phantom, "\lrcorner", very near start]
& G \times G \arrow["\pi" ,d] \\
T \times T \times U \arrow["\phi", r]
& \lX
\end{tikzcd}
\end{equation}
where the leftmost vertical map is defined by $ (u^-_l,u^-_r,h_l,h_r,u_l,u_r) \longmapsto (h_l,h_r,u_lu_r\inv)$, $\pi$ is the natural projection and $\psi$ is defined by $(u^-_l,u^-_r,h_l,h_r,u_l,u_r) \longmapsto (u^-_lh_lu_l \, , \, u^-_rh_ru_r)$. By Lemma \ref{lem:valuations in G not X}, for any $\alpha \in \Delta$, computing the valuation $\val_{\alpha_l}$ (resp. $\val_{\alpha_r}$) of a rational function on $X$ is the same as computing its valuation along the divisor $\overline{B^- s_\alpha B} \times G$ (resp. $ G \times \overline{B^- s_\alpha B} $).

\begin{lemma}
    \label{iota, j tensor product}
    For $\alpha \in \Delta$, let
     \begin{equation*}
        \begin{array}{r c l}
         \iota_\alpha :  T \times T  \times \CC^* \times Y    & \longto   & G \times G\\
         x=(h_l, h_r,t,y) & \longmapsto & \iota_\alpha(x)= (h_l \alpha^\vee(t)  x_{\alpha}(t \inv)y , \, \,  h_r)
        \end{array}
    \end{equation*}
    and 
      \begin{equation*}
        \begin{array}{r c l}
         j_\alpha :  T \times T  \times \CC^* \times Y  & \longto   & G \times G\\
         x=(h_l, h_r,t,y) & \longmapsto & j_\alpha(x)= (h_l y x_{\alpha}(t \inv) , \, \, h_r \alpha^\vee(t)).
        \end{array}
    \end{equation*}
    Then, for any $f \in \CC(\lX) $, $$\val_{\alpha_l}(f)= \val_{\{t=0 \}}(\iota_\alpha^*(f)) \quad \text{and} \quad \val_{\alpha_r}(f)= \val_{\{t=0 \}}(j_\alpha^*(f)).
    $$
\end{lemma}

\begin{proof}
First, note that the maps obtained by composing  $\iota_\alpha$ and $j_\alpha$ with the natural projection $G \times G \longto \lX$ are dominant. In fact, at $t \in \CC^*$ fixed, both maps are obtained from $\phi$ by an appropriate "multiplication twist" by $\alpha^\vee(t)$ and $x_\alpha(t\inv)$, so they're dominant at $t$ fixed. Hence $\iota_\alpha^*$ and $j_\alpha^*$ make sense. Since $f$ is generically defined on $\lX$, we make the assumption, throughout the proof, that the points on which we evaluate $f$ are sufficiently generic so that the evaluation makes sense.

\bigskip

We start from the "left side", so we prove the statement about $\iota_\alpha$. First, we prove that the map $\phi_\alpha$ defined by

\begin{equation*}
        \begin{array}{r c l}
         U^- \times U^- \times  T \times T  \times \CC \times \CC  \times U_{\setminus \alpha} \times U_{\setminus \alpha}  & \longto   & G \times G\\
         x=(u^-_l, u^-_r, h_l, h_r,t , 
         r ,u, \widetilde{u}) & \longmapsto & \phi_\alpha(x)= (u_l^-h_l \sba^{-1} x_{-\alpha}(t) u\widetilde{u}, \, \, u_r^-h_rx_\alpha(r) \widetilde{u})
        \end{array}
    \end{equation*}
is an open embedding whose image intersects $\overline{B^- s_\alpha B} \times G$. Note that, for $x$ as in the definition of $\phi_\alpha$, we have that $\phi_\alpha(x) \in \overline{B^- s_\alpha B} \times G$ if and only if $t = 0$ and, by \eqref{x alpha s alpha = ..}, $\phi_\alpha(x) \in G_0 \times G_0$ otherwise. 

The map 

\begin{equation*}
        \begin{array}{r c l}
         \widetilde \phi_\alpha: U^- \times U^- \times T \times T \times U \times U  & \longto   & G \times G\\
         (u^-_l, u^-_r, h_l, h_r,u_l,u_r) & \longmapsto & (u_l^-h_l  u_l \sba^{-1}, \, \, u_r^-h_r u_r)
        \end{array}
    \end{equation*}
is clearly an open embedding. Moreover, the map 
\begin{equation*}
        \begin{array}{r c l}
         \CC \times U_{\setminus \alpha} & \longto   & U \\
         (t, u)& \longmapsto & x_\alpha(t)u
        \end{array}
    \end{equation*}
is an isomorphism and conjugation by $\sba$ is an automorphism of $U_{\setminus \alpha}$. Furthermore, the map
\begin{equation*}
        \begin{array}{r c l}
         U_{\setminus \alpha } \times U_{\setminus \alpha} & \longto   &  U_{\setminus \alpha } \times U_{\setminus \alpha}\\
         u , \widetilde u & \longmapsto & u \widetilde u, \widetilde u
        \end{array}
    \end{equation*}
is an isomorphism. In particular, looking in the definition of $\widetilde \phi_\alpha$, we see that we can write uniquely $(u_l \sba^{-1}, u_r) = (x_\alpha(-t) \sba^{-1} u \widetilde u \, , \, x_\alpha(r) \widetilde u)$, with $t, r \in \CC$ and $u, \widetilde u \in U_{\setminus \alpha}$. Using that  $$ x_\alpha(-t) \sba^{-1}= \sba^{-1} x_{-\alpha}(t) $$
we deduce that $\phi_\alpha$ is an open embedding. 
Now, let $x$ as in the definition of $\phi_\alpha$, such that $t \neq 0$. Then, by the invariance properties of $f$ and  \eqref{x alpha s alpha = ..}, we have that $$f(\phi_\alpha(x))= f(h_l \sba^{-1} x_{-\alpha}(t)u, \, \, h_rx_\alpha(r))=f(h_l \alpha^\vee(t) x_\alpha(t\inv) u, \, \,  h_r x_\alpha(r)).$$
In particular, it's clear that if 

\begin{equation*}
        \begin{array}{r c l}
         \widetilde \iota_\alpha: T \times T \times \CC^* \times \CC \times U_{\setminus \alpha}  & \longto   & G \times G\\
         ( h_l, h_r, t , r, u) & \longmapsto & (h_l \alpha^\vee(t)x_\alpha(t\inv) u , \, \, h_r x_\alpha(r))
        \end{array}
    \end{equation*}
then $$\val_{\alpha_l}(f)=\val_{\{t=0\}} \bigl( \phi_\alpha^*(f) \bigr)= \val_{\{t=0\}} \bigl( \widetilde \iota^*_\alpha (f)\bigr).$$
But \begin{equation}
\label{eq 5}
        \begin{array}{r c l}
         \CC \times U_{\setminus \alpha} & \longto   &  Y \\
         (r ,  u) & \longmapsto & u x_\alpha(-r)
        \end{array}
    \end{equation}
is an isomorphism and, since $f$ is invariant by the right action of diagonal multiplication by $U$, we have that $$f(\widetilde \iota_\alpha(x))= f\biggl((h_l \alpha^\vee(t) x_\alpha(t\inv) u x_\alpha(-r), h_r) \biggr).$$
In particular, under the identification induced by \eqref{eq 5}, we have that $$\widetilde \iota_\alpha^*(f)= \iota_\alpha^*(f).$$

We conclude thanks to the following very simple statement.

\begin{lemma}
\label{lem:technical valuations}
    Let $\psi: Z \longto W$ be an isomorphism of normal, complex, algebraic varieties and $\phi : \CC^* \times Z \longto \CC^* \times W$ be defined by $(t,z) \longmapsto (t, \psi(z))$. Then, for any $f \in \CC(\CC^* \times W)$, $$\val_{\{t=0\}}(f)= \val_{\{t=0\}}\phi^*(f).$$
\end{lemma}

Next, we pass to the "right side", that is to $j_\alpha$. The argument is very similar, we sketch through it. First, one proves that the map 

\begin{equation*}
        \begin{array}{r c l}
         \psi_\alpha: U^- \times U^- \times  T \times T  \times \CC   \times Y \times U_{\setminus \alpha}  & \longto   & G \times G\\
         x=(u^-_l, u^-_r, h_l, h_r,t , 
         y , u ) & \longmapsto & \psi_\alpha(x)= (u_l^-h_l y u , \, \, u_r^-h_rx_\alpha(t) \sba u )
        \end{array}
    \end{equation*}
is an open embedding. Moreover, $\psi_\alpha(x) \in G \times \overline{B^- s_\alpha B}$ if $t=0$ and $\psi_\alpha(x) \in G_0 \times G_0$ otherwise. But then, for $t \neq 0$ we compute that 
$$f(\psi_\alpha(x))=f\biggl((h_l y, \, \, h_r \alpha^\vee(t) x_\alpha(-t\inv)\biggr).$$ 
In particular, considering 

 \begin{equation*}
        \begin{array}{r c l}
        \widetilde j_\alpha :  T \times T  \times \CC^* \times Y  & \longto   & G \times G\\
         (h_l, h_r,t,y) & \longmapsto &  (h_l y  , \, \, h_r \alpha^\vee(t)x_{\alpha}(-t \inv))
        \end{array}
    \end{equation*}
we clearly have that $$\val_{\alpha_r}(f)= \val_{\{t=0\}}\psi_\alpha^*(f)=\val_{\{t=0\}}\widetilde{j_\alpha}^* (f).$$
But then, by invariance of $f$, we see that $\widetilde{j_\alpha}^* (f)= j^*_\alpha(f).$
\end{proof}

Then, we make this important but very easy calculation. 

\begin{lemma}
\label{fundamental calculation tensor product}
    Let  $(\lambda, \mu) \in X(T) \times X(T)=X(T \times T)$, $f \in \CC(\lX)_{(\lambda,\mu), \bullet}$  and $x$ generic as in the notation of Lemma \ref{iota, j tensor product}. Then, for any $\alpha \in \Delta$,
    
   \begin{equation*}
       \begin{array}{ll}
        f(\iota_\alpha(x))=\lambda(h_l)\mu(h_r)t^{\langle \lambda , \alpha^\vee \rangle } f((x_\alpha(t\inv) y,  e)) \quad & \text{and}\\[0.3em]
f(j_\alpha(x))=\lambda(h_l)\mu(h_r)t^{\langle \mu , \alpha^\vee \rangle } f((y x_\alpha(t\inv), \, \, e) ).
       \end{array}
 \end{equation*}
    Moreover, for any $\beta \in \Delta$

    \begin{equation*}
        \begin{array}{ll}
           X_{\beta_l}(\iota_\alpha(x))=\pib(h_l)t^{\langle \pib , \alpha^\vee \rangle}  & X_{\beta_r}(\iota_\alpha(x))=\pib(h_r) \\
           X_{\beta_l}(j_\alpha(x))=\pib(h_l)  &  X_{\beta_r}(j_\alpha(x))=\pib(h_r)t^{\langle \pib , \alpha^\vee \rangle}
        \end{array}
    \end{equation*}
\end{lemma}

\begin{remark}
    By little abuse, $\CC(\lX)_{(\lambda,\mu), \bullet}$ denotes the set of $T \times T$ semi-invariant rational functions of weight $(\lambda,\mu).$
\end{remark}
\begin{proof}
   The proof is trivial.
\end{proof}
Recall that, an element $\mu \in \ZZ^{D}$, is identified with the character 
$$ \displaystyle \sum_{\alpha \in \Delta}(\mu_{\alpha_l} \pia, \mu_{\alpha_r} \pia ) \in X(T \times T).$$
For any $\alpha \in \Delta$ and $v,w \in W$, we denote by $D^{\pia}_{v,w}$ the restriction of $\delavw$ to $Y=U$.
Here, we describe explicitly the cluster variables $\lif x$ and how they behave in the charts given by $\iota_\alpha$ and $j_\alpha$.

\begin{lemma}
 \label{valuation and formula cluster variable tensor product}
For any $k \in [l]$ and $\alpha \in \Delta$ $$\nu_{\alpha_l,k}= \delta_{i_k, \alpha} \quad \text{and} \quad \nu_{\alpha_r,k}= \max \{ 0 , \langle - z_{\leq k}^{-1} \varpi_{i_k}, \alpha^\vee \rangle \}. $$
Moreover, 
$$\lif x_k\bigl(\iota_\alpha(x)\bigr)=\nu_{\bullet,k}\bigl((h_l, h_r)\bigr) \biggl(p_{k,0}(y) + t p_{k,1}(y)\biggr) \quad  \text{where} \quad p_{k,0}= D^{\varpi_{i_k}}_{s_\alpha, z_{\leq k}^{-1}}$$
and 
$$\lif x_k\bigl(j_\alpha(x)\bigr)=\nu_{\bullet,k}\bigl((h_l, h_r)\bigr) \biggl(\sum_{n=0}^{\nu_{\alpha_r,k}} t^n q_{k,n}(y) \biggr) \quad \text{where} \quad
    q_{k,0}= \begin{cases}
        D^{\varpi_{i_k}}_{e, z_{\leq k}^{-1}}
 & \text{if} \quad (z_{\leq k }) \alpha \in \Phi^+ \\
 D^{\varpi_{i_k}}_{e, s_\alpha z_{\leq k}^{-1}}
 & \text{if} \quad (z_{\leq k}) \alpha \in \Phi^-.\end{cases}$$

\end{lemma}

\begin{proof}
    Note that $1 \otimes x_k \in \CC(\lX)_{(0,0), \bullet}$. In particular, by Lemma \ref{fundamental calculation tensor product} we have that 
    $$1 \otimes x_k\bigl(\iota_\alpha(x)\bigr)=1 \otimes x_k\bigl(( x_\alpha(t\inv)y, e)\bigr).$$
    But (see diagram \eqref{eq: cartesian tensor}) 
    $$\phi \biggl( (e, e, x_\alpha(t\inv)y) \biggr) = \pi \biggl( ( x_\alpha(t\inv)y, e) \biggr).$$
 Then, by the definition of $1 \otimes x_k$, we have that 
 $$1 \otimes x_k\bigl(\iota_\alpha(x)\bigr)=x_k\bigl(x_\alpha(t\inv)y\bigr).$$
    Now we apply Lemma \ref{minors, developpement left translation}.
    
    If $i_k \neq \alpha$, then 

    \begin{equation}
        \label{eq 12}
        x_k(x_\alpha(t\inv)y)=D^{\varpi_{i_k}}_{e, z_{\leq k}^{-1}}(y)= D^{\varpi_{i_k}}_{s_\alpha, z_{\leq k}^{-1}}(y)=p_{k,0}(y).
    \end{equation}
The central equality follows from Lemma \ref{lem:basic prop minors}. In particular, by Lemma \ref{iota, j tensor product}, we deduce that $\nu_{\alpha_l,k}=0.$

If $i_k=\alpha$, then
    \begin{equation}
        \label{eq 6}
   x_k(x_\alpha(t\inv)y)= p_{k,1}(y) + t\inv p_{k,0}(y)
    \end{equation}
    where $p_{k,0}$ is as described in the statement and $p_{k,1}$ can be deduced from Lemma \ref{minors, developpement left translation}.  By Lemma \ref{vanishing minor s,w}, $p_{k,0}$ doesn't vanish on $Y$. Hence, form Lemma \ref{iota, j tensor product}, we deduce that $\nu_{\alpha_l,k}= \delta_{i_k,\alpha}$. 

\bigskip

    We switch to $j_\alpha$. The argument is very similar. We set $N_k(\alpha)= \max \{ 0 , \langle - z_{\leq k}^{-1} \varpi_{i_k}, \alpha^\vee \rangle \} $. 
    First, observe that if $(z_{\leq k}) \alpha \in \Phi^+$, then $N_k(\alpha) =  0$ and 
    \begin{equation}
        \label{eq:37}
        1 \otimes x_k(j_{\alpha}(x))= D^{\varpi_{i,k}}_{e, z_{\leq k^{-1}}}(y)
    \end{equation}
    by Lemmas \ref{fundamental calculation tensor product} and \ref{minors, developpement right translation}. If $(z_{\leq k}) \alpha \in \Phi^-$, then $N_k(\alpha)=  \langle - z_{\leq k}^{-1} \varpi_{i_k}, \alpha^\vee \rangle. $ Again, Lemmas \ref{fundamental calculation tensor product} and \ref{minors, developpement right translation} imply that
    
    \begin{equation}
    \label{eq 7}
        1 \otimes x_k(j_\alpha(x))= \sum_{n=0}^{N_k(\alpha)} t^{-n} Q_{(N_k(\alpha) -n )}(y).
    \end{equation}

Then, we set $q_{k,m}=Q_m$. By Lemma \ref{minors, developpement right translation}, $q_{k,0}$ agrees with the expression given in the statement. Since the minors of the form $\delbew$ don't vanish on $Y=U$, we have that $q_{k,0} \neq 0.$ 
Then, we deuce from expressions \eqref{eq 7}, \eqref{eq:37} and Lemma \ref{iota, j tensor product} that $\nu_{\alpha_r,k}=N_k(\alpha)$.
Now, since 
$$\lif x_k=X^{\nu_{\bullet,k}} (1 \otimes x_k),$$ the desired expression for $\iota_\alpha^*(\lif x_k)$ and $j_\alpha^*(\lif x_k)$ is obtained from \eqref{eq 12}, \eqref{eq 6}, \eqref{eq:37}, \eqref{eq 7} and Lemma \ref{fundamental calculation tensor product}.

\end{proof}

Before going on, we characterise the $X(T)^3$-weight of the cluster variables of the seeds $\lif t_{\ii'}^D$, for the various $\ii' \in R(w_0)$. Recall the following theorem.

\begin{theo}[\cite{parthasarathy1967representations}]
    \label{thm:PRV without multiplicity}
    For any $\lambda, \mu, \nu \in X(T)^+$, if there exist $v \in W$ such that
    \begin{equation}
    \label{eq:PRV condition without mult}
        w_0 \lambda + v \mu =v\nu
    \end{equation} 
    then, for any $k \in \NN$, 
    $$\dim \biggl( \Hom\bigl(V(k \lambda) \otimes V(k \mu) \, , \, V(k \nu)\bigr)^G\biggr)=1.$$
\end{theo}

We say that the triple of dominant weights $(\lambda,\mu, \nu)$ satisfies the PRV condition, for $v \in W$, if condition \eqref{eq:PRV condition without mult} holds.
If $\lambda = \sum n_\alpha \pia \in X(T)$, we denote $\lambda^+= \sum n_\alpha^+ \pia$ and $\lambda^-= \sum n_\alpha^- \pia.$

\begin{prop}
    \label{prop:weights c variable tensor}
    For any $k \in[l],$ $\lif x_k$ is $X(T)^3$-homogeneous of weight 
    $$\bigl(\varpi_{i_k}, ( z_{\leq k}^{-1} \varpi_{i_k})^-, ( z_{\leq k}^{-1} \varpi_{i_k})^+\bigr).$$
    This triple of dominant weights satisfies the PRV condition for $w_0 z_{\leq k}.$
\end{prop}

\begin{proof}
     By Lemma \ref{valuation and formula cluster variable tensor product}, $\nu_{\bullet,k}= (\varpi_{i_k}, ( z_{\leq k}^{-1} \varpi_{i_k})^-).$ Moreover, Lemma \ref{lem:degree lifting variables} implies that (recall that we modified our notation for this section) $\lif x_k$ is $X(T)^3$-homogeneous of weight 
$$(\varpi_{i_k}, ( z_{\leq k}^{-1} \varpi_{i_k})^-, \varpi_{i_k} + ( z_{\leq k}^{-1} \varpi_{i_k})^- +  z_{\leq k}^{-1} \varpi_{i_k}- \varpi_{i_k}).$$
Recalling that, for any $b \in \RR$, $b=b^+-b^-$, we immediately have that the triple of weights above coincides with the one of the statement. The same formula allows to verify, easily, the PRV condition.
\end{proof}

Understanding the multiplicities of the weights of cluster variables, of seeds equivalent to $\lif t_\ii$, feels to the author a difficult and intriguing question. We refer the reader to question \ref{ques: on weights}.

\begin{prop}
\label{prop:components cluster variables tensor prod}
    Let $w \in W$ and $\alpha \in \Delta$.
    \begin{enumerate}
        \item There exist $\ii' \in R(w_0)$ and a cluster variable $\lif x_{\alpha,w}$ of $\lif t_{\ii'}^D$,  which is $X(T)^3$-homogeneous of weight $(\varpi_\alpha, (w \pia)^-, (w \pia)^+).$
        \item  If $G$ is simply laced and $\widetilde x$ is a cluster variable of a seed $\widetilde t \simeq \lif t_{\ii''}^D$, for a certain $\ii'' \in R(w_0)$, such that $\widetilde x$ is $X(T)^3$-homogeneous of weight $(\varpi_\alpha, (w \pia)^-, (w \pia)^+)$, then $\widetilde x= \lif x_{\alpha,w}$.
    \end{enumerate}
\end{prop}

\begin{proof}
    For statement one, let $\jj=(j_1, \dots, j_m) \in R(w)$ (pay attention to the unusual enumeration). Let $k= \alpha^{\{\max\}}$ and $v= \prod_{n=1}^k s_{j_k}$. Since $w \pia=v \pia$, it's sufficient to prove the statement for $v$. If $k=0$, hence $v=e$, then $\lif x_{\alpha,e}= \lif x_{\alpha_l}$ works. Note that $\lif x_{\alpha_l}=X_{\alpha_l}$ by definition a cluster variable of $\lif t_{\ii'}^D$ for any $\ii'$. If $k>0$, take $\ii' \in R(w_0)$ such that $i'_n=j_n$ for $n \leq k$. By Proposition \ref{prop:weights c variable tensor}, setting $\lif x_{\alpha, v}= \lif x_{\ii',k}$ works.

\bigskip

    For statement two, by Corollary \ref{cor:independence reduced exp tensor} we have that $\widetilde t \simeq \lif t_{\ii'}^D$. By Theorem \ref{thm:PRV without multiplicity} and Proposition \ref{prop:weights c variable tensor}, $\dim \Br_{(\varpi_\alpha, (w \pia)^-, (w \pia)^+)}=1$. Hence, there exist $c \in \CC^*$ such that $\lif x_{\alpha,w}= c \widetilde x$. The statement follows from Theorem \ref{cluster variables irred}.
\end{proof}

Note that, for any $\alpha \in \Delta $, $\alpha^{\min}$ and $\alpha(-)= \min \{ k \, : \, (z_{\leq k}) \alpha \in \Phi^-\}$ are defined.

\begin{lemma}
\label{p q min polinomio nei precedenti tensor product}
    For any $\alpha \in \Delta$ we have $$p_{\alpha^{\min},0}= \prod_{j=1}^{\alpha^{\min} - 1}p_{j,0}^{c_j} \quad \text{and} \quad q_{\alpha(-),0}= \prod_{j=1}^{\alpha(-) - 1}q_{j,0}^{d_j} $$
    for some  non-negative integers $c_j$ and $d_j$.
\end{lemma}

    \begin{proof}
        For the functions $p$, this is a reformulation of Lemma \ref{delta alpha min polinomio nei precedenti U(w)}. For the functions $q$ it is a reformulation of Lemma \ref{f alpha(-) = }.
    \end{proof}

\begin{lemma}
    \label{p,q algebraic independent}
    The functions $p_{k,0}$ for $k \neq \alpha^{\min}$ (resp.  $q_{r,0}$ for $r \neq \alpha(-)$) are algebraically independent.
\end{lemma}

\begin{proof}
   This is a reformulation of Lemma \ref{delta s w alg indip U(w)} for the $p$ and of Lemma \ref{f alg indip} for the $q$.
\end{proof}

We are ready to prove Theorem \ref{equality tensor product}.

\begin{proof}[Proof of Theorem \ref{equality tensor product}]
    By contradiction, suppose that $\uclu(\lif t^D)$ is strictly contained in $\Br$. By Proposition \ref{conditions equality minimal lifting}, there exists a $d \in D$ and $f \in \Br$ such that $\cval_d(f) < 0$. We can suppose that $f$ is homogeneous for the $T \times T$-action. In particular, up to multiplying for a monomial in the $X_s$, $s \in D$, and in the unfrozen variables of $\lif x$, we can suppose that there exist $(\lambda, \mu) \in X(T) \times X(T)$ and $f \in \Br_{(\lambda, \mu), \bullet}$ such that 

    \begin{equation}
    \label{eq 14}
        f= \frac{P}{\lif x_d} \quad \text{where} \quad P= \sum_{n  \in \NN^{\lif \spc I}}a_n \lif x^n  \, \, : \, \, \lif x_d \not | \, P.
    \end{equation}
Hence, $P$ is a polynomial in the variables of the cluster $\lif x$, which is not divisible by $\lif x_d$. Here, divisibility is intended in the polynomial ring. Up to changing $f$, we can assume that, if $n_d > 0 $, then $a_n=0$, so that the sum defining $P$ runs over $\NN^{\lif \spc I \setminus \{ d \}}.$ 
We use the convention that, if $\alpha \in \Delta$, then $\nu_{\bullet, \alpha_l}= (\pia,0) $ and $\nu_{\bullet, \alpha_r}=(0, \pia)$. The fact that $f$ is homogeneous of degree $(\lambda, \mu)$ imposes that, for any $n$ such that $a_n \neq 0$,

\begin{equation}
    \label{eq 13}
    \sum_{j \in \lif \spc I } n_j \nu_{\bullet,j}= (\lambda,\mu) + \nu_{\bullet,d}.
\end{equation}
For $s \in D$, we set $p_{s,0}=1=q_{s,0}$. We distinguish two cases.

\bigskip

Suppose first that $d=\alpha_l$ for a certain $\alpha \in \Delta.$ Let $x$ as in Lemma \ref{iota, j tensor product}. Using Proposition \ref{valuation and formula cluster variable tensor product} and Lemma \ref{fundamental calculation tensor product} we deduce that
$$f(\iota_\alpha(x))=(\lambda,\mu)\bigl((h_l,h_r)\bigr) \biggl( t^\inv \sum_{k\geq 0} t^k f_k(y)\biggr)$$
for certain $f_k \in \CC[Y]$. 
Using Lemma \ref{fundamental calculation tensor product}, the fact that $n_{\alpha_l}=0$ if $a_n \neq 0$ and Proposition \ref{valuation and formula cluster variable tensor product}, we get that 
$$f_0= \sum_{n \in \NN^{\lif \spc I \setminus \{ \alpha_l \}}}a_n p_{\bullet,0}^n \quad \text{where} \quad p_{\bullet,0}^n= \prod_{j \in \lif \spc I \setminus \{\alpha_l \}}p_{j,0}^{n_j}.$$ 
But since $f \in \Br$, $\val_{\{t=0\}}\bigl(\iota_\alpha^*(f) \bigr)\geq 0$, in particular $f_0=0.$

\bigskip

Consider the linear map $\widetilde \pi : \ZZ^{\lif \spc I \setminus \{\alpha_l \}} \longto \ZZ^{I \setminus \alpha^{\min}}$ defined in the following way.
For $s \in D \setminus \{\alpha_l \}$, $\widetilde \pi(e_s)=0$. For $j \in I \setminus \{\alpha^{\min} \}$, $\widetilde \pi(e_j)=e_j$ and $\widetilde \pi(e_{\alpha^{\min}})= \sum_{j < \alpha^{\min}} c_j e_j$, where the coefficients $c_j$ are the ones of Lemma \ref{p q min polinomio nei precedenti tensor product}. Since $p_{s,0}=1$ for $s \in D$ and because of Lemma \ref{p q min polinomio nei precedenti tensor product} we have that, for any $n \in \NN^{\lif \spc I \setminus \{\alpha_l \}}$,
$$p_{\bullet,0}^n=p_{\bullet,0}^{\widetilde \pi(n) }.$$

In particular,
$$f_0 = \sum_{ m \in \NN^{I \setminus \{\alpha^{\min} \}}}b_m p_{\bullet,0}^m \quad \text{where} \quad b_m= \sum_{n \in \NN^{{\lif \spc I \setminus \{\alpha_l \}}} \, : \, \widetilde \pi(n)=m}a_n.$$
Since, by Lemma \ref{p,q algebraic independent}, the functions $p_{j,0}$ with $j \in I \setminus \{ \alpha^{\min} \}$ are algebraically independent and $f_0=0$, then for any $m \in  \NN^{I \setminus \{\alpha^{\min} \}}$, $b_m=0$. We claim that for any $m \in  \NN^{I \setminus \{\alpha^{\min} \}}$ there exists at most one  $n \in \NN^{{\lif \spc I \setminus \{\alpha_l \}}}$ such that $a_n \neq 0$ and $\widetilde \pi(n)=m$. This implies at ones that, for any $n \in \NN^{{\lif \spc I \setminus \{\alpha_l \}}} $, $a_n= 0$, which is a contradiction.

\bigskip

To prove the claim, just notice that any $n$ such that $a_n \neq 0$ satisfies the weight condition \eqref{eq 13}. In particular, consider the linear map $\pi : \ZZ^{{\lif \spc I \setminus \{\alpha_l \}}} \longto  X(T \times T) \times \ZZ^{I \setminus \{\alpha^{\min} \}}$ defined by $\pi(e_j)=(\nu_{\bullet, j}, \widetilde \pi(e_j))$. Remember that, for $s \in D$, $\nu_{\bullet,s}$ has been previously defined. If $n$ is such that $a_n \neq 0$ and $\widetilde \pi(n)=m $, then $\pi(n)= ((\lambda + \pia, \mu), m)$. But the map $\pi$ is injective. To prove this, we can consider $\pi$ as a matrix with columns indexed by $\lif I \setminus \{ \alpha_l \}$ and rows indexed by $\lif I \setminus \{ \alpha^{\min} \}.$ Recall that $X(T \times T)$ is identified with $\ZZ^D$ by means of the fundamental weights. Enumerate the roots of $\Delta $  from $1$ to $k$, such that $\alpha=\beta_k$. Then, order the columns of the matrix representing $\pi$ accordingly to the order 
$$\beta_{1,l} < \dots < \beta_{k-1,l} < \beta_{1,r} < \dots < \beta_{k,r} < 1 < \dots < l$$
and the rows according to the order
$$\beta_{1,l} < \dots < \beta_{k-1,l}  < \beta_{1,r} < \dots < \beta_{k,r} < 1 < \dots < \alpha^{\min}-1 < \alpha_l < \alpha^{\min} + 1 < \dots < l.$$
Then, the matrix representing $\pi$ is upper triangular with ones on the diagonal (note that $\nu_{\alpha_l,\alpha^{\min}}=1$ by Lemma \ref{valuation and formula cluster variable tensor product}). In particular, $\pi$ is invertible.

\bigskip

Next suppose that $d=\alpha_r$, for a certain $\alpha \in \Delta.$ We can compute $j_\alpha^*(f)$, similarly as before, and by the same argument we deduce that the function $f_0 \in \CC[Y]$ defined by

$$ f_0= \sum_{n \in \NN^{\lif \spc I \setminus \{ \alpha_r \}}}a_n q_{\bullet,0}^n \quad \text{where} \quad q_{\bullet,0}^n= \prod_{j \in \lif \spc I \setminus \{\alpha_r \}}q_{j,0}^{n_j}.$$
is zero.

\bigskip

We introduce the linear map $\widetilde \pi : \ZZ^{\lif \spc I \setminus \{\alpha_r \}} \longto \ZZ^{I \setminus \alpha(-)}$ defined in the following way. For $s \in D \setminus \{\alpha_r \}$, $\widetilde \pi(e_s)=0$. For $j \in I \setminus \{\alpha(-) \}$, $\widetilde \pi(e_j)=e_j$ and $\widetilde \pi(e_{\alpha(-)})= \sum_{j < \alpha^{\min}} d_j e_j$, where the coefficients $d_j$ are the ones of Lemma \ref{p q min polinomio nei precedenti tensor product}. Since $q_{s,0}=1$ for $s \in D$ and because of Lemma \ref{p q min polinomio nei precedenti tensor product}, we have that for any $n \in \NN^{\lif \spc I \setminus \{\alpha_r\}}$,
$$q_{\bullet,0}^n=q_{\bullet,0}^{\widetilde \pi(n) }.$$
In particular, $$f_0 = \sum_{ m \in \NN^{I \setminus \{\alpha(-) \}}}b_m q_{\bullet,0}^m \quad \text{where} \quad b_m= \sum_{n \in \NN^{{\lif \spc I \setminus \{\alpha_r \}}} \, : \, \widetilde \pi(n)=m}a_n.$$
Since, by Lemma \ref{p,q algebraic independent}, the functions $q_{j,0}$ with $j \in I \setminus \{ \alpha(-) \}$ are algebraically independent and $f_0=0$, then for any $m \in  \NN^{I \setminus \{\alpha(-) \}}$, $b_m=0$. We claim that, for any $m \in  \NN^{I \setminus \{\alpha(-) \}}$, there exists at most one  $n \in \NN^{{\lif \spc I \setminus \{\alpha_r \}}}$ such that $a_n \neq 0$ and $\widetilde \pi(n)=m$. This implies at ones that for any $n \in \NN^{{\lif \spc  I \setminus \{\alpha_r \}}} $, $a_n= 0$, which is a contradiction.

\bigskip

As before, consider the linear map $\pi : \ZZ^{{\lif \spc I \setminus \{\alpha_r \}}} \longto  X(T \times T) \times \ZZ^{I \setminus \{\alpha(-) \}}$ defined by $\pi(e_j)=(\nu_{\bullet,j}, \widetilde \pi(e_j))$. If $n$ is such that $a_n \neq 0$ and $\widetilde \pi(n)=m $, then $\pi(n)= ((\lambda, \mu + \pia), m)$. But the map $\pi$ is injective. To prove this,  consider $\pi$ as a matrix with columns indexed by $\lif I \setminus \{ \alpha_r \}$ and rows indexed by $\lif I \setminus \{ \alpha(-) \}.$ Enumerate the roots of $\Delta $  from $1$ to $k$ such that $\alpha=\beta_k$. Then, order the columns of the matrix representing $\pi$ accordingly to the order 
$$\beta_{1,l} < \dots < \beta_{k,l} < \beta_{1,r} < \dots < \beta_{k-1,r} < 1 < \dots < l$$
and the rows according to the order
$$\beta_{1,l} < \dots < \beta_{k,l}  < \beta_{1,r} < \dots < \beta_{k-1,r} < 1 < \dots < \alpha(-)-1 < \alpha_r < \alpha(-) + 1 < \dots < l.$$
Then, the matrix representing $\pi$ is upper triangular with ones on the diagonal. Here we used that $\nu_{\alpha_r,\alpha(-)}=1$ by Lemma \ref{valuation and formula cluster variable tensor product} and the third statement in Lemma \ref{right translation alpha(-)}. In particular, $\pi$ is invertible.
\end{proof}

\begin{example}
     Let $G= \SL_4$, then $\Delta= \{1,2,3\}$ with standard notation. Consider the reduced expression of $w_0$: $\ii= (1, 2, 3, 1, 2, 1)$ and let $t=t_\ii$. Then $$ B= \bmat 
    0 & -1 & 1 \\
    1 & 0 & -1\\
    -1 & 1 & 0\\
    0 & 1 & 0\\
    0 & -1 & 1\\
    0 & 0 & -1
    \emat$$
    where the rows and columns of $B$ are ordered in the obvious way. Graphically, the valued quiver of $t$ is the following:
    $$
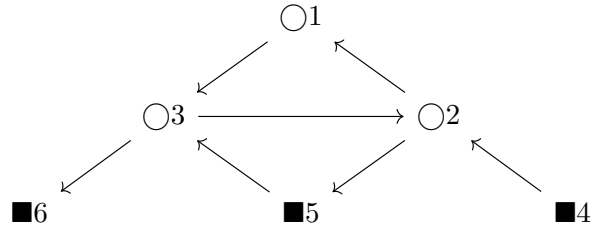
\begin{tikzcd}
	&& {\bigcirc 1} \\
	& {\bigcirc 3} && {\bigcirc 2} \\
	{\blacksquare 6} && {\blacksquare 5} && {\blacksquare 4}
	\arrow[from=1-3, to=2-2]
	\arrow[from=2-4, to=3-3]
	\arrow[from=2-2, to=3-1]
	\arrow[from=3-3, to=2-2]
	\arrow[from=2-4, to=1-3]
	\arrow[from=3-5, to=2-4]
	\arrow[from=2-2, to=2-4]
\end{tikzcd}$$
Here $D= \{1_l,2_l,3_l,1_r,2_r,3_r\}$. Reading the previous list from left to right allows to identify the minimal lifting matrix $\nu$ with a six by six matrix. Using Lemma \ref{valuation and formula cluster variable tensor product}, we have that

$$\nu= \bmat
1&0&1&0&0&1\\
0&1&0&0&1&0\\
0&0&0&1&0&0\\
1&1&0&1&0&0\\
0&0&1&0&1&0\\
0&0&0&0&0&1
\emat.
$$
A direct computation of $-\nu B$ allows to the deduce that, the seed $t_\ii^D$, corresponds to the following valued quiver

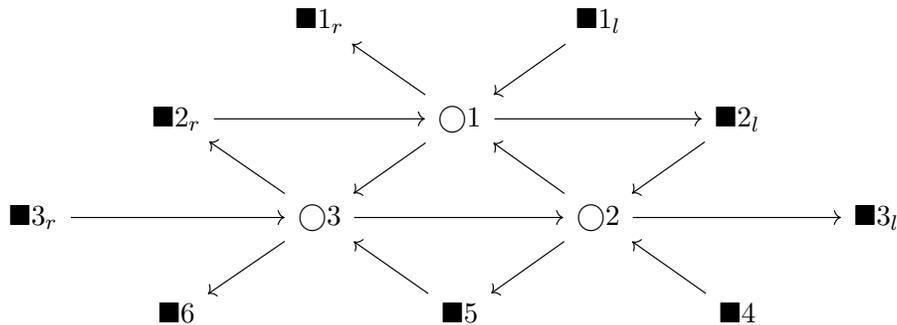
\[\begin{tikzcd}
	&& {\blacksquare 1_r} && {\blacksquare 1_l} \\
	& {\blacksquare 2_r} && {\bigcirc 1} && {\blacksquare 2_l} \\
	{\blacksquare 3_r} && {\bigcirc 3} && {\bigcirc 2} && {\blacksquare 3_l} \\
	& {\blacksquare 6} && {\blacksquare 5} && {\blacksquare 4}
	\arrow[from=2-4, to=3-3]
	\arrow[from=3-5, to=4-4]
	\arrow[from=3-3, to=4-2]
	\arrow[from=4-4, to=3-3]
	\arrow[from=3-5, to=2-4]
	\arrow[from=4-6, to=3-5]
	\arrow[from=3-3, to=3-5]
	\arrow[from=1-5, to=2-4]
	\arrow[from=2-6, to=3-5]
	\arrow[from=2-2, to=2-4]
	\arrow[from=2-4, to=2-6]
	\arrow[from=3-1, to=3-3]
	\arrow[from=3-5, to=3-7]
	\arrow[from=2-4, to=1-3]
	\arrow[from=3-3, to=2-2]
\end{tikzcd}\]
which is the \textit{ice hive quiver} defined in \cite{fei2017cluster}.
\end{example}

\begin{example}
\label{ex:G_2 tensor prod}
    Let $G= G_2$. Label the two simple roots as prescribed by the following picture. 
    
    
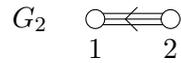
\begin{center}
  \begin{tikzpicture}[scale=0.5]
    \draw (-1,0) node[anchor=east] {$G_2$}; \draw (0,0) -- (2 cm,0);
    \draw (0, 0.15 cm) -- +(2 cm,0); \draw (0, -0.15 cm) -- +(2 cm,0);
    \draw[shift={(0.8, 0)}, rotate=180] (135 : 0.45cm) -- (0,0) --
    (-135 : 0.45cm); \draw[fill=white] (0 cm, 0 cm) circle (.25cm)
    node[below=4pt]{$1$}; \draw[fill=white] (2 cm, 0 cm) circle
    (.25cm) node[below=4pt]{$2$};
  \end{tikzpicture}
\end{center}
Let $\ii= (1, 2, 1, 2, 1, 2) \in R(w_0)$ and $t=t_\ii$. 
Then $$ B= \bmat 
    0 & -1 & 1 & 0 \\
    3 & 0 & -3 & 1\\
    -1 & 1 & 0 & -1\\
    0 & -1 & 3 & 0\\
    0 & 0 & -1 & 1\\
    0 & 0 & 0 & -1
    \emat$$
    where the rows and columns of $B$ are ordered in the obvious way. Graphically, the valued quiver of $t$ is the following:
    
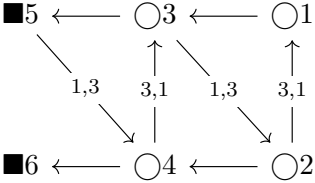
\[\begin{tikzcd}
	{\blacksquare 5} & \bigcirc3 & \bigcirc1 \\
	\\
	\blacksquare6 & \bigcirc4 & \bigcirc2
	\arrow[from=3-3, to=3-2]
	\arrow[from=3-2, to=3-1]
	\arrow[from=1-3, to=1-2]
	\arrow[from=1-2, to=1-1]
	\arrow["{ 3,1}"{description}, from=3-3, to=1-3]
	\arrow["{3,1}"{description}, from=3-2, to=1-2]
	\arrow["{ 1,3}"{description}, from=1-2, to=3-3]
	\arrow["{ 1,3}"{description}, from=1-1, to=3-2]
\end{tikzcd}\]
Here $D= \{1_l,2_l,1_r,2_r\}$. Reading the previous list from left to right allows to identify the minimal monomial lifting matrix $\nu$ with a four by six matrix. A direct computation using Lemma \ref{valuation and formula cluster variable tensor product}, implies that
$$\nu= \bmat
0&1&0&1&0&1\\
1&0&1&0&1&0\\
0&0&0&0&0&1\\
1&1&2&1&1&0\\

\emat.
$$
We can compute $-\nu B$ and deduce that, the seed $t_\ii^D$, corresponds to the following valued quiver

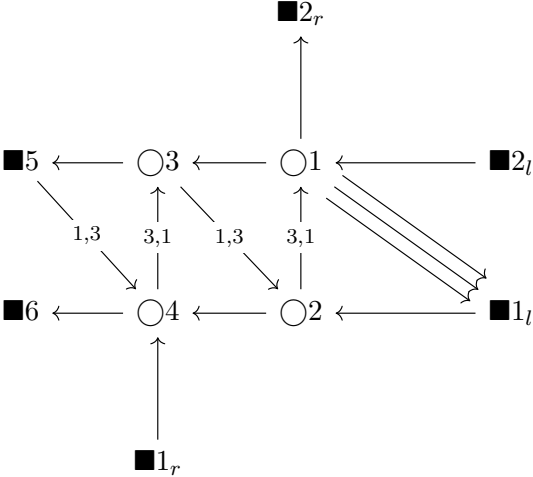
\[\begin{tikzcd}
	&& {\blacksquare2_r} \\
	\\
	{\blacksquare 5} & \bigcirc3 & \bigcirc1 && {\blacksquare2_l} \\
	\\
	\blacksquare6 & \bigcirc4 & \bigcirc2 && {\blacksquare 1_l} \\
	\\
	& {\blacksquare 1_r}
	\arrow[from=5-3, to=5-2]
	\arrow[from=5-2, to=5-1]
	\arrow[from=3-3, to=3-2]
	\arrow[from=3-2, to=3-1]
	\arrow["{ 3,1}"{description}, from=5-3, to=3-3]
	\arrow["{ 3,1}"{description}, from=5-2, to=3-2]
	\arrow["{ 1,3}"{description}, from=3-2, to=5-3]
	\arrow["{ 1,3}"{description}, from=3-1, to=5-2]
	\arrow[from=3-3, to=1-3]
	\arrow[from=3-5, to=3-3]
	\arrow[shift left=2, from=3-3, to=5-5]
	\arrow[shift right=2, from=3-3, to=5-5]
	\arrow[from=3-3, to=5-5]
	\arrow[from=5-5, to=5-3]
	\arrow[from=7-2, to=5-2]
\end{tikzcd}\]
which, in the notation of \cite{fei2016tensor}, up to some arrows between frozen vertices, is the iART-quiver $\Delta^2_Q$ corresponding to $G_2$.

\end{example}

We end this section comparing our construction to \cite{fei2016tensor}, which has been of great inspiration for this work. Indeed, in \cite{fei2016tensor}, the author constructs cluster structures on $\lX(G \times G, G)$ assuming $G$ to be simple and simply laced. The proof highly relies on the theory of quiver with potentials  \cite{derksen2008quivers} \cite{derksen2010quivers}. We believe our construction to be simpler, especially because it only relies on the geometric properties of the branching scheme and on the minimal monomial lifting, which is a versatile technique.  As an evidence of this fact, in the present text we drop the hypothesis that $G$ is simple and simply laced. Moreover, in the following we prove that Fei's cluster structures are obtained through minimal monomial lifting. It should also be true the the cluster structures constructed here agree with Fei's ones. We return on this aspect later. Finally, in \cite{fei2016tensor}, the proof that the branching scheme has cluster structure is simultaneously achieved with the existence of a good basis for the corresponding upper cluster algebra. Nevertheless, the two statements can't be proved independently. Even if the construction of a good basis for the upper cluster algebra is a remarkable achievement, we believe that the present approach, in which the construction of cluster structures on branching schemes are independent on the existence of good bases is more in the spirit of studying branching problems through cluster theory. In a forthcoming work, we will see how to construct good bases for some of the upper cluster algebras constructed in the present text, only relying on cluster theory.   

\bigskip

Suppose from now on that $G$ is simple. In \cite{fei2016tensor}, the author constructs a seed 
$$\widetilde t= (\widetilde I_{uf}, \widetilde I_{sf}=\emptyset, \widetilde I_{hf}, B^2, S^2 )$$
of $\CC(\lX)$, which is $X(T)^3$-graded by a degree configuration $\sigma^2$, such that \cite{fei2016tensor}[Theorem 8.1]

\begin{theo}
    \label{thm:Fei tensor} 
  If $G$ is simply laced
    \begin{enumerate}
    \item $\uclu(\widetilde t)= \Orb_\lX(\lX)$.
    \item The  graduation induced by $\sigma^2$, on $\uclu(\widetilde t)$, coincides with the natural graduation on $\Orb_\lX(\lX)$.
    \item The generic cluster character $C_{W^2}$, associated to a certain potential $W^2$ on the seed $\widetilde t$, gives a good (in the sense of Qin \cite{qin2022bases}) basis  for $\uclu( \widetilde t).$
\end{enumerate} 
\end{theo}

Suppose from now on that $G$ is simply laced and fix an orientation of the Dynkin diagram of $G$. This choice allows to construct a maximal rank seed $t$, of $\CC(U)$, such that $\uclu(t)= \CC[U]$ (see \cite{fei2016tensor}[Section 6.3]). Then, let $ \widetilde \iota : T \times T \times U \longto G \times G$ be the map defined by $(h_l,h_r, u) \longmapsto (h_l , h_ru)$ and $\iota : T \times T \times U \longto \lX$ be the map obtained by composing $\widetilde \iota$ with the natural projection $G \times G \longto \lX$. As for Corollary \ref{X phi suitable lifting}, one can easily prove that the triple $(\lX, \iota, X)$ is homogeneously suitable for $D$-lifting. We write $\lX_\iota$, for the scheme $\lX$, endowed with the suitable for $D$-lifting scheme structure described above. Moreover, $\lif t_\iota$ denotes the minimal monomial lifting of $t$, with respect to $\lX_\iota.$

\begin{prop}
    \label{prop:Fei=ours}
    We have that $\widetilde t= \lif t^D_\iota$.
\end{prop}

\begin{proof}
    We want to apply Theorem \ref{unicity minimal monomial lifting}. First, $S^2_D=X$ because of \cite{fei2016tensor}[Corollary 8.8, Corollary 7.9]. By \cite{fei2016tensor}[Proposition 3.6], we have that $\widetilde I_{uf}= I_{uf}$, $\widetilde I_{hf}= I_{hf} \sqcup D$ and $B^2_{I \times I_{uf}}=B$. Hence, $\widetilde t$ is a $D$-pointed seed extension of $t$. Moreover, condition 2 and 3 of Theorem \ref{unicity minimal monomial lifting} hold because of \cite{fei2016tensor}[Lemma 7.8, Corollary 7.11] and the discussion at the beginning of \cite{fei2016tensor}[Section 6.3]. Since $\uclu(\widetilde t)= \Orb_{\lX_\iota}(\lX_\iota)$ by Theorem \ref{thm:Fei tensor}, we have that $\widetilde t= \lif t^D_\iota$ by Theorem \ref{unicity minimal monomial lifting}.
\end{proof}

It should be true that the seed $t$ equals $t_\ii$, for a certain $\ii \in R(w_0)$ adapted, in the sense of \cite{geiss2007auslander}[Section 1.3], to the chosen orientation of the Dynkin diagram, but this is not clear to the author, in general. Nevertheless, it is easy to verify this statement on small examples.

\bigskip

Finally, if $G$ is not simply laced, the inclusion $\uclu(\widetilde t) \subseteq \Orb_\lX(\lX)$ still holds and the first part of \cite{fei2016tensor}[Conjecture 8.2] states that this inclusion is actually an equality.
Proceeding as in the proof of Proposition \ref{prop:Fei=ours}, if one could identify $\ii \in R(w_0)$ such that $t=t_\ii$ and prove that  $B^2= \lif B_\ii$, then Theorem \ref{equality tensor product} would imply the first part of \cite{fei2016tensor}[Conjecture 8.2]. For example, in the $G_2$ case, since everything is explicit by Example \ref{ex:G_2 tensor prod} and \cite{fei2016tensor}[Appendix A], we have that the conjecture holds for $G_2$. 

\begin{remark}
    As a final remark, note that the cluster algebra $\uclu(t)$ is categorified and the work of Fei produces a categorification for $\lif t_\iota$ which is compatible with the on on $\uclu(t)$. Moreover, \cite{fei2016tensor}[Corollary 7.8] gives an homological interpretation of the minimal  lifting matrix $\nu_\iota$ of the seed $t$, with respect to $\lX_\iota$, in terms of this categorification. A more explicit understanding of the relation between the minimal lifting matrix in term of this categorification seems to the author the fundamental tool needed to prove \cite{fei2016tensor}[Conjecture 8.2] via Theorem \ref{equality tensor product}. Actually, comparing the minimal monomial lifting to the extension construction of \cite{fei2023crystal}[Section 2.3], could provide a possible lead to accomplish this task.
\end{remark}

\subsection{The base affine space, or branching to a maximal Torus}
\label{sec:base aff space}

We use the previous result to construct a cluster structure on $\lX:= \Spec(\CC[G]^{U^-})$. Note that $\lX = \lX(G, T)$ and $T$ is a Levi subgroup of $G$, so the results of this section, for $G$ simple, can be deduced form the ones of Section \ref{equality Levi}. Moreover, the results of Section \ref{equality Levi} hold in the case of $G_I= T$, even if $G$ is not necessarily simple, because the same proofs work. See Remark \ref{rem: simple Levi}. The aim of this section is to reinterpret this cluster structure on $\lX$, as a quotient of the one on $\lX(G \times G,G).$ We slightly modify the notation of Subsection \ref{Levi subgroup setting}.

\bigskip

Let $D_l=\{ \alpha_l \, : \, \alpha \in \Delta\}$ and $\widetilde X_{\alpha_l}= \delaee$. We have a cartesian square 
$$ \begin{tikzcd}
G_0 \arrow[r] \arrow[d] \arrow[dr, phantom, "\lrcorner", very near start]
& G \arrow[d] \\
T \times U \arrow["\widetilde \phi", r]
& \lX
\end{tikzcd}
$$
where $\widetilde \phi : T \times U \longto \lX$ is  the $T$-equivariant open embedding induced by the product $T \times U \longto G$ and the leftmost vertical map is defined by $x \longmapsto ([x]_0, [x]_+)$. Then $(\lX, \widetilde \phi, \widetilde X)$ is homogeneously suitable for $D_l$-lifting by Corollary \ref{X phi suitable lifting}.

Note also that we have a commutative diagram 
\begin{equation}
\label{eq:base affine commutative}
     \begin{tikzcd}
\lX \times T \arrow["\iota" , r]
& \lX(G \times G, G)  \\
T \times U \times T \arrow[r] \arrow["(\widetilde \phi \text{,} id)", u]
& T \times T \times U \arrow["\phi", u]
\end{tikzcd}
\end{equation}

where the map $T \times U \times T \longto T \times T \times U$ is defined by $(h_l,u,h_r) \longmapsto (h_l,h_r,u)$ and $\iota$ is the open embedding induced by the inclusion $G \times T  \subseteq G \times G$.

Let $\ii \in R(w_0)$. We denote by $t_\ii$ (resp $\widetilde t_\ii$) the associated seed of $\CC(U)$ when $\CC(U)$ is considered as a subfield of $\CC(\lX(G \times G,G))$ (resp. $\CC(\lX))$. We call $I_\ii$ the vertex set of both $\widetilde t_\ii$ and $t_\ii$. Let $\widetilde \nu_\ii \in \ZZ^{D_l \times I} $ and $\nu_\ii \in \ZZ^{(D_l \sqcup D_r) \times I}$ the minimal lifting matrix corresponding to $\widetilde t_\ii$ and $t_\ii$ respectively.

\begin{lemma}
    \label{lem:valuation affine space}
     We have that $(\nu_\ii)_{D_l, \bullet}= \widetilde \nu_\ii$.
\end{lemma}

\begin{proof}
    Note that, for any $\alpha \in \Delta$, $\iota^*(X_{\alpha_l})= \widetilde X_{\alpha_l} \otimes 1$. The statement follows at ones from the fact that $\iota$ is an open embedding and the diagram \eqref{eq:base affine commutative} is commutative.
\end{proof}

\begin{theo}
    We have equality $\uclu(\lif \widetilde t_\ii\strut^{D_l})=\Orb_\lX(\lX).$
\end{theo}

\begin{proof}
    It's easy to verify that $\Br(G \times G, G) / (X_{\alpha_r} -1 \, : \, \alpha \in \Delta) \simeq \CC[G]^{U^-}$. Geometrically, the isomorphism  corresponds to the closed embedding $\lX \longto \lX(G \times G, G)$ defined by $p \longmapsto \iota(p,e)$. By Theorem \ref{equality tensor product}, $\CC[G]^{U^-}$ is the quotient of $\uclu(\lif t_\ii^{D_l \sqcup D_r})$ by the ideal $(\lif x_{\ii, \alpha_r} - 1 \, : \, \alpha_r \in D_r)$. We conclude using Corollary \ref{cor:quotients of lifting} and Lemma \ref{lem:valuation affine space}.
\end{proof}

\begin{coro}
\label{cor:independence red exprs aff space}
    If $G$ is simply laced, the mutation equivalence class of $\lif \widetilde t_\ii$ doesn't depend on $\ii.$
\end{coro}

\begin{proof}
    It's a direct consequence of Corollary \ref{cor:independence reduced exp tensor}.
\end{proof}

For $\alpha \in \Delta$ and $w \in W$, we denote by $D^{\pia}_{e,w}$ the restriction of $\delaew$ to $U.$

\begin{lemma}
\label{lem:variables affine space minors}
    For $k \in [l]$, $\lif \widetilde x_{\ii,k}= \Delta^{\varpi_{i_k}}_{e, (w_0)_{\leq k}}.$
\end{lemma}
Here $(w_0)_{\leq k}$ is computed with respect to $\ii.$
\begin{proof}
    By Lemmas \ref{lem:valuation affine space} and \ref{valuation and formula cluster variable tensor product} we have that $\lif \widetilde x_{\ii,k}= \widetilde X_{(i_{k})_l}(1 \otimes \widetilde x_{\ii,k})$. Hence 
    $$\lif \widetilde x_{\ii,k}(hu)=\varpi_{i_k}(h)D^{\varpi_{i_k}}_{e, (w_0)_{\leq k}}(u)=\Delta^{\varpi_{i_k}}_{e, (w_0)_{\leq k}}(hu).$$
    Since $T \times U$ is open in $\lX$, which is irreducible, this proves the lemma.
\end{proof}

\begin{lemma}
\label{lem:all minors c variable affine space}
        Let $w \in W$ and $\alpha \in \Delta$.
    \begin{enumerate}
        \item There exist $\ii' \in R(w_0)$ such that $\delaew$ is a cluster variable of $\lif \widetilde{t}_{\ii'}^{D_l}$.
        \item If $G$ is simply laced, then $\delaew$ is a cluster variable of $\uclu(\lif \widetilde{t_\ii}^{D_l})$, for any $\ii$. 
        \end{enumerate}
\end{lemma}

\begin{proof}
    For the first statement, using Lemma \ref{lem:variables affine space minors}, the proof is literally the same of Proposition \ref{prop:components cluster variables tensor prod}. The second statement is an obvious consequence of the first one and of Corollary \ref{cor:independence red exprs aff space}.
\end{proof}

 \bibliographystyle{alpha}

\bibliography{biblio.bib}

\end{document}